\documentclass{daj}

\usepackage{amsmath}
\usepackage{amsthm}
\usepackage{amssymb}
\usepackage{tikz}
\usepackage{tikz-cd}
\usetikzlibrary{shapes,shapes.geometric,arrows,fit,calc,positioning,arrows,automata,}
\usepackage{enumitem}
\usepackage{mathrsfs}
\usepackage{scalerel}

\newtheorem{theorem}{Theorem}[section]

\newtheorem{alphatheorem}{Theorem}

\theoremstyle{definition}

\newtheorem{definition}[theorem]{Definition}
\newtheorem{proposition}[theorem]{Proposition}
\newtheorem{step}{Step}

\newtheorem*{proposition*}{Proposition}
\newtheorem*{observation*}{Observation}
\newtheorem*{claim*}{Claim}

\newtheorem*{lemma*}{Lemma}
\newtheorem{example}[theorem]{Example}
\newtheorem{corollary}[theorem]{Corollary}
\newtheorem{remark}[theorem]{Remark}

\newtheorem*{conjecture*}{Conjecture}

\newtheorem*{convention*}{Convention}

\theoremstyle{plain}
\newtheorem{lemma}[theorem]{Lemma}

\newcommand{\bra}[1]{\left(#1\right)}
\newcommand{\abra}[1]{ \left< #1 \right> }

\renewcommand{\tilde}{\widetilde}

\newcommand{\abs}[1]{\left|#1\right|}

\newcommand{\norm}[1]{\left\lVert #1 \right\rVert}

\newcommand{\cZ}{\mathcal{Z}}

\newcommand{\id}{\mathrm{id}}

\newcommand{\Sym}{\operatorname{Sym}}

\newcommand{\cT}{\mathscr{T}}

\newcommand{\cV}{\mathcal{V}}
\newcommand{\cU}{\mathcal{U}}

\newcommand{\cQ}{\mathcal{Q}}
\newcommand{\bfw}{\mathbf{w}}

\newcommand{\rb}[1]{\left( #1 \right)}

\newcommand{\e}{\varepsilon}

\usepackage{mathrsfs}

\newcommand{\NN}{\mathbb{N}}
\newcommand{\QQ}{\mathbb{Q}}

\newcommand{\bblambda}{{\boldsymbol\lambda}}
\newcommand{\bbalpha}{{\boldsymbol\alpha}}

\newcommand{\bbrho}{{\boldsymbol\rho}}

\newcommand{\bbpi}{{\boldsymbol\pi}}

\DeclareRobustCommand{\coprod}{\mathop{\text{\fakecoprod}}}
\newcommand{\fakecoprod}{%
  \sbox0{$\prod$}%
  \smash{\raisebox{\dimexpr.9625\depth-\dp0}{\scalebox{1}[-1]{$\prod$}}}%
  \vphantom{$\prod$}%
}

\newcommand{\Ob}{\mathrm{Ob}}

\DeclareMathOperator*{\EE}{\scalerel*{\mathbb{E}}{\textstyle\sum}}

\newcommand{\ZZ}{\mathbb{Z}}
\newcommand{\RR}{\mathbb{R}}

\newcommand{\N}{\NN}
\newcommand{\Z}{\ZZ}
\newcommand{\CC}{\mathbb{C}}

\newcommand{\cA}{\mathcal{A}}

\newcommand{\cF}{\mathcal{F}}

\newcommand{\Mor}{\mathrm{Mor}}
\newcommand{\U}{\mathrm{U}}

\newcommand{\floor}[1]{\left\lfloor #1 \right\rfloor}

\newcommand{\word}[1]{\mathbf{#1}}

\newcommand{\set}[2]{\left\{ #1 \ \middle| \ #2 \right\} }

\newcommand{\parbreak}[1]{
\begin{center}
***
\end{center}
}

\newcommand{\HK}{\operatorname{HK}}

\newcommand{\length}{\operatorname{length}}

\newcommand{\bb}{\mathbf}

\usepackage{soul}
\usepackage{bm}
\usepackage{stmaryrd}
\newcommand{\ifbra}[1]{\left\llbracket #1 \right\rrbracket}

\newcommand{\ceil}[1]{\left\lceil #1 \right\rceil}

\newcommand{\gea}{group extension of an automaton}
\newcommand{\geas}{group extensions of automata}
\newcommand{\geka}{group extension of a $k$-automaton}
\newcommand{\geaab}{GEA}
\newcommand{\gekaab}{$k$-GEA}
\newcommand{\egeaab}{efficient \geaab{}}
\newcommand{\egea}{efficient \gea{}}
\newcommand{\egeka}{efficient \geka{}}

\newcommand{\geaoab}{GEAO}
\newcommand{\egeaoab}{efficient \geaoab{}}
\newcommand{\gekaoab}{$k$-GEAO}

\newcommand{\gekao}{group extension of a $k$-automaton with output}
\newcommand{\geao}{group extension of an automaton with output}

\newcommand{\End}{\operatorname{End}}

\dajAUTHORdetails{%
  title = {Gowers Norms for Automatic Sequences}, 
  author = {Jakub Byszewski, Jakub Konieczny, and Clemens M\"{u}llner},
  plaintextauthor = {Jakub Byszewski, Jakub Konieczny, and Clemens Muellner},
  keywords = {Gowers norms, automatic sequences, higher degree uniformity},
}   

\dajEDITORdetails{%
   year={2023},
   number={4},
   received={17 May 2021},   
   published={24 May 2023},  
   doi={10.19086/da.75201},       
}   

\begin{document}

\begin{frontmatter}[classification=text]


\author[jb]{Jakub Byszewski\thanks{Supported by National Science Centre, Poland grant number 2018/29/B/ST1/01340.}}
\author[jk]{Jakub Konieczny\thanks{Supported by  the ERC grant ErgComNum 682150 at the Hebrew University of Jerusalem and is currently working within the framework of the LABEX MILYON (ANR-10-LABX-0070) of Universit\'{e} de Lyon, within the program "Investissements d'Avenir" (ANR-11-IDEX-0007) operated by the French National Research Agency (ANR). He also acknowledges support from the Foundation for Polish Science (FNP). }}
\author[cm]{Clemens M\"{u}llner\thanks{Supported by  European Research Council (ERC) under the European Union’s
Horizon 2020 research and innovation programme under the Grant Agreement No 648132,
by the project F55-02 of the Austrian Science Fund FWF which is part of the Special Research Program “Quasi-Monte Carlo Methods: Theory and Applications” and by Project I1751
(FWF), called MUDERA (Multiplicativity, Determinism, and Randomness).}}

\begin{abstract}
We show that any automatic sequence can be separated into a structured part and a Gowers uniform part in a way that is considerably more efficient than guaranteed by the Arithmetic Regularity Lemma. For sequences produced by strongly connected and prolongable automata, the structured part is rationally almost periodic, while for general sequences the description is marginally more complicated. In particular, we show that all automatic sequences orthogonal to periodic sequences are Gowers uniform.
As an application, we obtain for any $l \geq 2$ and any automatic set $A \subset \NN_0$ lower bounds on the number of $l$-term arithmetic progressions -- contained in $A$ -- with a given difference. The analogous result is false for general subsets of $\NN_0$ and progressions of length $\geq 5$.\end{abstract}
\end{frontmatter}


\section{Introduction}\label{sec_1}

Automatic sequences, that is, sequences computable by finite automata, constitute one of the basic classes of sequences defined in terms of complexity. Being both simple enough to be rigorously analysed and complex enough to be interesting, they are the subject of extensive investigation in various branches of mathematics and computer science. (For precise definitions and extended background, see Section \ref{sec:Auto}.)

The study of various notions of uniformity for automatic sequences can be traced back at least as far as 1968, when Gelfond \cite{Gelfond-1967} showed that the integers whose sum of base-$k$ digits lie in a given residue class modulo $l$ are well distributed in arithmetic progressions (subject to certain congruence conditions). In the same paper, Gelfond posed several influential questions on distribution of the sum of base-$k$ digits within residue classes along subsequences which sparked much subsequent research \cite{Kim-1999,  MauduitRivat-2009, MauduitRivat-2010, MauduitRivat-2015, DrmotaMauduitRivat-TM-squares, Muellner2018, Mauduit2018, Drmota2011, MullnerSpiegelhofer, Spiegelhofer2018}. An accessible introduction can be found in \cite{Morgenbesser-thesis}.

A systematic study of various notions of pseudorandomness was undertaken by Mauduit and Sark\"{o}zy in \cite{MauduitSarkozy-1998} for the Thue--Morse and Rudin--Shapiro sequences. Specifically, they showed that these sequences do not correlate with periodic sequences, but do have large self-correlations. 
In this paper we consider a notion of pseudorandomness originating from higher order Fourier analysis, corresponding to Gowers uniformity norms (for more on Gowers norms, see Section \ref{sec:Gowers}). The second-named author showed \cite{JK-Gowers-Thue-Morse} that the Thue--Morse and Rudin--Shapiro sequences are highly Gowers uniform of all orders. Here, we obtain a similar result in a much more general context.

The celebrated Inverse Theorem for Gowers uniformity norms \cite{GreenTaoZiegler-2012} provides a helpful criterion for Gowers uniformity. It asserts, roughly speaking, that any sequence which does not correlate with nilsequences of bounded complexity has small Gowers norms. 
We do not follow this path here directly, but want to point out some striking similarities to related results. 
For the purposes of this paper, there is no need to define what we mean by a nilsequence or its complexity, although we do wish to point out that nilsequences include polynomial phases, given by $n \mapsto e\bra{p(n)}$ where $e(t) = e^{2\pi i t}$ and $p \in \RR[x]$.

For a number of natural classes of sequences, in order to verify Gowers uniformity of all orders it is actually sufficient to verify lack of correlation with linear phases $n \mapsto e(n\alpha)$ where $\alpha \in \RR$, or even just with periodic sequences. In particular, Frantzikinakis and Host \cite{FrantzikinakisHost-2017} showed that a multiplicative sequence which does not correlate with periodic sequences is Gowers uniform of all orders.
Eisner and the second-named author showed \cite{JK-AutSeq-ergo} that an automatic sequence which does not correlate with periodic sequences also does not correlate with any polynomial phases. This motivates the following result. For the sake of brevity, we will say that a bounded sequence $a \colon \NN_0 \to \CC$ is \emph{highly Gowers uniform} if 
\begin{equation}\label{eq:def-of-h-G-uni}
\text{ for each } d \geq 1 \text{ there exists } c = c_d > 0 \text{ such that } \norm{a}_{U^d[N]} \ll_d N^{-c}.
\end{equation}
(See Sec.\ \ref{sec:Gowers}.\ref{ssec:asymptnot} for the asymptotic notation and  Sec.\ \ref{sec:Gowers}.\ref{ssec:Gowers:basic} for the definition of $\norm{a}_{U^d[N]}$.)

\begin{alphatheorem}\label{thm:correlations}
Let $a \colon \NN_0 \to \CC$ be an automatic sequence and suppose that $a$ does not correlate with periodic sequences in the sense that
\[
	\lim_{N \to \infty} \frac{1}{N} \sum_{n=0}^{N-1} a(n) b(n) = 0
\]
for any periodic sequence $b \colon \NN_0 \to \CC$. Then $a$ is highly Gowers uniform.
\end{alphatheorem}

In fact, we obtain a stronger decomposition theorem. The Inverse Theorem is essentially equivalent to the Arithmetic Regularity Lemma \cite{GreenTao-2010-ARL}, which asserts, again roughly speaking, that any $1$-bounded sequence $f \colon [N] \to [-1,1]$ can be decomposed into a sum 
\begin{equation}
\label{eq:arithmetic-regularity}
	f = f_{\mathrm{nil}} + f_{\mathrm{sml}} + f_{\mathrm{uni}},
\end{equation}
where the structured component $f_{\mathrm{nil}}$ is a (bounded complexity) nilsequence, $f_{\mathrm{sml}}$ has small $L^2$ norm and $f_{\mathrm{uni}}$ has small Gowers norm of a given order. In light of the discussion above, one might expect that in the case when $f$ is an automatic sequence, it should be possible to ensure that $f_{\mathrm{nil}}$ is essentially a periodic sequence.

This expectation is confirmed by the following new result, which is a special case of our main theorem. For standard terminology used, see Section \ref{sec:Gowers} (for Gowers norms) and \ref{sec:Auto} (for automatic sequences). 
 Rationally almost periodic sequences were first introduced in \cite{BergelsonRuzsa-2002}, and their properties are studied in more detail in \cite{BergelsonKulagaPrzymusLemanczykRichter-2016}. A sequence is rationally almost periodic (RAP) if it can be approximated by periodic sequences arbitrarily well in the Besicovitch metric; i.e., $x \colon \NN_0 \to \Omega$ is RAP if for any $\e > 0$ there is a periodic sequence $y \colon \NN_0 \to \Omega$ with $\abs{\set{ n < N}{ x(n) \neq y(n) }}/ N \leq \e$ for large enough $N$.

\begin{alphatheorem}\label{thm:main_simple}
	Let $a\colon \NN_0 \to \CC$ be an automatic sequence produced by a strongly connected, prolongable automaton. 
	Then there exists a decomposition 
	\begin{equation}
	\label{eq:001}
	 a(n) = a_{\mathrm{str}}(n) + a_{\mathrm{uni}}(n),
	\end{equation}
	where $a_{\mathrm{str}}$ is rationally almost periodic and $a_{\mathrm{uni}}$ is highly Gowers uniform (cf.\ \eqref{eq:def-of-h-G-uni}).
\end{alphatheorem}

Note that any RAP sequence can be decomposed as the sum of a periodic sequence and a sequence with a small $L^1$ norm. Hence, \eqref{eq:001} can be brought into the form analogous to \eqref{eq:arithmetic-regularity}, with a periodic sequence in place of a general nilsequence.
Furthermore, this decomposition works simultaneously for all orders.

For general automatic sequences we need a more general notion of a structured sequence. There are three basic classes of $k$-automatic sequences which fail to be Gowers uniform, which we describe informally as follows:
\begin{enumerate}[wide]
\item periodic sequences, whose periods may be assumed to be coprime to $k$;
\item sequences which are only sensitive to terminal digits, such as $\nu_k(n) \bmod{2}$ where $\nu_k(n)$ is the largest power of $k$ which divides $n$;
\item sequences which are only sensitive to initial digits, such as $\nu_k(n^{\mathrm{rev}}_k +1) \bmod{2}$ where $n^{\mathrm{rev}}_k$ denotes the result of reversing the base $k$ digits of $n$. 
\end{enumerate}
By changing the basis, we can include in the last category also sequences which depend on the length of the expansion of $n$. For instance, if $\length_k(n)$ denotes the length of the expansion of $n$ in base $k$ then $\length_k(n) \bmod{2}$ depends only on the leading digit of $n$ in base $k^2$.

Our main result asserts that any automatic sequence can be decomposed as the sum of a structured part and a highly Gowers uniform part, where the structured part is a combination of the examples outlined above. More precisely, let us say that a $k$-automatic sequence $a \colon \NN_0 \to \Omega$ is \emph{weakly structured} if there exist a periodic sequence $a_{\mathrm{per}} \colon \NN_0 \to \Omega_{\mathrm{per}}$ with period coprime to $k$, a forward synchronising $k$-automatic sequence $a_{\mathrm{fs}} \colon \NN_0 \to \Omega_{\mathrm{fs}}$ and a backward synchronising $k$-automatic sequence $a_{\mathrm{bs}} \colon \NN_0 \to \Omega_{\mathrm{bs}}$, as well as a map $F \colon \Omega_{\mathrm{per}} \times \Omega_{\mathrm{fs}} \times \Omega_{\mathrm{bs}} \to \Omega$ such that 
\begin{equation}\label{eq:def-of-wk-str}
	a(n) = F\bra{a_{\mathrm{per}}(n), a_{\mathrm{fs}}(n), a_{\mathrm{bs}}(n)}.
\end{equation}
(For definitions of synchronising sequences, we again refer to Sec.\ \ref{sec:Auto}.)



\begin{alphatheorem}\label{thm:main_nobasis}
	Let $a\colon \NN_0 \to \CC$ be an automatic sequence. Then there exists a decomposition 
	\begin{equation}
	\label{eq:003}
	 a(n) = a_{\mathrm{str}}(n) + a_{\mathrm{uni}}(n),
	\end{equation}
	where $a_{\mathrm{str}}$ is weakly structured (cf.\ \eqref{eq:def-of-wk-str}) and $a_{\mathrm{uni}}$ is highly Gowers uniform (cf.\ \eqref{eq:def-of-h-G-uni}).
\end{alphatheorem}

\begin{remark} 
The notion of a weakly structured sequence is very sensitive to the choice of the basis. If $k,k' \geq 1$ are both powers of the same integer $k_0$ then $k$-automatic sequences are the same as $k'$-automatic sequences, but $k$-automatic weakly structured sequences are \emph{not} the same as a $k'$-automatic weakly structured sequences. If the sequence $a$ in Theorem \ref{thm:main_nobasis} is $k$-automatic then $a_{\mathrm{str}}$ is only guaranteed to be weakly structured in some basis $k'$ that is a power of $k$, but it does not need to be weakly structured in the basis $k$.
\end{remark}

\begin{example}\label{ex:main-1}
	Let $a \colon \NN_0 \to \RR$ be the $2$-automatic sequence computed by the following automaton. 
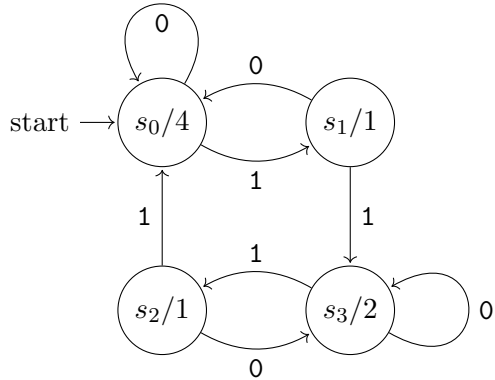
\begin{center}
\begin{tikzpicture}[shorten >=1pt,node distance=2.5cm, on grid, auto] 
   \node[initial, state] (s_00)   {$s_{0}/4$}; 
   \node[state] (s_01) [below=of s_00] {$s_{2}/1$}; 
   \node[state] (s_11) [right=of s_01] {$s_{3}/2$}; 
   \node[state] (s_10) [above=of s_11] {$s_{1}/1$}; 

 \tikzstyle{loop}=[min distance=6mm,in=120,out=60,looseness=7]
    
	\path[->]
     (s_00) edge [loop below] node  {\texttt 0} (s_00);

  \tikzstyle{loop}=[min distance=6mm,in=210,out=150,looseness=7]
    \path[->]     
    (s_01) edge [bend right] node [below]  {\texttt 0} (s_11);

    \path[->]     
    (s_00) edge [bend right] node [below]  {\texttt 1} (s_10);

 \tikzstyle{loop}=[min distance=4mm,in=120,out=240,looseness=1]
    \path[->]     
    (s_01) edge [left] node {\texttt 1} (s_00);
    \path[->]     
    (s_10) edge [right] node {\texttt 1} (s_11);

 \tikzstyle{loop}=[min distance=4mm,in=-60,out=60,looseness=1]
          
 \tikzstyle{loop}=[min distance=6mm,in=30,out=-30,looseness=7]
	\path[->]
    (s_11) edge [bend right] node [above]  {\texttt 1} (s_01);
   	\path[->]
    (s_10) edge [bend right] node [above]  {\texttt 0} (s_00);
	\path[->]
     (s_11) edge [loop right] node  {\texttt 0} (s_11);
\end{tikzpicture}
\end{center}
Formal definitions of automata and the associated sequence can be found it Section \ref{sec:Auto}. For now, it suffices to say that in order to compute $a(n)$, $n \in \NN_0$, one needs to expand $n$ in base $2$ and traverse the automaton using the edges corresponding to the consecutive digits of $n$ and then read off the output at the final state. For instance, the binary expansion of $n = 26$ is $(26)_2 = \mathtt{11010}$, so the visited states are $s_0,s_1,s_3,s_3,s_2,s_3$ and $a(26) = 2$.

Let  $b \colon \NN_0 \to \RR$ be the sequence given by $b(n) = (-1)^{\nu_2(n+1)}$,
where $\nu_2(m)$ is the largest value of $\nu$ such that $2^\nu \mid m$. For instance, $\nu_2(27) = 0$ and $b(26) = 1$. Then the structured part of $a$ is $a_{\mathrm{str}} = 2 + b$, and the uniform part is necessarily given by $a_{\mathrm{uni}} = a - a_{\mathrm{str}}$. Note that $b$ (and hence also $a_{\mathrm{str}}$ and $a_{\mathrm{uni}}$) can be computed by an automaton with the same states and transitions as above, but with different outputs.
Let also $c \colon \NN_0 \to \RR$ denote the sequence given by $c(n) = (-1)^{f(n)}$ where $f(n)$ is the number of those maximal blocks of $\mathtt{1}$s in the binary expansion of $n$ that have length congruent to $2$ or $3$ modulo $4$. For instance, $f(26) = 1$ and $c(26) = -1$. Then $a_{\mathrm{uni}} = (\frac{1}{2}+ \frac{1}{2}b)c$.
\end{example}

This example is very convenient as it allows one to give easy representations of the structured and uniform part.
However, the situation can be more complicated in general and we include another example to emphasize this fact.

\begin{example}\label{ex:main-2}
	Let $a \colon \NN_0 \to \RR$ be the $2$-automatic sequence computed by the following automaton. 
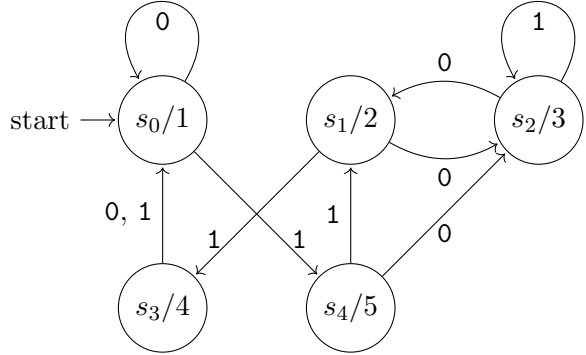
\begin{center}
\begin{tikzpicture}[shorten >=1pt,node distance=2.5cm, on grid, auto] 
   \node[initial, state] (s_0)   {$s_{0}/1$}; 
   \node[state] (s_3) [below=of s_0] {$s_{3}/4$}; 
   \node[state] (s_4) [right=of s_3] {$s_{4}/5$}; 
   \node[state] (s_1) [above=of s_4] {$s_{1}/2$}; 
   \node[state] (s_2) [right=of s_1] {$s_2/3$};

 \tikzstyle{loop}=[min distance=6mm,in=120,out=60,looseness=7]
    
	\path[->] (s_0) edge [loop below] node  {\texttt 0} (s_0);
	\path[->] 	(s_0) edge node [right, pos = 0.7] {\texttt 1} (s_4);
	
	\path[->] 	(s_1) edge [bend right] node [below] {\texttt 0} (s_2);
	\path[->] 	(s_1) edge node [left, pos = 0.7] {\texttt 1} (s_3);
	
	\path[->] (s_2) edge [bend right] node [above] {\texttt 0} (s_1);
	\path[->] (s_2) edge [loop below] node {\texttt 1} (s_2);
	
	\path[->] (s_3) edge node {\texttt 0, \texttt 1} (s_0);
	
	\path[->] (s_4) edge node [below] {\texttt 0} (s_2);
	\path[->] (s_4) edge node {\texttt 1} (s_1);

    
\end{tikzpicture}
\end{center}

	It turns out that the structured part can again be expressed using $b$, i.e., $a_{\mathrm{str}} = 3b -1$, but it is very difficult to find a simple closed form for the uniform part. Indeed, even writing it as an automatic sequence requires an automaton with $6$ states rather than the $5$ states needed for $a$.
\end{example}

We discuss three possible applications of Theorems~\ref{thm:main_simple} and~\ref{thm:main_nobasis} as well as of the related estimates of Gowers norms of automatic sequences. 
Firstly, they can be used to study subsequences of automatic sequences along various sparse sequences.  
Secondly, they allow us to count solutions to linear equations with variables taking values in automatic sets, that is, subsets of $\NN_0$ whose characteristic functions are automatic sequences. Lastly, they give a wide class of explicit examples of sequences with small Gowers norms of all orders. We will address these points independently.

We start by discussing the treatment of automatic sequences along primes by the third author to highlight the usefulness of a structural result as in Theorem~\ref{thm:main_simple}. 
In~\cite{Mullner-2017} a similar decomposition was used (with the uniform component satisfying a weaker property (called the Fourier-Property in \cite{Adamczewski2020}), which is almost the same as being Gowers uniform of order $1$) together with the so called carry Property (see already~\cite{MauduitRivat-2015} and a more general form in~\cite{Muellner2018}). This essentially allows one to reduce the problem to the case of structured and uniform sequences.
The structured component is very simple to deal with, as it suffices to study primes in arithmetic progressions. 
The study of the uniform component followed the method of Mauduit and Rivat developed to treat the Rudin--Shapiro sequence along primes \cite{MauduitRivat-2015}.
A similar approach was used by Adamczewski, Drmota and the third author to study the occurrences of digits in automatic sequences along squares~\cite{Adamczewski2020}.
It seems likely that a higher-order uniformity as in Theorem~\ref{thm:main_simple} might allow one to study the occurrences of blocks in automatic sequences along squares (see for example \cite{DrmotaMauduitRivat-TM-squares, Muellner2018} for related results).

Recently, Spiegelhofer used the fact that the Thue--Morse sequence is highly Gowers uniform to show that the level of distribution of the Thue--Morse sequence is $1$~\cite{Spiegelhofer2018}.
As a result, he proves that the sequence  is simply normal along $\floor{n^c}$ for $1<c<2$, i.e.\ the asymptotic frequency of both $0$ and $1$ in the Thue--Morse sequence along $\floor{n^c}$ is $1/2$. 
This result, together with our structural result (Theorem~\ref{thm:main_simple}) indicates a possible approach to studying automatic sequences produced by strongly connected, prolongable automata along $\floor{n^c}$. 
As the structured component is rationally almost periodic, we can simply study $\floor{n^c} \bmod m$ to deal with the first component.
The uniform component needs to be dealt with similarly to Spiegelhofer's treatment of the Thue--Morse sequence, but conditioned on $\floor{n^c} \bmod m$, to take care of the structured component at the same time.
For the possible treatment of all the subsequences of automatic sequences discussed above it is essential to have (for the uniform component) both some sort of Gowers uniformity as well as the carry Property. Both these properties are guaranteed by the decomposition used in this paper, while the Arithmetic Regularity Lemma cannot guarantee the carry Property for the uniform component.

Secondly, let us recall one of the many formulations of the celebrated theorem of Szemer\'{e}di on arithmetic progressions which says that any set $A \subset \NN_0$ with positive upper density $\overline{d}(A) = \limsup_{N \to \infty} \abs{ A \cap [N]}/N > 0$ contains arbitrarily long arithmetic progressions. It is natural to ask what number of such progressions are guaranteed to exist in $A \cap [N]$, depending on the length $N$ and the density of $A$.

Following the work of Bergelson, Host and Kra (and Ruzsa) \cite{BergelsonHostKra-2005}, Green and Tao \cite{GreenTao-2010-ARL} showed that for progressions of length $\leq 4$, the count of $d$-term arithmetic progressions in a subset $A \subset [N]$ is essentially greater than or equal to what one would expect for a random set of similar magnitude. 
\begin{theorem}\label{thm:many-AP-GT}
Let $2 \leq l \leq 4$, $\alpha > 0 $ and $\e > 0$. Then for any $N \geq 1$ and any $
A \subset [N]$ of density $\abs{A}/N \geq \alpha$ there exist $ \gg_{\alpha,\e} N$ values of $m \in [N]$ such that $A$ contains $\geq (\alpha^{l} - \e)N$ $l$-term arithmetic progressions with common difference $m$.
The analogous statement is false for any $l \geq 5$.
\end{theorem}

For automatic sets, the situation is much simpler: Regardless of the length $l \geq 1$, the count of $l$-term arithmetic progressions in $A \cap [N]$ is, up to a small error, at least what one would expect for a random set.

\begin{alphatheorem}\label{thm:many-AP-auto}
Let $l \geq 3$, and let $A$ be an automatic set (that is, a subset of $\NN_0$ whose characteristic sequence is automatic). Then there exists $C = O_{l,A}(1)$ such that for any $N \geq 1$ and $\e > 0$ there exist $ \gg_{l,A} \e^{C} N$ values of $m \in [N]$ such that $A \cap [N]$ contains $\geq (\alpha^{l} - \e)N$ $l$-term arithmetic progressions with common difference $m$, where $\alpha = \abs{A}/N$.
\end{alphatheorem}

Thirdly, we remark that there are few examples of sequences that are simultaneously known to be highly Gowers uniform and given by a natural, explicit formula. Polynomial phases $e(p(n))$ ($p \in \RR[x]$) are standard examples of sequences that are uniform of order $\deg p-1$ but dramatically non-uniform of order $\deg p$. Random sequences are highly uniform (cf.\ \cite[Ex.{} 11.1.17]{TaoVu-book}) but are not explicit. As already mentioned, many multiplicative sequences are known to be Gowers uniform of all orders, but with considerably worse bounds than the power saving which we obtain. For a similar result for a much simpler class of $q$-multiplicative sequences, see \cite{FanKonieczny-2019}. Examples of highly Gowers uniform sequences of number-theoretic origin in finite fields of prime order were found in \cite{FouvryKowalskiMichel-2013}; see also \cite{Liu-2010} and \cite{NiederreiterRivat-2009} where Gowers uniformity of certain sequences is derived from much stronger discorrelation estimates.





\section{Gowers norms}\label{sec:Gowers}
\newcommand{\conjugate}{\mathscr{C}}

\subsection{Notation}\label{ssec:asymptnot} We use standard asymptotic notation --- if $f$ and $g$ are two functions defined on (sufficiently large) positive integers, we write $f \ll g$ or $f=O(g)$ if there exists a constant $C>0$ such that $|f(n)| \leq C |g(n)|$ for all sufficiently large $n$. If the constant $C$ is allowed to depend on some extra parameters ($\alpha,\varepsilon$, etc.), we may specify that by writing $f \ll_{\alpha,\varepsilon} g$ or $f=O_{\alpha,\varepsilon}(g)$. In some cases when such dependence is clear from the context, we may omit such indices (this is the case for example for the order $d$ of Gowers uniformity norms, defined below).

We also use the Iverson bracket notation $\ifbra{P}$ for the value of a logical statement $P$, that is, $$\ifbra{P}=\begin{cases} 1&\text{if } P \text{ is true;}\\0&\text{otherwise}.\end{cases}$$

\subsection{Basic facts and definitions}\label{ssec:Gowers:basic}

Gowers norms, originally introduced by Gowers in his work on Szemer\'{e}di's theorem \cite{Gowers-2001}, are a fundamental object in what came to be known as higher order Fourier analysis. For extensive background, we refer to \cite{Green-book} or \cite{Tao-book}. Here, we just list several basic facts. \textit{Throughout, we treat $d$ (see below) as fixed unless explicitly stated otherwise, and allow all implicit error terms to depend on $d$.}

For a finite abelian group $G$ and an integer $d \geq 1$, the \emph{Gowers uniformity norm on $G$ of order $d$} is defined for $f \colon G \to \CC$ by the formula
\begin{equation}
\label{eq:def_Gowers_G}
\norm{f}_{U^d(G)}^{2^d} = \EE_{\vec n \in G^{d+1}} \prod_{\vec\omega \in \{0,1\}^d} \conjugate^{\abs{\vec\omega}} f( 1 \vec\omega \cdot \vec n), 
\end{equation}
where $\conjugate$ denotes the complex conjugation, $\vec \omega$ and $\vec n$ are shorthands for $(\omega_1,\dots,\omega_d)$ and $(n_0, n_1,\dots,n_d)$, respectively, 
$\abs{\vec\omega} = \abs{\set{i \leq d}{\omega_i = 1}}$, and $1 \vec \omega \cdot \vec n = n_0 + \sum_{i=1}^d \omega_i n_i$. More generally, for a family of functions $f_{\vec\omega} \colon G \to \CC$ with $\vec\omega \in \{0,1\}^d$ we can define the corresponding Gowers product
\begin{equation}
\label{eq:def_Gowers_product}
\left<\bra{f_{\vec\omega}}_{\vec\omega \in \{0,1\}^d} \right>_{U^d(G)} = \EE_{\vec n \in G^{d+1}} \prod_{\vec\omega \in \{0,1\}^d} \conjugate^{\abs{\vec\omega}} f_{\vec\omega}( 1 \vec\omega \cdot \vec n).
\end{equation}
A simple computation shows that $\norm{f}_{U^1(G)} = \abs{ \EE_{n \in G} f(n)}$ and 
\begin{align*}
	\norm{f}_{U^2(G)}^4 &= \EE_{n,m,l \in G} f(n) \bar f(n + m) \bar f(n + l) f(n + m + l)
	 = \sum_{\xi \in \hat{G}} \abs{ \hat{f}(\xi)}^4,
\end{align*}
where $\hat{G}$ is the group of characters $G \to \mathbb{S}^1$ and $\hat{f}(\xi) = \EE_{n \in G} \bar{\xi}(n) f(n)$.

One can show that definition \eqref{eq:def_Gowers_G} is well-posed in the sense that the right hand side of \eqref{eq:def_Gowers_G} is real and non-negative. If $d \geq 2$, then $\norm{\cdot}_{U^d(G)}$ is indeed a norm, meaning that it obeys the triangle inequality $\norm{f+g}_{U^d(G)} \leq \norm{f}_{U^d(G)} + \norm{g}_{U^d(G)}$, is positive definite in the sense that $\norm{f}_{U^d(G)} \geq 0$ with equality if only if $f = 0$, and is homogeneous in the sense that $\norm{\lambda f}_{U^d(G)} = \abs{\lambda} \norm{f}_{U^d(G)}$ for all $\lambda \in \CC$. If $d = 1$, then $\norm{\cdot}_{U^d(G)}$ is only a seminorm. Additionally, for any $d \geq 1$ we have the nesting property $\norm{f}_{U^d(G)} \leq \norm{f}_{U^{d+1}(G)}$.

In this paper we are primarily interested in the uniformity norms on the interval $[N]$, where $N \geq 1$ is an integer. Any such interval can be identified with the subset $[N]=\{0,1,\dots, N-1\}$ of a cyclic group $\ZZ/\tilde N \ZZ$, where $\tilde N$ is an integer significantly larger than $N$. For $d \geq 1$ and $f \colon [N] \to \CC$ we put
\begin{equation}
\label{eq:def_Gowers_N}
\norm{f}_{U^d[N]} = \norm{1_{[N]} f }_{U^d(\ZZ/\tilde N \ZZ)}/\norm{1_{[N]} }_{U^d(\ZZ/\tilde N \ZZ)}.
\end{equation}
The value of $\norm{f}_{U^d[N]}$ given by \eqref{eq:def_Gowers_N} is independent of $\tilde N$ as long as $\tilde N$ exceeds $2d N$, and for the sake of concreteness we let $\tilde N = \tilde N(N,d)$ be the least prime larger than $2d N$ (the primality assumption will make Fourier analysis considerations slightly easier at a later point). As a consequence of the corresponding properties for cyclic groups, $\norm{\cdot}_{U^{d}[N]}$ is a norm for all $d \geq 2$ and a seminorm for $d = 1$, and for all $d \geq 1$ we have a slightly weaker nesting property $\norm{f}_{U^d[N]} \ll_{d} \norm{f}_{U^{d+1}[N]}$.

Definition \eqref{eq:def_Gowers_N} can equivalently be expressed as
\begin{equation}
\label{eq:def_Gowers_N2}
\norm{f}_{U^d(G)}^{2^d} = \EE_{\vec n \in \Pi(N) } \prod_{\vec\omega \in \{0,1\}^d} \conjugate^{\abs{\vec\omega}} f( 1 \vec\omega \cdot \vec n), 
\end{equation}
where the average is taken over the set (implicitly dependent on $d$)
\begin{equation}
\label{eq:def_PiN}
\Pi(N) = \set{ \vec n \in \ZZ^{d+1}}{ 1\vec\omega \cdot \vec n \in [N] \text{ for all } \vec\omega \in \{0,1\}^d }.
\end{equation}
As a direct consequence of \eqref{eq:def_Gowers_N2}, we have the following phase-invariance: If $p \in \RR[x]$ is a polynomial of degree $< d$ and $g \colon [N] \to \CC$ is given by $g(n) = e(p(n))$, then $\norm{ f }_{U^d[N]} = \norm{ f \cdot g }_{U^d[N]}$ for all $f \colon [N] \to \CC$. (Here and elsewhere, $e(t) = \exp(2 \pi i t)$.) In particular, $\norm{ g }_{U^d[N]} = 1$. The analogous statement is also true for finite cyclic abelian groups. In particular, if $p \in \ZZ[x]$ is a polynomial of degree $< d$ and $g \colon \ZZ/N\ZZ \to \CC$ is given by $g(n) = e(p(n)/N)$, then $\norm{f}_{U^d(\ZZ/N\ZZ)} = \norm{f \cdot g}_{U^d(\ZZ/N\ZZ)} $ for all $f \colon \ZZ/N\ZZ \to \CC$.

We will  say that a bounded sequence $a \colon \NN_0 \to \CC$ is \emph{uniform of order} $d \geq 1$ if $\norm{a}_{U^d[N]} \to 0$ as $N \to \infty$. The interest in Gowers norms stems largely from the fact that uniform sequences behave much like random sequences in terms of counting additive patterns. To make this intuition precise, for a $(d+1)$-tuple of sequences $f_0,f_1,\dots,f_d \colon \NN_0 \to \CC$ let us consider the corresponding weighted count of arithmetic progressions
$$
	\Lambda_d^{N}(f_0,\dots,f_d) = \sum_{n,m \in \ZZ} \prod_{i=0}^d (f_i 1_{[N]})(n + im),
$$
so that in particular $\Lambda_d^{N}(1_A,\dots,1_A)$ is the number of arithmetic progressions of length $d+1$ in $A \cap [N]$. The following proposition is an easy variant of the generalised von Neumann theorem, see for example~\cite[Exercise 1.3.23]{Tao-book} We say that a function $f \colon X \to \CC$ is $1$-\emph{bounded} if $\abs{f(x)} \leq 1$ for all $x \in X$. 

\begin{proposition}
	Let $d \geq 1$ and let $f_0,f_1,\dots,f_d \colon \NN_0 \to \CC$ be $1$-bounded sequences. Then 
	$$
		\Lambda_d^{N}(f_0,\dots,f_d) \ll N^2 \min_{0 \leq i \leq d} \norm{f_i}_{U^d[N]}.
	$$  
\end{proposition}

As a direct consequence, if $f_i, g_i \colon \NN_0 \to \CC$ are $1$-bounded and $\norm{f_i - g_i}_{U^d[N]} \leq \e$ for all $0 \leq i \leq d$, then 
	$$
		\Lambda_d^{N}(f_0,\dots,f_d) = \Lambda_d^{N}(g_0,\dots,g_d) + O(\e N^2).
	$$  
In particular, if $A \subset \NN_0$ has positive asymptotic density $\alpha$ and $1_A - \alpha 1_{\NN_0}$ is  uniform of order $d$, then the count of $(d+1)$-term arithmetic progressions in $A \cap [N]$ is asymptotically the same as it would be if $A$ was a random set with density $\alpha$.

It is often helpful to control Gowers norms by other norms which are potentially easier to understand. We equip $[N]$ with the normalised counting measure, whence $\norm{f}_{L^p([N])} = \bra{ \EE_{n < N} \abs{f(n)}^p }^{1/p}$. The following bound is a consequence of Young's inequality (see e.g.\ \cite{EisnerTao-2012} for a derivation).  
\begin{proposition}\label{prop:Ud<Lp}
	Let $d \geq 1$ and $p_d = 2^d/(d+1)$. Then $\norm{f}_{U^d[N]} \ll \norm{f}_{L^{p_d}([N])}$ for any $f \colon [N] \to \CC$.
\end{proposition}

\subsection{Fourier analysis and reductions}
We will use some simple Fourier analysis on finite cyclic groups $\ZZ/N\ZZ$. We equip $\ZZ/N\ZZ$ with the normalised counting measure and its dual group $\widehat{\ZZ/N\ZZ}$ (which is isomorphic to $\ZZ/N\ZZ$) with the counting measure. With these conventions, the Plancherel theorem asserts that for $f \colon \ZZ/N\ZZ \to \CC$ we have
$$
	\EE_{n \in \ZZ/N\ZZ} \abs{f(n)}^2 = \norm{f}_{L^2(\ZZ/N\ZZ)}^2 = \norm{\hat{f}}_{\ell^2(\ZZ/N\ZZ)}^2 = \sum_{\xi \in \ZZ/N\ZZ} \abs{ \hat{f}(\xi)}^2,
$$
where $\hat{f}(\xi) = \EE_{n \in \ZZ/N\ZZ} f(n) e(-\xi n/N)$. Recall also that for $f,g \colon \ZZ/N\ZZ \to \CC$ we have $\widehat{f \ast g} = \hat{f} \cdot \hat{g}$ where $f \ast g(n) = \EE_{m \in \ZZ/N\ZZ} f(m)g(n-m)$.

The following lemma will allow us to approximate characteristic functions of arithmetic progressions with smooth functions. While much more precise variants exist (cf.\ Erd\H{o}s--Tur\'{a}n inequality), this basic result will be sufficient for the applications we have in mind. We say that a set $P \subset \ZZ/N\ZZ$ is an arithmetic progression of length $M$ if $\abs{P} = M$ and $P$ takes the form $\set{ am+b}{ m \in [M]}$ with $a,b \in \ZZ/N\ZZ$.  

\begin{lemma}\label{lem:smooth_approx}
	Let $N$ be prime and let $P\subset \ZZ/N\ZZ$ be an arithmetic progression of length $M \leq N$. Then for any $ 0 < \eta \leq 1$ there exists a function $f = f_{P,\eta} \colon \ZZ/N\ZZ \to [0,1]$ such that
	\begin{enumerate}
	\item $\norm{f - 1_{P}}_{L^p(\ZZ/N\ZZ)} \leq \eta^{1/p}$ for each $1 \leq p < \infty$;
	\item $\norm{ \hat{f} }_{\ell^1(\ZZ/N\ZZ)} \ll \eta^{-1/2}$.
	\end{enumerate}
\end{lemma}
\begin{remark}
	We will usually take $\eta = N^{-\e}$ where $\e > 0$ is a small constant.
\end{remark}
\begin{proof}
	We pick $f = 1_P * \frac{N}{K} 1_{a[K]}$, where $a$ is the common difference of the arithmetic progression and the integer $K \geq 1$ remains to be optimised.	
	Note that $f(n) \neq 1_P(n)$ for at most $2K$ values of $n \in \ZZ/N\ZZ$, and $\abs{f(n)-1_P(n)} \leq 1$ for all $n \in \ZZ/N\ZZ$. Hence, 
	$$
		\norm{f - 1_{P}}_{L^p} \leq \bra{2K/N}^{1/p}.
	$$
	Using the Cauchy--Schwarz inequality and Plancherel theorem we may also estimate
	\begin{align*}
		\norm{ \hat f }_{\ell^1} 
		&= \frac{N}{K} \norm{\hat 1_{P} \cdot \hat 1_{a[K] }}_{\ell^1} 
		\leq \frac{N}{K} \norm{\hat 1_P}_{\ell^2} \cdot  \norm{\hat 1_{a[K]}}_{\ell^2} 
		\\ & = \frac{N}{K} \norm{1_P}_{L^2} \cdot  \norm{1_{a[K]}}_{L^2} \leq (N/K)^{1/2}.
	\end{align*}
	It remains to put $K = \max(\floor{\eta N/2},1)$ and note that if $K=1$, then $f= 1_P$.
\end{proof}

As a matter of general principle, the restriction of a Gowers uniform sequence to an arithmetic progression is again Gowers uniform. We record the following consequence of Lemma \eqref{lem:smooth_approx} which makes this intuition more precise.

\begin{proposition}\label{cor:interval_Ud}
	Let $d \geq 2$ and $\alpha_d = (d+1)/(2^{d-1} + d + 1)$. Let $a \colon [N] \to \CC$ be a $1$-bounded function and let $P \subset [N]$ be an arithmetic progression. Then 
	$$
		\norm{a 1_{P} }_{U^d[N]} \ll \norm{a}_{U^d[N]}^{\alpha_d}.
	$$
(Recall that we allow the implicit constants to depend on $d$.)
\end{proposition}
\begin{proof}
	Throughout the argument we consider $d$ as fixed and allow implicit error terms to depend on $d$. Let $\tilde N = \tilde N(N,d)$ be the prime with $N < \tilde N \ll N$ defined in Section \ref{ssec:Gowers:basic}. Let $\eta > 0$ be a small parameter, to be optimised in the course of the proof, and let $f \colon \ZZ/\tilde N \ZZ \to [0,1]$ be the approximation of $1_P$ such that
	\[ \norm{f - 1_{P}}_{L^{p_d}(\ZZ/\tilde N\ZZ)} \ll \eta^{1/p_d} \text{ and } \norm{ \hat{f} }_{\ell^1(\ZZ/\tilde N\ZZ)} \ll \eta^{-1/2},\]
	whose existence is guaranteed by Lemma \ref{lem:smooth_approx}. (Recall that $p_d$ is defined in Proposition \ref{prop:Ud<Lp}.) Using the triangle inequality we can now estimate 
	\begin{align*}	
	\norm{a 1_{P} }_{U^d[N]} &\ll \norm{a 1_{P} }_{U^d(\ZZ/\tilde N\ZZ)} 
	= \norm{a \bra{f 1_{[N]} + (1_P-f)1_{[N]} } }_{U^d(\ZZ/\tilde N\ZZ)} 
	\\&\leq  \norm{a f 1_{[N]} }_{U^d(\ZZ/\tilde N\ZZ)} + \norm{a (1_P-f)1_{[N]} }_{U^d(\ZZ/\tilde N\ZZ)}.
	\end{align*}
	
	We consider the two summands independently. For the first one, expanding $f(n) = \sum_{\xi} \hat f(\xi) e(\xi n/\tilde N)$ and using phase-invariance of Gowers norms we obtain
$$
\norm{ a f 1_{[N]}}_{U^d(\ZZ/\tilde N\ZZ)} \leq \norm{\hat f }_{\ell^1(\ZZ/\tilde N\ZZ)} \cdot \norm{ a 1_{[N]} }_{U^d(\ZZ/\tilde N\ZZ)} \ll \eta^{-1/2} \norm{a}_{U^d[N]}. 
$$
For the second one, it follows from Proposition \ref{prop:Ud<Lp} that
\begin{align*}
 \norm{ a (1_P-f)1_{[N]} }_{U^d(\ZZ/\tilde N\ZZ)}
 &\ll \norm{a (1_P-f)1_{[N]} }_{L^{p_d}(\ZZ/\tilde N\ZZ)} 
 \\ &\leq \norm{1_P-f}_{L^{p_d}(\ZZ/\tilde N\ZZ)} 
 \leq \eta^{1/p_d}.
\end{align*}
It remains to combine the two estimates and insert the near-optimal value
$\eta = \norm{a}_{U^d[N]}^{{1}/\bra{1/2 + 1/p_d}}$. 
\end{proof}

We will use Proposition \ref{cor:interval_Ud} multiple times to estimate Gowers norms of restrictions of uniform sequences to sets which can be covered by few arithmetic progressions. For now, we record one immediate consequence, which will simplify the task of showing that a given sequence is Gowers uniform by allowing us to restrict our attention to uniformity norms on initial intervals whose length is a power of $k$.

\begin{corollary}\label{prop:wlog-N=k^L}
	Let $d \geq 2$ and $k \geq 2$. Let $a \colon \NN_0 \to \CC$ be a $1$-bounded sequence, and suppose that 
	\begin{equation}
	\label{eq:401}
	\norm{a}_{U^d[k^L]} \ll k^{-cL} \text{ as $L \to \infty$}
	\end{equation}
	for a constant $c > 0$. Then 
	\begin{equation}
	\label{eq:402}
	\norm{a}_{U^d[N]} \ll_{k} N^{-\alpha_d c} \text{ as $N \to \infty$}.
	\end{equation}
\end{corollary}
\begin{proof}
Let $N$ be a large integer and put $L = \ceil{\log_k N}$. We may then estimate
\begin{align*}
	\norm{a}_{U^d[N]} & \ll \norm{ a 1_{[N]} }_{U^d[k^L]} \ll \norm{ a }_{U^d[k^L]}^{\alpha_d}. \qedhere
\end{align*}
\end{proof}
\begin{remark} The argument is not specific to powers of $k$.
	The same argument shows that to prove that $\norm{a}_{U^d[N]} \ll N^{-c}$, it suffices to check the same condition for an increasing sequence $N_i$ where the quotients $N_{i+1}/N_i$ are bounded. 
\end{remark}
%
%
%

\section{Automatic sequences}\label{sec:Auto}
	
\subsection{Definitions}
In this section we review the basic terminology concerning automatic sequences. Our general reference for this material is \cite{ASbook}. To begin with, we introduce some notation concerning digital expansions.

For $k \geq 2$, we let $\Sigma_k = \{\texttt 0,\texttt 1,\dots,\texttt{k-1}\}$ denote the set of digits in base $k$. For a set $X$ we let $X^*$ denote the monoid of words over the alphabet $X$, with the operation of concatenation and the neutral element being the empty word $\epsilon$. In particular, $\Sigma_k^*$ is the set of all possible expansions in base $k$ (allowing leading zeros). While formally $\Sigma_k \subset \NN_0$, we use different fonts to distinguish between the digits $\texttt 0, \texttt 1, \texttt 2\dots$ and numbers $0,1,2,\dots$; in particular $\mathtt{11} = \mathtt{1}^2$ denotes the string of two $\mathtt{1}$s, while $11 = 10+1$ denotes the integer eleven. For a word $\word w\in X^*$, we let $|\word w|$ denote the length of the word $w$, that is, the number of letters it contains, and we let $\word w^{\mathrm{rev}}$ denote the word whose letters have been written in the opposite order (for instance, $\texttt{10110}^{\mathrm{rev}}= \texttt {01101}$).

For an integer $n \in \NN_0$, the expansion of $n$ in base $k$ without leading zeros is denoted by $(n)_k \in \Sigma_k^*$ (in particular $(0)_k = \epsilon$). Conversely, for a word $\word w \in \Sigma_k^*$ the corresponding integer is denoted by $[\word w]_k$. We also let $\length_k(n) = \abs{(n)_k}$ be the length of the expansion of $n$ (in particular $\length_k(\texttt 0) = 0$). 

Leading zeros are a frequent source of technical inconveniences, the root of which is the fact that we cannot completely identify $\NN_0$ with $\Sigma_k^*$. This motivates us to introduce another piece of notation. For $n \in \NN_0$ we let $(n)_k^l \in \Sigma_k^l$ denote the expansion of $n$ in base $k$ truncated or padded with leading zeros to length $l$, that is, $(n)_k^l$ is the suffix of the infinite word $0^\infty (n)_k$ of length $l$ (for example, $(43)_2^8 = \texttt {00101011}$ and $(43)_2^4 = \texttt {1011}$).

A (\emph{deterministic finite}) $k$-\emph{automaton without output} $\cA = (S,s_0,\Sigma_k,\delta)$ consists of the following data:
\begin{itemize}
	\item a finite \emph{set of states} $S$ with a distinguished \emph{initial state} $s_0$;
	\item a \emph{transition function} $\delta \colon S \times \Sigma_k \to S$.
\end{itemize}
A (\emph{deterministic finite}) $k$-\emph{automaton with output} $\cA = (S,s_0,\Sigma_k,\delta,\Omega,\tau)$ additionally includes
\begin{itemize}
	\item an \emph{output function} $\tau \colon S \to \Omega$ taking values in an \emph{output set} $\Omega$.
\end{itemize}
By an \emph{automaton} we mean a $k$-automaton for some unspecified $k \geq 2$. By default, all automata are deterministic, finite and with output. When we refer to automata without output, we say so explicitly.

The transition map $\delta \colon S \times \Sigma_k \to S$ extends naturally to a map (denoted by the same letter) $\delta \colon S \times \Sigma_k^* \to S$ so that $\delta(s,uv) = \delta(\delta(s,u),v)$. 
If $\cA = (S, s_0, \Sigma_{k}, \delta, \Omega, \tau)$ is an automaton with output, then $a_{\cA}$ denotes the \emph{automatic sequence produced by} $\cA$, which is defined by the formula $a(n) = \tau( \delta(s_0,(n)_k))$.  More generally, for $s \in S$, $a_{\cA,s}$ denotes the automatic sequence produced by $(S, s, \Sigma_{k}, \delta, \Omega, \tau)$; if $\cA$ is clear from the context, we simply write $a_s$. A sequence $a \colon \NN_0 \to \Omega$ is $k$-\emph{automatic} if it is produced by some $k$-automaton.

We say that an automaton (with or without output) with initial state $s_0$ and transition function $\delta$ is \emph{prolongable} (or \emph{ignores the leading zeros}) if $\delta(s_0,\texttt{0}) = s_0$. Any automatic sequence can be produced by an automaton ignoring leading zeros. We call an automaton $\cA$ \emph{idempotent} if it ignores the leading zeros and $\delta(s,\texttt{00}) = \delta(s,\texttt{0})$ for each $s \in S$, that is, if the map $\delta(\cdot,\texttt{0}) \colon S \to S$ is idempotent.

Note that with the above definitions, automata read input forwards, that is, starting with the most significant digit. One can also consider the opposite definition, where the input is read backwards, starting from the least significant digit, that is, $a_{\cA}^{\mathrm{rev}}(n) = \tau\bra{ \delta(s_0,(n)_k^{\mathrm{rev}})}$. The class of sequences produced by automata reading input forwards is precisely the same as the class of sequences produced by automata reading input backwards. However, the two concepts lead to different classes of sequences if we impose additional assumptions on the automata, such as synchronisation.

An automaton $\cA$ is \emph{synchronising} if there exists a \emph{synchronising word} $\word w \in \Sigma_k^*$, that is, a word $\word w$ such that the value of $\delta(s,\word w)$ does not depend on the state $s \in S$. Note that a synchronising word is by no means unique; indeed, any word $\word w'$ containing a synchronising word as a factor is itself synchronising. As a consequence, if $\cA$ is synchronising then the number of words $\word w \in \Sigma_k^l$ that are not synchronising for $\cA$ is $\ll k^{l(1-c)}$ for some constant $c > 0$. An automatic sequence is \emph{forwards} (resp.\ \emph{backwards}) \emph{synchronising} if it is produced by a synchronising automaton reading input forwards (resp.\ backwards).

An automaton $\cA$ is \emph{invertible} if for each $j \in \Sigma_k$ the map $\delta(\cdot,j) \colon S \to S$ is bijective and additionally $\delta(\cdot,\texttt{0}) = \id_S$. A sequence is \emph{invertible} if it is produced by an invertible automaton (reading input forwards). One can show that reading input backwards leads to the same notion, but we do not need this fact. 
Any invertible sequence is a coding of a generalised Thue--Morse sequence, meaning that there exists a group $G$ and group elements $\id_G = g_0, g_1,\dots, g_{k-1}$ such that the sequence is produced by an automaton with $S = G$, $s_0 = e_G$ and $\delta(s,j) = s g_j$ for each $j \in \Sigma_k$ \cite{DrmotaMorgenbesser-2012}.

A state $s$ in an automaton $\cA$ is \emph{reachable} if $\delta(s_0,\word w)=s$ for some $\word w\in \Sigma_k^*$. Unreachable states in an automaton are usually irrelevant, as we may remove them from the automaton without changing the automatic sequence produced by it. 
We call two distinct states $s,s'\in S$ satisfying  $\tau(\delta(s,\word v)) = \tau(\delta(s',\word v))$ for all $\word v \in \Sigma_k^*$ \emph{nondistinguishable}. One sees directly, that we could merge them (preserving outgoing arrows of one of the states) and still obtain a well-defined automaton producing $a$ and having a smaller number of states. 
This leads us to the definition of a \emph{minimal} automaton, i.e. an automaton with no unreachable states and no nondistinguishable states. It is classical, that for any automatic sequence there exists a minimal automaton producing that sequence (see for example~\cite[Corollary 4.1.9]{ASbook}).

An automaton $\cA$ is \emph{strongly connected} if for any two states $s, s'$ of $\cA$ there exists $\word w \in \Sigma_k^*$ with $\delta(s,\word w)=s'$. A \emph{strongly connected component} of $\cA$ is a strongly connected automaton $\cA'$ whose set of states $S'$ in a subset of $S$ and whose transition function $\delta'$ is the restriction of the transition function $\delta$ of $\cA$; we often identify $\cA'$ with $S'$. The following observation is standard, but we include the proof for the convenience of the reader.

\begin{lemma}\label{lem:killing-word}
	Let $\cA$ be an automaton, as introduced above. Then there exists a word $\word w \in \Sigma_k^*$ such that if $\word v \in \Sigma_k^*$ contains $\word w$ as a factor then for every state $s$ of $\cA$, $\delta(s,\word v)$ belongs to a strongly connected component of $\cA$. 
\end{lemma}
\begin{proof}
	Let $S = \{s_0,s_1,\dots, s_{N-1}\}$ be an enumeration of $S$. We construct inductively a sequence of words $\epsilon = \word w_0, \dots, \word w_N$, with the property that $\delta(s_i, \word w_j)$ belongs to a strongly connected component for any $0 \leq i < j \leq N$. Once $\word w_j$ has been constructed, it is enough to define $\word w_{j+1} = \word w_j \word u$, where $\word u \in \Sigma_k^*$ is an arbitrary word such that $\delta\bra{ \delta(s_j, \word w_j), \word u}$ belong to a strongly connected component, which is possible since from any state there exists a path leading to a strongly connected component.
\end{proof}

We can consider $k$-automata with or without output as a category. A morphism between automata without output $\cA = (S, s_0, \Sigma_{k}, \delta)$ and $\cA = (S', s_0', \Sigma_{k}, \delta')$ is a map $\phi \colon S \to S'$ such that $\phi(s_0) = s_0'$ and $\phi(\delta(s,j)) = \delta'(\phi(s),j)$ for all $s \in S$ and $j \in \Sigma_k$. A morphism between automata with output $\cA = (S, s_0, \Sigma_{k}, \delta, \Omega, \tau)$ and  $\cA' = (S', s_0', \Sigma_{k}, \delta', \Omega', \tau')$ is a pair $(\phi, \sigma)$ where $\phi$ is a morphism between the underlying automata without output and $\sigma \colon \Omega \to \Omega'$ is a map such that $\sigma(\tau(s)) = \tau'(\phi(s))$. In the situation above, $a_{\cA'}$ is the image of $a_{\cA}$ via a coding, that is, $a_{\cA'}(n) = \sigma(a_{\cA}(n))$ for all $n \in \NN_0$. While this---perhaps overly abstract---terminology is not strictly speaking needed for our purposes, it will be helpful  at a later point when we consider morphisms between \geas.

\subsection{Change of base}\label{ssec:Auto:base}

 A sequence $a \colon \NN_0 \to \Omega$ is \emph{eventually periodic} if there exists $n_0 \geq 0$ and $d \geq 1$ such that $a(n+d) = a(n)$ for all $n \geq n_0$. Two integers $k,k' \geq 2$ are \emph{multiplicatively independent} if $\log(k)/\log(k')$ is irrational. A classical theorem of Cobham asserts that if $k,k' \geq 2$ are two multiplicatively independent integers, then the only sequences which are both $k$- and $k'$-automatic are the eventually periodic ones, and those are automatic in all bases. On the other hand, if $k,k' \geq 2$ are multiplicatively dependent, meaning that $k = k_0^l$ and $k' = k_0^{l'}$ for some integers $k_0,l,l'\geq 1$, then the classes of $k$-automatic and $k'$-automatic sequences coincide. 

Hence, when we work with a given automatic sequence that is not ultimately periodic, the base (denoted by $k$) is determined uniquely up to the possibility to replace it by its power $k' = k^{t}$, $t \in \QQ$. We will take advantage of this possibility, which is useful because some of the properties discussed above (specifically synchronisation and idempotence) depend on the choice of base. We devote the remainder of this section to recording how various properties of automatic sequences behave when the base is changed. An instructive example to keep in mind is that $n \mapsto \length_2(n) \bmod 2$ is backwards synchronising in base $4$ but not in base $2$ (see Proposition \ref{prop:sync-vs-basis} for details).

We first briefly address the issue of idempotency. Any automatic sequence is produced by an idempotent automaton, possibly after a change of basis \cite[Lem.\ 2.2]{ByszewskiKonieczny-Cobham}. Additionally, if the sequence $a_{\cA}$ is produced by the automaton $\cA = (S, s_0, \Sigma_{k}, \delta, \Omega, \tau)$ then for any power $k' = k^l$, $l \in \NN$, there is a natural construction of a $k'$-automaton $\cA'$ which produces the same sequence $a_{\cA'} = a_{\cA}$ and is idempotent.

We next consider synchronising sequences. The following lemma provides a convenient criterion for a sequence to be synchronising.

\begin{lemma}\label{lem:sync-FCAE}
	Let $a \colon \NN_0 \to \Omega$ be a $k$-automatic sequence and let $\word w \in \Sigma_k^*$. Then the following conditions are equivalent:
	\begin{enumerate}[wide]
		\item\label{it:749:A} the sequence $a$ is produced by a $k$-automaton $\cA$ reading input forwards (resp.\ backwards) for which $\word w$ is synchronising;
		\item\label{it:749:B} there exists a map $b \colon \Sigma_k^* \to \Omega$ such that for any $\word{u}, \word{v} \in \Sigma_k^*$ we have $a( [\word{u} \word{w} \word{v}]_k )= b(\word{v})$ (resp.\ $a( [\word{u} \word{w} \word{v}]_k )= b(\word{u})$). 
	\end{enumerate}	
\end{lemma}
\begin{proof} For the sake of clarity we only consider the ``forward'' variant; the ``backward'' case is fully analogous. It is clear that \eqref{it:749:A} implies \eqref{it:749:B}, so it remains to prove the reverse implication. 
Let $\cA$ be a minimal $k$-automaton which produces $a$. We will show that if $\word w$ satisfies \eqref{it:749:B} then it is synchronising for $\cA$.

Let $s,s' \in S$ be any two states. Pick $\word u, \word u'$ such that $s = \delta(s_0,\word u)$ and $s' = \delta(s_0,\word u')$. Since
\[
	\tau(\delta(s,\word w \word v)) = a( [\word{u} \word{w} \word{v}]_k ) = b(\word v) = 
	 a( [\word{u}' \word{w} \word{v}]_k ) = \tau(\delta(s',\word w \word v))
\]
for any $\word v, \word w \in \Sigma_k^*$, we get that $\tau(\delta(\delta(s, \word w), \word v )) = \tau(\delta(\delta(s', \word w), \word v))$ for all $\word v \in \Sigma_k^{*}$. This implies by minimality of $\cA$ that $\delta(s, \word w) = \delta(s', \word w)$.
Thus, we have showed that the word $\word w$ is synchronising.
\end{proof}

	As a consequence, we obtain a good understanding of how a change of base affects the property of being synchronising.

\begin{proposition}\label{prop:sync-vs-basis}
	Let $a \colon \NN_0 \to \Omega$ be a $k$-automatic sequence and let $l \in \NN$.
	\begin{enumerate}[wide]
	\item If $a$ is forwards (resp.\ backwards) synchronising as a $k$-automatic sequence, then $a$ is also forwards (resp.\ backwards) synchronising as a $k^l$-automatic sequence.
	\item If $a$ is forwards synchronising as a $k^l$-automatic sequence, then $a$ is also forwards synchronising as a $k$-automatic sequence.
	\item If $l \geq 2$ then there exist backwards synchronising $k^l$-automatic sequences which are not backwards synchronising as $k$-automatic sequences.
	\end{enumerate}
\end{proposition}
\begin{proof}
\begin{enumerate}[wide]
\item Let $\word w \in \Sigma_k^*$ be a synchronising word for a $k$-automaton producing $a$. Replacing $\word w$ with a longer word if necessary, we may assume without loss of generality that the length of $\word w$ is divisible by $l$. Hence, we may identify $\word w$ with an element of $\Sigma_{k^l}^* \simeq \bra{\Sigma_k^l}^*$ in a natural way. It follows from Lemma \ref{lem:sync-FCAE} that $\word w$ is a synchronising word for a $k^l$-automaton producing $a$. 
\item Let $\word w \in \Sigma_{k^l}^* \simeq \bra{\Sigma_k^l}^*$ be a synchronising word for a $k^l$-automaton which produces $a$ and consider the word $\word w' = (\word w0)^l \in \Sigma_k^*$. This is set up so that if the expansion $(n)_k$ of an integer $n \geq 0$ contains $\word w'$ as a factor then $(n)_{k^l}$ contains $\word w$ as a factor. It follows from Lemma \ref{lem:sync-FCAE} that $\word w'$ is a synchronising word for a $k$-automaton producing $a$. 
\item Consider the sequence $b(n) = \length_k(n) \bmod{l}$. In base $k^l$, the value of $b(n)$ depends only on the leading digit of $n$, whence $b$ is backwards synchronising. On the other hand, $b([\word v]_k) \neq b([\word v\texttt{0}]_k)$ for all $\word v \in \Sigma_k^*$ with $[\word v]_k \neq 0$, whence $b$ is not backwards synchronising as a $k$-automatic sequence. \qedhere
	\end{enumerate}
\end{proof}

\section{Derivation of the main theorems}

\subsection{Strongly connected case}

Having set up the relevant terminology in Sections \ref{sec:Gowers} and \ref{sec:Auto}, we are now ready to deduce our main results, Theorems \ref{thm:correlations}, \ref{thm:main_simple}, \ref{thm:main_nobasis} and \ref{thm:many-AP-auto} from the following variant, applicable to strongly connected automata. We also address the issue of uniqueness of the decomposition in Theorems \ref{thm:main_simple} and \ref{thm:main_nobasis}.

We say that a $k$-automatic sequence $a \colon \NN_0 \to \Omega$ is \emph{strongly structured} if there exists a periodic sequence $a_{\mathrm{per}} \colon \NN_0 \to \Omega_{\mathrm{per}}$ with period coprime to $k$, a forwards synchronising $k$-automatic sequence $a_{\mathrm{fs}} \colon \NN_0 \to \Omega_{\mathrm{fs}}$, as well as a map $F \colon \Omega_{\mathrm{per}} \times \Omega_{\mathrm{fs}} \to \Omega$ such that 
\begin{equation}\label{eq:def-of-st-str}
	a(n) = F\bra{ a_{\mathrm{per}}(n), a_{\mathrm{fs}}(n)}.
\end{equation}
Note that thanks to Proposition \ref{prop:sync-vs-basis} this notion does not change upon replacing the base $k$ by a multiplicatively dependent one.

\begin{theorem}\label{thm:main_sc_nobasis}
	Let $a\colon \NN_0 \to \CC$ be a $k$-automatic sequence produced by a strongly connected, prolongable automaton. Then there exists a decomposition 
	\begin{equation}
	\label{eq:003a}
	 a = a_{\mathrm{str}} + a_{\mathrm{uni}},
	\end{equation}
	where $a_{\mathrm{str}}$ is strongly structured (cf.\ \eqref{eq:def-of-st-str}) and $a_{\mathrm{uni}}$ is highly Gowers uniform (cf.\ \eqref{eq:def-of-h-G-uni}).
\end{theorem}

Note that the formulation of Theorem \ref{thm:main_sc_nobasis} is very reminiscent of Theorem \ref{thm:main_simple}, except that the assumptions on the structured part are different. Indeed, one is an almost immediate consequence of the other.

\begin{proof}[Proof of Theorem \ref{thm:main_simple} assuming Theorem \ref{thm:main_sc_nobasis}]
	The only difficulty is to show that any forwards synchronising automatic sequence is rationally almost periodic. This is implicit in \cite{DeshouillersDrmotaMullner-2015}, and showed in detail in \cite[Proposition 3.4]{BergelsonKulagaPrzymusLemanczykRichter-2016}. It follows that any strongly structured sequence is rationally almost periodic. 
\end{proof}

The derivation of Theorem \ref{thm:main_nobasis} is considerably longer, and involves reconstruction of an automatic sequence produced by an arbitrary automaton from the automatic sequences produced by the strongly connected components.

\begin{proof}[Proof of Theorem \ref{thm:main_nobasis} assuming Theorem \ref{thm:main_sc_nobasis}]

Let $a \colon \NN_0 \to \CC$ be an automatic sequence. We may assume (changing the base if necessary) that $a$ is produced by an idempotent automaton $\cA = (S,s_0,\Sigma_k,\delta,\CC,\tau)$ with $\delta(s_0,\texttt{0}) = s_0$. Throughout the argument we consider $\cA$ to be fixed and we do not track dependencies of implicit error terms on $\cA$.

Let $S_0$ denote the set of states $s \in S$ which lie in some strongly connected component of $S$ which also satisfy $\delta(s, \texttt{0}) = s$ (or, equivalently, $\delta(s',\texttt{0}) = s$ for some $s' \in S_0$). Note that each strongly connected component of $S$ contains a state in $S_0$. For each $s \in S_0$, the sequence $a_s = a_{\cA,s}$ is produced by a strongly connected automaton, so it follows from Theorem \ref{thm:main_sc_nobasis} that there exists a decomposition
\[
	a_s = a_{s,\mathrm{str}} + a_{s,\mathrm{uni}},
\]
where $a_{s,\mathrm{str}}$ is strongly structured and $a_{s,\mathrm{uni}}$ is highly Gowers uniform. For $s \in S_0$ let
\[
	a_{s,\mathrm{str}}(n) = F_s\bra{ a_{s,\mathrm{per}}(n), a_{s,\mathrm{fs}}(n)}
\]
be a representation of $a_{s,\mathrm{str}}$ as in \eqref{eq:def-of-st-str}. Let $M$ be an integer coprime to $k$ and divisible by the period of $a_{s,\mathrm{per}}$ for each $s \in S_0$ (for instance, the least common multiple of these periods). Let $\word z \in \Sigma_k^*$ be a word that is synchronising for $a_{s,\mathrm{fs}}$ for each $s \in S_0$ (it can be obtained by concatenating synchronising words for all strongly connected components of $\cA$).

We will also need a word $\word y \in \Sigma_k^*$ with the property that if we run $\cA$ on input which includes $\word y$ as a factor, we will visit a state from $S_0$ at some point when the input read so far encodes an integer divisible by $M$. More formally, we require that for each $\word u \in \Sigma_k^*$ there exists a decomposition $\word y = \word x_1 \word x_2$ such that $\delta(s_0, \word u \word x_1) \in S_0$ and $M \mid [\word u \word x_1]_k$. The word $\word y$ can be constructed as follows. Take a word  $\word y_0 \in \Sigma_k^*$ with the property that $\delta(s,\word y_0)$ belongs to a strongly connected component for each $s \in S$, whose existence is guaranteed by Lemma \ref{lem:killing-word}. Let $A \geq 1$ be an integer that is multiplicatively rich enough that $M \mid k^A-1$, and let $B \geq M-1$. Put $\word y = \word y_0 (\texttt{0}^{A-1}\texttt{1})^B$. Then, using notation above, we can take $\word x_1= \word y_0(\texttt{0}^{A-1}\texttt{1})^i$, where $i \equiv - [\word u \word y_0]_l \bmod{M}$.

For $n \in \NN_0$ such that $(n)_k$ contains $\word y \word z$ as a factor, fix the decomposition $(n)_k = \word u_n \word v_n$ where $\delta(s_0,\word u_n) \in S_0$, $M \mid [\word u_n]_k$ and $\word u_n$ is the shortest possible subject to these constraints. Note that $\word v_n$ contains $\word z$ as a factor. Let $Z \subset \NN_0$ be the set of those $n$ for which $(n)_k$ does not contain $\word y \word z$ as a factor, and for the sake of completeness  define $\word u_n = \word v_n = \diamondsuit$ for $n \in Z$, where $\diamondsuit$ is a symbol not belonging to $\Sigma_k^*$. Note also that there exists a constant $\gamma > 0$ such that $\abs{Z \cap [N]} \ll N^{1-\gamma}$.

We are now ready to identify the structured part of $a$, which is given by
\begin{equation}\label{eq:837:1}
	a_{\mathrm{str}}(n) = \sum_{s \in S_0} \ifbra{ \delta(s_0,\word u_n) = s } a_{s,\mathrm{str}}(n).
\end{equation}
(If $n \in Z$, the statement $\delta(s_0,\word u_n) = s$ is considered to be false by convention, whence in particular $a_{\mathrm{str}}(n) = 0$; recall that $\ifbra{ \delta(s_0,\word u_n) = s }$ uses the Iverson bracket notation, that is, $\ifbra{ \delta(s_0,\word u_n) = s }$ equals $1$ if $\delta(s_0,\word u_n) = s$, and equals $0$ otherwise.) 
The uniform part is now necessarily given by $a_{\mathrm{uni}} = a - a_{\mathrm{str}}$. It remains to show that $a_{\mathrm{str}}$ and $a_{\mathrm{uni}}$ are strongly structured and highly Gowers uniform, respectively (note that strongly structured sequences are necessarily automatic).

We begin with $a_{\mathrm{str}}$. For any $s \in S_0$, we will show that $n \mapsto  \ifbra{ \delta(s_0,\word u_n) = s }$ is a backwards synchronising $k$-automatic sequence. This is most easily accomplished by describing a procedure which computes it. To this end, we consider an automaton that mimics the behaviour of $\cA$, and additionally keeps track of the remainder modulo $M$ of the part of the input read so far. Next, we modify it so that if an arbitrary state $s'$ in $S_0$ and residue $0$ is  reached, the output becomes fixed to $\ifbra{s'=s}$. The output for all remaining pairs of states and residues are $0$. More formally, we take $\cA' = \bra{ S \times (\ZZ/M\ZZ), (s_0,0), \Sigma_k, \delta', \{0,1\},\tau') }$, where $\delta'$ is given by
\[
	\delta'((r,i),j) = 
	\begin{cases}
		( \delta(r,j), ki + j \bmod{M} ) &\text{ if } i \neq 0 \text{ or } r \not \in S_0,\\
		(r,i) &\text{ otherwise},
	\end{cases}
\]
and the output function is given by 
\[
	\tau'(r,i) = 
	\begin{cases}
		0&\text{ if } i \neq 0 \text{ or } r \not \in S_0,\\
	 	 \ifbra{ r = s}& \text{ otherwise}.
	\end{cases}
\]
It is clear that $a_{\cA'} = \ifbra{ \delta(s_0,\word u_n) = s }$ for all $n \in \NN_0$. Additionally, since the output becomes constant once we read $\word y \word z$, this procedure gives rise to a backwards synchronising sequence. Hence, each of the summands in \eqref{eq:837:1} is the product of a backwards synchronising sequence and a strongly structured sequence. 
Moreover, we have by Lemma~\ref{lem:sync-FCAE} that the cartesian product of forwards (backwards) synchronizing $k$-automatic sequences is again a forwards (backwards) synchronizing $k$-automatic sequence. A synchronizing word for the new automaton can be constructed by concatenating synchronizing words of the individual automata. 
Thus, $a_{\mathrm{str}}$ is weakly structured. 


Next, let us consider $a_{\mathrm{uni}}$. Thanks to Proposition \ref{prop:wlog-N=k^L}, we only need to show that for any $d \geq 2$ there exists a constant $c > 0$ such that $\norm{a_{\mathrm{uni}}}_{U^d[k^L]} \ll k^{-cL}$. Fix a choice of $d$ and let $L$ be a large integer. If $n \in \NN_0 \setminus Z$ and $s = \delta(s_0, \word u_n)$, then 
\begin{align*}
	a(n) &= a\bra{[\word u_n \word v_n]_k} = a_s \bra{[ \word v_n]_k}
	 =  a_{s,\mathrm{str}}([\word v_n]_k) + a_{s,\mathrm{uni}}([\word v_n]_k) 
	\\&= F_s\bra{ a_{s,\mathrm{per}}([\word v_n]_k), a_{s,\mathrm{fs}}([\word v_n]_k)} + a_{s,\mathrm{uni}}([\word v_n]_k) 
	\\&= F_s\bra{ a_{s,\mathrm{per}}(n), a_{s,\mathrm{fs}}(n)} + a_{s,\mathrm{uni}}([\word v_n]_k)
	= a_{s,\mathrm{str}}(n) + a_{s,\mathrm{uni}}([\word v_n]_k),
\end{align*}
where in the last line, we have used the fact $M \mid [\word u_n]_k$ and $\word v_n$ is synchronising for $a_{s,\mathrm{fs}}$. Since $a_{\mathrm{str}}(n) = a_{s,\mathrm{str}}(n)$, it follows that
\[
	 a_{\mathrm{uni}}(n) = a_{s,\mathrm{uni}}([\word v_n]_k).
\]
For a word  $\word x \in \Sigma_k^*$ containing $\word y \word z$ as a factor and integer $l \geq 0$, consider the interval
\begin{equation}\label{eq:def-of-P}
	P = \set{ [\word w]_k }{ \word w \in \word x \Sigma_k^l} = \left[ [\word x]_k k^l, \bra{[\word x]_k+1} k^l \right).  
\end{equation} 
Since $\word u_n$ and $\abs{ \word v_n }$ are constant on $P$, it follows from Proposition \ref{cor:interval_Ud} and the assumption that $a_{s,\mathrm{uni}}$ are highly Gowers uniform that 
\[
	\norm{a_{\mathrm{uni}} 1_P}_{U^d[k^L]}=\norm{a_{s,\mathrm{uni}} 1_P}_{U^d[k^L]} \ll \max_{s \in S_0} \norm{a_{s,\mathrm{uni}} }_{U^d[k^L]}^{\alpha_d} \ll k^{-c'L} 
\]
for some constant $1>c' > 0$, which does not depend on $P$. It remains to cover $[k^L]$ with a moderate number of intervals $P$ of the form \eqref{eq:def-of-P} and a small remainder set.

Let $\eta > 0$ be a small parameter to be optimised in the course of the argument and let $R$ be the set of those $n \in [k^L]$ which are \emph{not} contained in any progression $P$ given by \eqref{eq:def-of-P} with $l \geq (1-\eta)L$. Hence, if $n \in R$ then the word $\word y \word z$ does not appear  in the leading $\lfloor \eta L \rfloor$ digits of $(n)_k^L$. It follows that $\abs{R} \ll k^{-c''_0 \eta L}$ and consequently 
\[
\norm{a_{\mathrm{uni}} 1_R }_{U^d[k^L]} \ll \norm{a_{\mathrm{uni}} 1_R }_{L^{p_d}[k^L]} \ll 
k^{-c'' \eta L}
\]
by Proposition \ref{prop:Ud<Lp}, where $c''_0 > 0$ and $c'' =c''_0/p_d$ are constants. Each $n \in [k^L] \setminus R$ belongs to a unique interval $P$ given by \eqref{eq:def-of-P} with  $l \geq (1-\eta)L$ and such that no proper suffix of $\word x$ contains $\word y \word z$. There are $\leq k^{\eta L}$ such intervals, corresponding to the possible choices of initial $\lfloor \eta L \rfloor$ digits of $(n)_k^L$ for $n \in P$. It now follows from the triangle inequality that
\begin{align*}
	\norm{a_{\mathrm{uni}} }_{U^d[k^L]} &\leq \norm{a_{\mathrm{uni}} 1_R }_{U^d[k^L]} + \sum_{P} \norm{a_{\mathrm{uni}} 1_P}_{U^d[k^L]}
	\ll k^{-c'' \eta L} + k^{(\eta - c') L}.
\end{align*}
It remains to pick $\eta = c'/2$, leading to $\norm{a_{\mathrm{uni}} }_{U^d[k^L]} 
	\ll k^{-c \eta L}$ with  $c = c'\min(c'',1)/2$.
\end{proof}

Finally, we record another reduction which will allow us to alter the initial state of the automaton in the proof of Theorem \ref{thm:main_sc_nobasis}. As the proof of the following result is very similar and somewhat simpler than the proof of Theorem \ref{thm:main_nobasis} discussed above, we skip some of the technical details. If fact, one could repeat said argument directly, only replacing $S_0$ with a smaller set (namely, a singleton); we do not pursue this route because a simpler and more natural argument is possible.

\begin{proposition}\label{prop:wlog_change_s0}
	Let $\cA = (S,s_0,\Sigma_k,\delta,\Omega,\tau)$ be a strongly connected, prolongable automaton and let $S_0 \subset S$ be the set of $s \in S$ such that $\delta(s,\texttt{0}) = s$. Then the following statements are equivalent:
	\begin{enumerate}
	\item\label{item:43A} Theorem \ref{thm:main_sc_nobasis} holds for $a_{\cA,s}$ for some $s \in S_0$;
	\item\label{item:43B} Theorem  \ref{thm:main_sc_nobasis} holds for $a_{\cA,s}$ for all $s \in S_0$.
	\end{enumerate}
\end{proposition}
\begin{proof}
	It is clear that \eqref{item:43B} implies \eqref{item:43A}. For the other implication, we may assume that Theorem \ref{thm:main_sc_nobasis} holds for $a_{\cA,s_0} = a_{\cA}$. Hence, there exists a decomposition $a_{\cA} = a_{\mathrm{str}} + a_{\mathrm{uni}}$ of $a_{\cA}$ as the sum of a strongly structured and highly Gowers uniform sequence. Let
\[
	a_{\mathrm{str}}(n) = F\bra{ a_{\mathrm{per}}(n), a_{\mathrm{fs}}(n)}
\]
	be a representation of $a_{\mathrm{str}}$ as in \eqref{eq:def-of-st-str}.
	
	Pick any $s \in S_0$ and pick $\word u \in \Sigma_k^*$, not starting with $\texttt{0}$ and such that $\delta(s_0,\word u) = s$, whence $a_{\cA,s}(n) = a_{\cA}([\word u(n)_k]_k)$ for all $n \in \NN_0$. Since $\delta(s,\texttt{0}) = s$, we also have $a_{\cA,s}(n) = a_{\cA}([\word u\texttt{0}^m(n)_k]_k)$ for any $m,n \in \NN_0$. Let $Q$ be a multiplicatively large integer, so that the period of $a_{\mathrm{per}}$ divides $k^Q - 1$, and put $m(n) := Q-(\length_k(n) \bmod Q) \in \{1,2,\dots,Q\}$. For $n \in \NN_0$ put
\[
	a'_{\mathrm{str}}(n) := a_{\mathrm{str}}([\word u \texttt{0}^{m(n)} (n)_k]_k)
	\ \text{ and } \
	 a'_{\mathrm{uni}}(n) := a_{\cA,s}(n) - a'_{\mathrm{str}}(n).
\]
	Clearly, $a_{\cA,s} = a'_{\mathrm{str}} + a'_{\mathrm{uni}}$.  Since the period of $a_{\mathrm{per}}$ divides $k^Q-1$, for all $n \in \NN_0$ we have
	\begin{equation}\label{eqn:per'def}
	a_{\mathrm{per}}([\word u \texttt{0}^{m(n)} (n)_k]_k) = a_{\mathrm{per}}(n+[\word u]_k)
	\end{equation}
Define the sequences $a_{\mathrm{per}}'$ and $a_{\mathrm{fs}}'$ by the formulas
	\[
	a_{\mathrm{per}}'(n) := a_{\mathrm{per}}([\word u \texttt{0}^{m(n)} (n)_k]_k), \qquad a_{\mathrm{fs}}'(n) := a_{\mathrm{fs}}([\word u \texttt{0}^{m(n)} (n)_k]_k).
	\]
	It follows from \eqref{eqn:per'def} that $a_{\mathrm{per}}'$ is periodic. Since the sequence $m(n)$ is $k$-automatic, so is $a'_{\mathrm{fs}}$. Indeed, in order to compute  $a'_{\mathrm{fs}}(n)$ it is enough to compute $m(n)$ and $a_{\mathrm{fs}}([\word u \texttt{0}^{i} (n)_k]_k)$ for $1 \leq i \leq Q$. Since $a_{\mathrm{fs}}$ is forwards synchronising, it follows from Lemma \ref{lem:sync-FCAE} that so is $a'_{\mathrm{fs}}$. (Alternatively, one can also show that $a'_{\mathrm{fs}}$ is automatic and forwards synchronising, by an easy modification of an automaton which computes $a_{\mathrm{fs}}$ reading input from the least significant digit.) 
Since $a_{\mathrm{str}}$ is given by
\[
	a_{\mathrm{str}}'(n) = F\bra{ a_{\mathrm{per}}'(n), a_{\mathrm{fs}}'(n)},
\]
	it follows that $a_{\mathrm{str}}'$ is strongly structured. To see that $a'_{\mathrm{uni}}$ is highly Gowers uniform, we estimate the Gowers norms $\norm{a'_{\mathrm{uni}}}_{U^d[k^L]}$ by covering $[k^L]$ with intervals $P=[k^l, k^{l+1})$ ($0 \leq l < L$) and using Proposition \ref{cor:interval_Ud} to estimate $\norm{a'_{\mathrm{uni}}1_P}_{U^d[k^L]}$.
\end{proof}

\subsection{Uniqueness of decomposition}

The structured automatic sequences we introduce in \eqref{eq:def-of-wk-str} and \eqref{eq:def-of-st-str} are considerably easier to work with than general automatic sequences (cf.{} the proof of Theorem \ref{thm:many-AP-auto} below). However, they are still somewhat complicated and it is natural to ask if they can be replaced with a smaller class in the decompositions in Theorems \ref{thm:main_nobasis} and \ref{thm:main_sc_nobasis}. Equivalently, one can ask if there exist any sequences which are structured in our sense and highly Gowers uniform.

In this section we show that the weakly structured sequences defined in \eqref{eq:def-of-wk-str} are essentially the smallest class of sequences for which Theorem \ref{thm:main_nobasis} is true and that the decomposition in \eqref{eq:003a} is essentially unique. As an application, we derive Theorem \ref{thm:correlations} as an easy consequence of Theorem \ref{thm:main_nobasis}.

\begin{lemma}\label{prop:unique_decomp}
	Let $a \colon \NN_0 \to \CC$ be a weakly structured $k$-automatic sequence such that
	\begin{equation}\label{eq:489:1}
		\lim_{N \to \infty}  \abs{ \EE_{n<N} a(n) b(n) }  = 0
	\end{equation}
	for any periodic sequence $b \colon \NN_0 \to \CC$. Then there exists a constant $c > 0$ such that 
	\begin{equation}\label{eq:489:2}
		\abs{\set{n<N}{{a(n)} \neq 0}} \ll N^{1-c}.
	\end{equation}
\end{lemma}
\begin{proof}
	Since $a$ is weakly structured, we can represent it as
	\begin{equation}\label{eq:489:3}
		a(n) = F\bra{a_{\mathrm{per}}(n), a_{\mathrm{fs}}(n), a_{\mathrm{bs}}(n)},
	\end{equation}
	using the same notation as in  \eqref{eq:def-of-wk-str}. Let $M$ be the period of $a_{\mathrm{per}}$. Pick any residue $r \in \ZZ/M\ZZ$ and synchronising words $\word w, \word v \in \Sigma_k^*$ for $a_{\mathrm{fs}}, a_{\mathrm{bs}}$ respectively. Assume additionally that $\word w$ and $\word v$ do not start with $\texttt 0$. Put $x = a_{\mathrm{per}}(r) \in \Omega_{\mathrm{per}}$, $y = a_{\mathrm{fs}}([\word w]_k)$ and $z = a_{\mathrm{bs}}([\word v]_k)$. Our first goal is to show that $F(x,y,z) = 0$.
	
	Let $P$ be the infinite arithmetic progression
	\begin{equation}\label{eq:489:4}
		P = \set{ n \in \NN_0}{n \bmod M = r \text{ and } (n)_k \in \Sigma_k^* \word w }.
	\end{equation}
	Since $1_P$ is periodic, we have the estimate 
	\begin{equation}\label{eq:489:5}
		\sum_{n=0}^{N-1} a(n) 1_P(n) = \sum_{n=0}^{N-1} F(x,y,a_{\mathrm{bs}}(n)) 1_P(n) = o(N) \text{ as } N \to \infty.
	\end{equation}
	Let $L$ be a large integer an put $N_0 = [\word v]_k k^L$ and $N_1 = ([\word v]_k+1) k^L$. Applying the above estimate \eqref{eq:489:5} with $N = N_0,N_1$ we obtain
	\begin{equation}\label{eq:489:6}
		\sum_{n=N_0}^{N_1-1} a(n) 1_P(n) = \abs{ [N_0,N_1) \cap P} F(x,y,z)  = o(k^L) \text{ as } L \to \infty.
	\end{equation}
	This is only possible if $F(x,y,z) = 0$.
	
	Since $r, \word w, \word v$ were arbitrary, it follows that $a(n) = 0$ if $(n)_k$ is synchronising for both $a_{\mathrm{fs}}$ and $a_{\mathrm{bs}}$. The estimate \eqref{eq:489:2} follows immediately from the estimate on the number of non-synchronising words, discussed in Section \ref{sec:Auto}.	
\end{proof}

\begin{corollary}
\begin{enumerate}[wide]
\item If $a \colon \NN_0 \to \CC$ is both structured and highly Gowers uniform then there exists a constant $c > 0$ such that $\abs{ \set{n < N}{a(n) \neq 0}} \ll N^{1-c}$.
\item If $a = a_{\mathrm{str}} + a_{\mathrm{uni}} = a_{\mathrm{str}}' + a_{\mathrm{uni}}'$ are two decompositions of a sequence $a \colon \NN_0 \to \CC$ as the sum of a weakly structured part and a highly Gowers uniform part then there exists a constant $c > 0$ such that $\set{n < N}{ a_{\mathrm{str}}(n) \neq a_{\mathrm{str}}'(n)} \ll N^{1-c}$.
\end{enumerate}
\end{corollary}

\begin{proof}[Proof of Theorem \ref{thm:correlations} assuming Theorem \ref{thm:main_nobasis}]
	Let $a = a_{\mathrm{str}} + a_{\mathrm{uni}}$ be the decomposition of $a$ as the sum of a weakly structured and a highly Gowers uniform part, whose existence is guaranteed by Theorem \ref{thm:main_nobasis}. Then
	\[
		\limsup_{N \to \infty} \EE_{n < N} \abs{ a_{\mathrm{str}}(n) b(n) } =
		\limsup_{N \to \infty} \EE_{n < N} \abs{ a(n) b(n) } = 0
	\]
	for any periodic sequence $b \colon \NN_0 \to \CC$, for instance by Proposition \ref{cor:interval_Ud}. Hence,  it follows from Lemma \ref{prop:unique_decomp} that there exists $c > 0$ such that $\abs{\set{n < N}{ a_{\mathrm{str}}(n) \neq 0}} \ll N^{1-c}$. In particular, $a_{\mathrm{str}}$ is highly Gowers uniform, and hence so is $a$.  
\end{proof}

\begin{remark}
	Since there exist non-zero weakly structured sequences which vanish almost everywhere, the decomposition in Theorem \ref{thm:main_nobasis} is not quite unique. A prototypical example of such a sequence is the Baum--Sweet sequence $b(n)$, taking the value $1$ if all maximal blocks of zeros in $(n)_2$ have even length and taking the value $0$ otherwise. It seems plausible that with a more careful analysis one could make the decomposition canonical. We do not pursue this issue further. 
\end{remark}

\subsection{Combinatorial application}\label{ssec:Gowers:combi}

In this section we apply Theorem \ref{thm:main_nobasis} to derive a result in additive combinatorics with a more direct appeal, namely Theorem \ref{thm:many-AP-auto}. We will need the following variant of the generalised von Neumann theorem.

\begin{lemma}\label{lem:von-Neumann-skew}
	Fix $d \geq 2$. Let $f_0, f_1, \dots, f_d \colon [N] \to \CC$ be $1$-bounded sequences and let $P \subset [N]$ be an arithmetic progression.
	Then
	$$
		\abs{ \EE_{n,m < N} \prod_{i=0}^d  (1_{[N]} f_i)(n+im) 1_P(m) } \ll \min_{0 \leq i \leq d} \norm{f_i}_{U^d[N]}^{2/3}. 
	$$
\end{lemma}
\begin{proof}
	This is essentially Lemma 4.2 in \cite{GreenTao-2010-ARL}. Using Lemma \ref{lem:smooth_approx} to decompose $1_P$ into a sum of a trigonometric polynomial and an error term small in the $L^1$ norm, for any $\eta > 0$ we obtain the estimate
\begin{align}
	& \abs{ \EE_{n,m < N} \prod_{i=0}^d  (1_{[N]} f_i)(n+im) 1_P(m) }
	 \\ \ll \ & (1/\eta)^{1/2} \sup_{\theta \in \RR} \abs{ \EE_{n,m < N} \prod_{i=0}^d  (1_{[N]} f_i)(n+im) e(\theta m) } + \eta.
	 \label{eq:428}
\end{align}
Given $\theta \in \RR$, put $f_0'(n) = e(-\theta n) f_0(n)$ and $f_1'(n) = e(\theta n) f_1(n)$, and $f_i'(n) = f_i(n)$ for $1 < i \leq d$, so that $\norm{f_i}_{U^d[N]} = \norm{f_i'}_{U^d[N]}$ for all $0 \leq i \leq d$ and 
\[ 
\prod_{i=0}^d  (1_{[N]} f_i)(n+im) e(\theta m) = \prod_{i=0}^d  (1_{[N]} f_i')(n+im) \text{ for all } n,m \in \NN_0.
\]
Applying \cite[Lemma 4.2]{GreenTao-2010-ARL} to $f_i'$ we conclude that
\begin{align}
\sup_{\theta \in \RR} \abs{ \EE_{n,m < N} \prod_{i=0}^d  (1_{[N]} f_i)(n+im) e(\theta m) } \ll  \min_{0 \leq i \leq d} \norm{f_i}_{U^d[N]}.\label{eq:429}
\end{align}
The claim now follows by optimising $\eta$. 
\end{proof}

\begin{proof}[Proof of Theorem \ref{thm:many-AP-auto}]
	Our argument follows a similar basic structure as the proof of Theorem 1.12 in \cite{GreenTao-2010-ARL}, although it is considerably simpler. Throughout the argument, $d = l-1 \geq 2$ and the $k$-automatic set $A \subset \NN_0$ are fixed and all error terms are allowed to depend on $d,k$ and $A$. We also let $N$ denote a large integer and put $L = \ceil{\log_k N}$ and $\alpha = \abs{A \cap [N]}/N$.
	
	 Let $1_A = a_{\mathrm{str}} + a_{\mathrm{uni}}$ be the decomposition given by Theorem \ref{thm:main_nobasis}, and let $c_{1}$ be the constant such that $\norm{a_{\mathrm{uni}}}_{U^d[N]} \ll N^{-c_{1}}$. Let $M$ be the period of the periodic component $a_{\mathrm{per}}$ of $a_{\mathrm{str}}$ and let $\eta > 0$ be a small parameter, to be optimised in the course of the argument. For notational convenience we additionally assume that $\eta L$ is an integer. Consider the arithmetic progression
	 \[
	 	P = \set{n < N}{ n \equiv 0 \bmod{M} \text{ and } (n)_k^L \in 0^{\eta L} \Sigma_k^{L-2\eta L} 0^{\eta L} }.
	 \]
Note $\abs{P}/N \gg N^{-2\eta}$ and that the second condition is just another way of saying that $n \equiv 0 \bmod{k^L}$ and $n/k^L < k^{-\eta L}$. Our general goal is, roughly speaking, to show that many $m \in P$ are common differences of many $(d+1)$-term arithmetic progressions in $A \cap [N]$. Towards this end, we will estimate the average
\begin{align}
	\EE_{m \in P} \EE_{n < N} \prod_{i=0}^d 1_{A \cap [N]}(n+im).
\label{eq:430:1}
\end{align}

Substituting $1_{A \cap [N]} = 1_{[N]}(a_{\mathrm{str}} + a_{\mathrm{uni}})$ into \eqref{eq:430:1} and expanding the product, we obtain the sum of $2^{d+1}$ expressions of the form
\begin{align}
	\EE_{m \in P} \EE_{n < N} \prod_{i=0}^d \bra{ 1_{[N]} a_{i} }(n+im),
\label{eq:430:2}
\end{align}
where $a_i = a_{\mathrm{str}}$ or $a_i = a_{\mathrm{uni}}$ for each $0 \leq i \leq d$. If $a_i = a_{\mathrm{uni}}$ for at least one $i$ then it follows from Lemma \ref{lem:von-Neumann-skew} that
\begin{align}
	\abs{ \EE_{m \in P} \EE_{n < N} \prod_{i=0}^d \bra{ 1_{[N]} a_{i} }(n+im) }
	\ll \frac{N}{\abs{P}} \norm{a_{\mathrm{uni}}}^{2/3} \ll N^{2\eta - 2c_1/3}.
\label{eq:430:3}
\end{align}
	Inserting this into \eqref{eq:430:1} we conclude that we may replace the function $1_{A \cap [N]}$ under the average with $1_{[N]}a_{\mathrm{str}}$ at the cost of introducing a small error term:
\begin{align}
	\EE_{m \in P} \EE_{n < N} \prod_{i=0}^d 1_{A \cap [N]}(n+im)=
	\EE_{m \in P} \EE_{n < N} \prod_{i=0}^d \bra{ 1_{[N]}a_{\mathrm{str}}}(n+im) + O(N^{2\eta - 2c_1/3}).
	\label{eq:430:4}
\end{align}	

Next, we will replace each of the terms $(1_{[N]}a_{\mathrm{str}})(n+im)$ with $(1_{[N]}a_{\mathrm{str}})(n)$ at the cost of introducing another error term. If $(1_{[N]}a_{\mathrm{str}})(n+im) \neq (1_{[N]}a_{\mathrm{str}})(n)$ for some $0 \leq i \leq d$, $m \in P$ and $n \in [N]$ then at least one of the following holds:
\begin{enumerate}[wide]
	\item\label{it:494:A} either the words $(n+i m)_k^L$ and $(n)_k^L$ differ at one of the first $\eta L/2$ positions or $n < N \leq n+im$;
	\item\label{it:494:B} the first $\eta L/2$ digits of $(n)_k^L$ do not contain a synchronising word for the backward synchronising component $a_{\mathrm{bs}}$ of $a_{\mathrm{str}}$;
	\item\label{it:494:C} the last $\eta L$ digits of $(n)_k^L$ do not contain a synchronising word for the forward synchronising $a_{\mathrm{fs}}$ component of $a_{\mathrm{str}}$.
	
Indeed, if neither of these conditions held, the first $\eta L/2$  digits of $n$ and $n+im$ would coincide, as would their last  $\eta L$ digits (because $m\in P$ implies that the last $\eta L$ digits of $m$ are zeros), and we would have $a_{\mathrm{per}}(n)=a_{\mathrm{per}}(n+im)$ (because $m\in P$ implies that $m$ is divisible by $M$, the period of $a_{\mathrm{per}}$); moreover, we would have $a_{\mathrm{fs}}(n)=a_{\mathrm{fs}}(n+im)$ (because the common last $\eta L$ digits of $(n)_k^L$ and $(n+im)_k^L$ contain a synchronising word)   and $a_{\mathrm{bs}}(n)=a_{\mathrm{bs}}(n+im)$ (because the common first $\eta L/2$ digits of $(n)_k^L$ and $(n+im)_k^L$ contain a synchronising word). It would then follow that $a_{\mathrm{str}}(n+im) =a_{\mathrm{str}}(n)$. Moreover, the negation of condition \eqref{it:494:A} would guarantee that $1_{[N]}(n)=1_{[N]}(n+im)$, contradicting our assumption $(1_{[N]}a_{\mathrm{str}})(n+im) \neq (1_{[N]}a_{\mathrm{str}})(n)$.
\end{enumerate}
If $m \in P$ and $n \in [N]$ are chosen uniformly at random then \eqref{it:494:A} holds with probability $\ll N^{-\eta/2}$, and there exist constants $c_{\mathrm{bs}}$ and $c_{\mathrm{fs}}$ (dependent on the synchronising words for the respective components of $a_{\mathrm{str}}$) such that \eqref{it:494:B} and \eqref{it:494:C} hold with probabilities $\ll N^{-c_{\mathrm{bs}}\eta}$ and $\ll N^{-c_{\mathrm{fs}}\eta}$ respectively. Letting $c_2 = \min\bra{1/2,c_{\mathrm{bs}},c_{\mathrm{fs}}}$ and using the union bound we conclude that
\begin{align}\label{eq:430:5}
	\EE_{m \in P} \EE_{n < N} \sum_{i=1}^d \ifbra{(1_{[N]}a_{\mathrm{str}})(n+im) \neq (1_{[N]}a_{\mathrm{str}})(n) } \ll N^{-c_2 \eta}.
\end{align}	
Inserting \eqref{eq:430:5} into \eqref{eq:430:4} and removing the average over $P$ we conclude that
\begin{align}\label{eq:430:6}
	\EE_{m \in P} \EE_{n < N} \prod_{i=0}^d 1_{A \cap [N]}(n+im)=
	\EE_{n < N}a_{\mathrm{str}}^{d+1} (n)+ O(N^{2\eta - 2 c_1/3} + N^{-c_2 \eta}).
\end{align}	
The main term in \eqref{eq:430:6} can now be estimated using H\"{o}lder inequality:
\begin{align}\label{eq:430:7}
\EE_{n < N} a_{\mathrm{str}}^{d+1}(n)
\geq \bra{ \EE_{n < N} a_{\mathrm{str}}(n)}^{d+1} 
\geq \alpha^{d+1} - O(N^{-c_1}),
\end{align}
where in the last transition we use the fact that
$$
\EE_{n < N} a_{\mathrm{str}}(n) = \alpha - \EE_{n < N} a_{\mathrm{uni}}(n) = \alpha - O(N^{-c_1}).
$$
Combining \eqref{eq:430:6} and \eqref{eq:430:7} and letting $\eta$ be small enough that $c_2 \eta < \min\bra{2c_1/3 - 2\eta, c_1}$, we obtain the desired bound for the average \eqref{eq:430:1}:
\begin{align}\label{eq:430:6a}
	\EE_{m \in P} \EE_{n < N} \prod_{i=0}^d 1_{A \cap [N]}(n+im)
\geq \alpha^{d+1} - O(N^{-c_2 \eta}),
\end{align}

Finally, applying a reverse Markov's inequality to \eqref{eq:430:6a} we conclude that 
\begin{align}\label{eq:430:8}
	\EE_{m \in P} \ifbra{  \EE_{n < N} \prod_{i=0}^d 1_{A \cap [N]}(n+im) \geq \alpha^{d+1} - \e } \geq \e - O(N^{-c_2 \eta})
\end{align}
for any $\e > 0$. Optimising the value of $\eta$ for a given $\e > 0$ we conclude that there exists $\gg \e^C N$ values of $m$ such that 
\[
 \EE_{n < N} \prod_{i=0}^d 1_{A \cap [N]}(n+im) \geq \alpha^{d+1} - \e,
\]
provided that $\e > N^{-1/C}$ for a certain constant $C > 0$ dependent on $d,k$ and $A$.  When $\e < N^{-1/C}$, it is enough to use $m = 0$.
\end{proof}
\begin{remark}
	The proof is phrased in terms which appear most natural when $\eta$ is a constant and $\e$ is a small power of $N$. This choice is motivated by the fact that this case is the most difficult. However, the theorem is valid for all $\e$ in the range $(N^{-1/C},1)$, including the case when $\e$ is constant as $N \to \infty$.
\end{remark}

\subsection{Alternative line of attack}\label{ssec:plan-B}

In this section we describe an alternative strategy one could try to employ in the proof of our main theorems. Since this approach is possibly more natural, we find it interesting to see where the difficulties arise and to speculate on how this hypothetical argument would differ from the one presented in the remainder of the paper. As the material in this section is not used anywhere else and has purely motivational purpose, we do not include all of the definitions (which the reader can find in \cite{GreenTao-2010-ARL}) nor do we prove all that we claim.

Let $a \colon \NN_0 \to \CC$ be a sequence with $\abs{a(n)} \leq 1$ for all $n \geq 0$. Fix $d \geq 1$ and a small positive constant $\e > 0$ and let also $\mathcal{F} \colon \RR_{>0} \to \RR_{>0}$ denote a rapidly increasing sequence (its meaning will become apparent in the course of the reasoning). The Arithmetic Regularity Lemma \cite{GreenTao-2010-ARL} ensures that for each $N > 0$ there exists a parameter $M = O(1)$ (allowed to depend on $d,\e,\cF$ but not on $N$) and a decomposition
\begin{equation}\label{eq:regularity}
a(n) = a_{\mathrm{str}}(n) + a_{\mathrm{sml}}(n) + a_{\mathrm{uni}}(n), \qquad (n \in [N]),
\end{equation}
where $a_{\mathrm{str}}$, $a_{\mathrm{sml}}$ and $a_{\mathrm{uni}} \colon [N] \to \CC$ are respectively structured, small and uniform in the following sense:
\begin{itemize}[wide]
\item	$a_{\mathrm{str}}(n) = F(g(n)\Gamma, n \bmod Q, n/N)$ where $F$ is a function with Lipschitz norm $\leq M$, $Q$ is an integer with $1 \leq Q \leq M$, $g \colon \NN_0 \to G/\Gamma$ is a $(\cF(M),N)$-irrational polynomial sequence of degree $\leq d-1$ and complexity $\leq M$, taking values in a nilmanifold $G/\Gamma$;
\item $\norm{a_{\mathrm{sml}}}_{L^2[N]} \leq \e$;
\item $\norm{a_{\mathrm{uni}}}_{U^d[N]} \leq 1/\cF(M)$.
\end{itemize}
Note that $\cF$ can always be replaced with a more rapidly increasing function and that definitions of many terms related to $a_{\mathrm{str}}$ are currently not provided. The decomposition depends on $N$, but for now we let $N$ denote a large integer and keep this dependence implicit.

Suppose now that $a$ is additionally $k$-automatic. We can use the finiteness of the kernel of $a$ to find $\alpha \geq 0$ and $0 \leq r < s < k^\alpha$ such that $a(k^\alpha n + r) = a(k^\alpha n + s)$ for all $n \geq 0$. For the sake of simplicity, suppose that a stronger condition holds: for each $q \in \ZZ/Q\ZZ$, there exist $0 \leq r < s < k^\alpha$ as above with $r \equiv s \equiv q \pmod{Q}$. Define also $N' = N/k^\alpha$ and $b_{\mathrm{str}}(n) = a_{\mathrm{str}}(k^\alpha n + s) - a_{\mathrm{str}}(k^\alpha n + r)$ for all $n \in [N']$, and accordingly for $b_{\mathrm{sml}}$ and $b_{\mathrm{uni}}$. Then $b_{\mathrm{str}} + b_{\mathrm{sml}} + b_{\mathrm{uni}} = 0$. In particular, 
\begin{align*}
	\EE_{n < N'} \abs{b_{\mathrm{str}}(n)}^2 \leq \abs{\EE_{n < N'} b_{\mathrm{sml}}(n) \bar b_{\mathrm{str}}(n)}
	+ \abs{ \EE_{n < N'} b_{\mathrm{uni}}(n) \bar b_{\mathrm{str}}(n) }.
\end{align*}
The first summand is $O(\e)$ by Cauchy-Schwarz. It follows from the Direct Theorem for Gowers norms that, as long as $\cF$ increases fast enough (the required rate depends on $\e$), the second summand is $\leq \e$. Hence,
\begin{equation}\label{eq:444:1}
	\EE_{n < N'} \abs{a_{\mathrm{str}}(k^\alpha n + r) - a_{\mathrm{str}}(k^\alpha n + s)}^2 = \EE_{n < N'} \abs{b_{\mathrm{str}}(n)}^2 = O(\e).
\end{equation}
Bearing in mind that $k^\alpha n + r$ and $k^\alpha n + s$ differ by a multiple of $Q$ which is small compared to $N$, one can hope to derive from \eqref{eq:444:1} that for each $q \in \ZZ/Q\ZZ$ and each $t \in [0,1]$,
\begin{equation}\label{eq:444:2-1}
F(g(k^\alpha n + r)\Gamma, q, t) \approx F(g(k^\alpha n + s)\Gamma, q, t), \qquad (n \in [N']).
\end{equation}
(We intentionally leave vague the meaning of the symbol ``$\approx$''.) From here, it is likely that one could show that $F(x,n,t)$ is essentially constant with respect to $x \in G/\Gamma$. This could possibly be achieved by a more sophisticated variant of the argument proving Theorem B in \cite{ByszewskiKonieczny-2019}. For the sake of exposition, let us rather optimistically suppose that $F(x,n,t) = F(n,t)$ is entirely independent of $x$.

We are then left with the structured part taking the form $a_{\mathrm{str}}(n) = F(n \bmod Q, n/N)$, which bears a striking similarity to the definition of a weakly structured automatic sequence. Unfortunately, there is no guarantee that $a_{\mathrm{str}}$ produced by the Arithmetic Regularity Lemma is $k$-automatic (or that it can be approximated with a $k$-automatic sequence in an appropriate sense). Ensuring $k$-automaticity of $a_{\mathrm{str}}$ seems to be a major source of difficulty. We note that \eqref{eq:444:1} can be construed as approximate equality between $a_{\mathrm{str}}(k^\alpha n + r)$ and $a_{\mathrm{str}}(k^\alpha n + s)$, which suggests (but does not prove) that $a_{\mathrm{str}}$ should be approximately equal to a $k$-automatic sequence $a_{\mathrm{str}}'$.

If the line of reasoning outlined above succeeded, it would allow us to decompose an arbitrary automatic sequence as the sum of a weakly structured automatic sequence and an error term, which is small in an appropriate sense. However, it seems rather unlikely that this reasoning could give better bounds on the error terms than the rather poor bounds provided by the Arithmetic Regularity Lemma. Hence, in order to obtain the power saving, we are forced to argue along similar lines as in Section \ref{sec:Recursive}. It is also worth noting that while the decomposition produced by our argument can be made explicit, it is not clear how to extract an explicit decomposition from an approach using the Arithmetic Regularity Lemma.
Finally, our approach also ensures that the uniform component fulfills the carry Property, which is essential to the possible applications discussed in Section~\ref{sec_1}, and which would be completely lost with the use of the Arithmetic Regularity Lemma.

{

\section{\geas{}}\label{sec:Trans}

\subsection{Definitions}\label{ssec:Trans:defs}
In order to deal with automatic sequences more efficiently, we introduce the notion of a \gea{}.\footnote{This construction was called a (naturally induced) transducer in~\cite{Mullner-2017}, but this name seems better suited here. One main motivation for this name is the fact that this construction corresponds to a group extension for the related dynamical systems, as was shown in~\cite{Lemanczyk2018}.}
  A \emph{\geka{} without output} (\gekaab{}) is a sextuple $\cT = (S,s_0,\Sigma_k,\delta,G,\lambda)$ consisting of the following data:
  \begin{itemize}
  \item a finite set of states $S$ with a distinguished initial state $s_0$; 
  \item a transition function $\delta\colon S \times \Sigma_k \to S$;
  \item a labelling $\lambda \colon S \times \Sigma_k \to G$ where $(G, \cdot)$ is a finite group.
  \end{itemize} 

Note that $\cT$ contains the data defining an automaton $(S,s_0,\Sigma_k,\delta)$ without output and additionally associates group labels to each transition. Recall that the transition function $\delta$ extends naturally to a map (denoted by the same letter) $\delta\colon S \times \Sigma_k^* \to S$ such that 
      $\delta( s, \word v \word u) = \delta( \delta(s,\word v), \word u)$ for all $\word u, \word v \in \Sigma_k^*$. The labelling function similarly extends to a map $\lambda \colon S \times \Sigma_k^* \to G$ such that $\lambda(s,\word v \word u) = \lambda(s,\word v) \cdot \lambda(\delta(s,\word v),\word u)$ for all $\word u, \word v \in \Sigma_k^*$. Thus, $\cT$ can be construed as a means to relate a word $\word w \in \Sigma_k^*$ to a pair consisting of the state $\delta(s_0,\word w)$ and the group element $\lambda(s_0,\word w)$.
 
  
  
  A \emph{\geka{} with output} (\gekaoab{}) $\cT = (S,s_0,\Sigma_k,\delta,G,\lambda,\Omega,\tau)$ additionally includes
  \begin{itemize}
   \item an output function $\tau \colon S \times G \to \Omega$, where $\Omega$ is a finite set.
  \end{itemize}
   We use the term \gea{} (\geaab{}) to refer to a \geka{} where $k$ is left unspecified. The term \geao{} (\geaoab{}) is used accordingly.

Let $\cT = (S,s_0,\Sigma_k,\delta,G,\lambda,\Omega,\tau)$ be a \gekao{}. Then $\cT$ produces the $k$-automatic map $a_{\cT}\colon \Sigma_k^* \to \Omega$ given by 
	\begin{align}\label{eq:def-of-a_T}
		a_{\cT}(\word u) = \tau\bra{\delta(s_0,\word u), \lambda(s_0,\word u)},
	\end{align}
	which in particular gives rise to the $k$-automatic sequence (denoted by the same symbol) $a_{\cT}\colon \NN_0 \to \Omega$ via the natural inclusion $\NN_0 \hookrightarrow \Sigma_k^*$, $n \mapsto (n)_k$. Accordingly, we say that the \geaab{} $\cT$ produces a sequence $a \colon \NN_0 \to \Omega$ if there exists a choice of the output function $\tau$ such that $a = a_{\cT}$. More generally, to a pair $(s,h) \in S \times G$ we associate the $k$-automatic sequence
	\begin{align}\label{eq:def-of-a_T-2}
	a_{\cT,s,h}(\word u) = \tau\bra{ \delta(s,\word u), h\cdot \lambda(s,\word u)}.
	\end{align}
If the \geaab{} $\cT$ is clear from the context, we omit it in the subscript. Note that with this terminology, \geaab{}s read input starting with the most significant digit. We could also define analogous concepts where the input is read from the least significant digit, but these will not play a role in our reasoning.

A {morhphism} from $\cT$ to another \gekaab{} $\cT' = (S',s_0',\Sigma_k,\delta',G',\lambda')$ without output is a pair $(\phi, \pi)$
  where $\phi \colon S \to S'$ is a map and $\pi \colon G \to G'$ is a morphism of groups obeying the following compatibility conditions: 
  \begin{itemize}
    \item $\phi(s_0) = s_0'$ and $\delta'( \phi(s), j) = \phi( \delta(s,j) )$ for all $s \in S,\ j \in \Sigma_k$;
    \item $\lambda'( \phi(s), j) = \pi( \lambda(s,j) )$ for all $s \in S,\ j \in \Sigma_k$.
  \end{itemize}
  If $\phi$ and $\pi$ are surjective, we will say that $\cT'$ is a factor of $\cT$.
A morphism from $\cT$ to another \geka{} with output $\cT' = (S',s_0',\Sigma_k,\delta',G',\lambda',\Omega',\tau')$ is a triple $(\phi,\pi,\sigma)$ where $(\phi,\pi)$ is a morphism from $\cT_0 = (S,s_0,\Sigma_k,\delta,G,\lambda)$ to $\cT_0' = (S',s_0',\Sigma_k,\delta',G',\lambda')$ and $\sigma \colon \Omega \to \Omega'$ is compatible with $(\phi,\pi)$ in the sense that
  \begin{itemize}    
    \item $\tau'( \phi(s), \pi(g) ) = \sigma(\tau(s, g))$ for all $s \in S$, $g \in G$.
  \end{itemize}
In the situation above the sequence $a_{\cT'}$ produced by $\cT'$ is a coding of the sequence $a_{\cT}$ produced by $\cT$, that is, $a_{\cT'}(n) = \sigma \circ a_{\cT} (n)$.

	We say that a \geaab{} $\cT$ (with or without output) is \emph{strongly connected} if the underlying automaton without output $\cA = (S,s_0,\Sigma_k,\delta)$ is strongly connected. The situation is slightly more complicated for synchronisation. We say that a word $\word w \in \Sigma_k^*$ \emph{synchronises} $\cT$ to a state $s \in S$ if $\delta(s',\word w) = s$ and $\lambda(s',\word w) = \id_G$ for each $s' \in S$, and that $\cT$ is \emph{synchronising} if it has a word that synchronises it to the state $s_0$.\footnote{It is not common to require a synchronizing word to a specific state, but this will not be a serious restriction for this paper.} (This is different than terminology used in \cite{Mullner-2017}.) Note that if $\cT$ is synchronising then so is the underlying automaton but not vice versa, and that even if $\cT$ is strongly connected and synchronising there is no guarantee that all states $s \in S$ have a synchronising word. We also say that $\cT$ (or $\cT$) is \emph{prolongable} if $\delta(s_0,\texttt 0) = s_0$ and $\lambda(s_0,\texttt 0) = \id_G$. Finally, $\cT$ is \emph{idempotent} if it ignores the leading zeros and $\delta(s,\texttt 0) = \delta(s,\texttt{00})$ and $\lambda(s, \texttt{00}) = \lambda(s,\texttt 0)$ for all $s \in S$.

As alluded to above, the sequence $a_{\cT}$ produced by the \geaoab{} $\cT$ is $k$-automatic. More explicitly, the \geaoab{} $\cT = (S,s_0,\Sigma_k,\delta,G,\lambda,\Omega,\tau)$ gives rise to the automaton $\cA_{\cT} = (S',s_0',\Sigma_k,\delta',\Omega,\tau)$ where $S' = S \times G$, $s_0' = (s_0,\id_G)$ and $\delta'((s,g),j) = (\delta(s,j),g \cdot \lambda(s,j))$. Conversely, any automaton $\cA = (S,s_0,\Sigma_k,\delta,\Omega,\tau)$ can be identified with a \geaoab{} $\cT_{\cA} = (S,s_0,\Sigma_k,\delta,\{\id\},\lambda_{\id},\Omega,\tau')$ with trivial group, $\lambda_{\id}(s,j) = \id$ and $\tau'(s,\id) = \tau(s)$. At the opposite extreme, any invertible automaton $\cA$ can be identified with a \geaoab{} $\cT_{\cA}^{\mathrm{inv}} = (\{s_0'\},s_0',\Sigma_k,\delta_0',\Sym(S),\lambda,\Omega,\tau')$ with trivial state set where $\delta_0'(s_0',j) = s_0'$, $\lambda(s_0',j) = \delta(\cdot , j)$ and $\tau'(s_0', g) = \tau(g(s_0))$. Accordingly, we will call any \geaoab{} (or \geaab{}) with a single state \emph{invertible} and we omit the state set from its description: any invertible \geaoab{} is fully described by the data $(G, \lambda, \Omega, \tau)$.

\begin{example}\label{ex:Rudin-Shapiro}
The Rudin--Shapiro sequence $r(n)$ is given recursively by $r(0) = +1$ and $r(2n) = r(n)$, $r(2n+1) = (-1)^n r(n)$. It is produced by the following $2$-automaton:
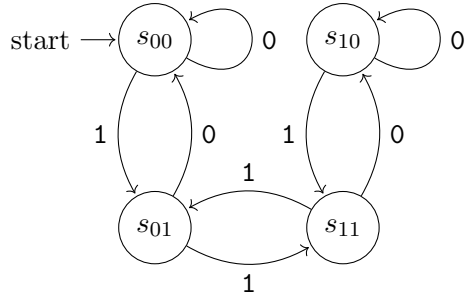
\begin{center}
\begin{tikzpicture}[shorten >=1pt,node distance=2.5cm, on grid, auto] 
   \node[initial, state] (s_00)   {$s_{00}$}; 
   \node[state] (s_01) [below=of s_00] {$s_{01}$}; 
   \node[state] (s_11) [right=of s_01] {$s_{11}$}; 
   \node[state] (s_10) [above=of s_11] {$s_{10}$}; 
  \tikzstyle{loop}=[min distance=6mm,in=210,out=150,looseness=7]

    \path[->]     
    (s_01) edge [bend right] node [below]  {\texttt 1} (s_11);

 \tikzstyle{loop}=[min distance=4mm,in=120,out=240,looseness=1]
    \path[->]     
    (s_00) edge [loop left] node {\texttt 1} (s_01);
    \path[->]     
    (s_10) edge [loop left] node {\texttt 1} (s_11);

 \tikzstyle{loop}=[min distance=4mm,in=-60,out=60,looseness=1]
    \path[->]     
    (s_01) edge [loop right] node {\texttt 0} (s_00);
    \path[->]     
    (s_11) edge [loop right] node {\texttt 0} (s_10);
          
 \tikzstyle{loop}=[min distance=6mm,in=30,out=-30,looseness=7]
	\path[->]
    (s_11) edge [bend right] node [above]  {\texttt 1} (s_01);
    
	\path[->]
     (s_10) edge [loop right] node  {\texttt 0} (s_10);
     
   \path[->] 
    (s_00) edge [loop right] node {\texttt 0} (s_00);

\end{tikzpicture}
\end{center}
where $s_{00}$ is the initial state, an edge labelled $j$ from $s$ to $s'$ is present if $\delta(s,j) = s'$ and the output function is given by $\tau(s_{00}) = \tau(s_{01}) = +1$ and $\tau(s_{10}) = \tau(s_{11}) = -1$. Alternatively, $r$ is produced by the \geaoab{} with group $G = \{+1,-1\}$, given by
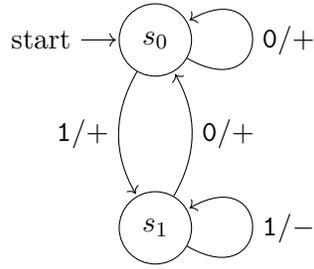
\begin{center}
\begin{tikzpicture}[shorten >=1pt,node distance=2.5cm, on grid, auto] 
   \node[initial, state] (s_0)   {$s_{0}$}; 
   \node[state] (s_1) [below=of s_0] {$s_{1}$}; 
 \tikzstyle{loop}=[min distance=4mm,in=120,out=240,looseness=1]
    \path[->]     
    (s_0) edge [loop left] node {\texttt 1/$+$} (s_1);

 \tikzstyle{loop}=[min distance=4mm,in=-60,out=60,looseness=1]
    \path[->]     
    (s_1) edge [loop right] node {\texttt 0/$+$} (s_0);

  \tikzstyle{loop}=[min distance=6mm,in=30,out=-30,looseness=7]        
   \path[->] 
    (s_0) edge [loop right] node {\texttt 0/$+$} (s_0);
   \path[->]     
    (s_1) edge [loop right] node  {\texttt 1/$-$} (s_1);
\end{tikzpicture}
\end{center}
where $s_0$ is the initial state, edge labelled $j/\pm$ from $s$ to $s'$ is present if $\delta(s,j) = s'$ and $\lambda(s,j) = \pm 1$, and the output function is given by $\tau(s,g) = g$. This is an example of an \egeaoab{}, which will be defined shortly.
\end{example}

\begin{example}\label{ex:main-1b}
	Recall the sequence $a(n)$ defined in Example \ref{ex:main-1}. It is produced by the \geaoab{} with group $G = \{+1,-1\}$, given by
	
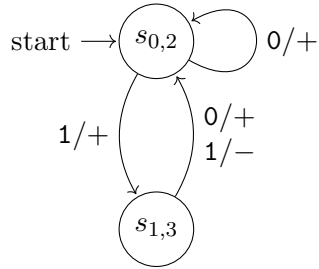
\begin{center}
\begin{tikzpicture}[shorten >=1pt,node distance=2.5cm, on grid, auto] 
   \node[initial, state] (s_0)   {$s_{0,2}$}; 
   \node[state] (s_1) [below=of s_0] {$s_{1,3}$}; 
 \tikzstyle{loop}=[min distance=4mm,in=120,out=240,looseness=1]
    \path[->]     
    (s_0) edge [loop left] node {\texttt 1/$+$} (s_1);

 \tikzstyle{loop}=[min distance=4mm,in=-60,out=60,looseness=1]
    \path[->]     
    (s_1) edge [loop right] node [ text width=1cm] {\texttt 0/$+$ \texttt 1/$-$} (s_0);

    
  \tikzstyle{loop}=[min distance=6mm,in=30,out=-30,looseness=7]        
   \path[->] 
    (s_0) edge [loop right] node {\texttt 0/$+$} (s_0);

\end{tikzpicture}
\end{center}
where we use the same conventions as in Example \ref{ex:Rudin-Shapiro} above and  the output is  
\begin{align*}
&&
	\tau(s_{0,2}, +1) &= 4,&&&
	\tau(s_{0,2}, -1) &= 2, &&& \\ &&
	\tau(s_{1,3}, +1) &= 1,&&& 
	\tau(s_{1,3}, -1) &= 1.&&& \\	
\end{align*}
	
\end{example}

\begin{example}\label{ex:main-2b}
	We also present a \geaoab{} that produces the sequence $a(n)$ defined in Example~\ref{ex:main-2}.
	The group is given by the symmetric group on $3$ elements $\Sym(3)$, where we use the cyclic notation to denote the permutations.
	\begin{center}
			\tikzset{elliptic state/.style={draw,ellipse}}
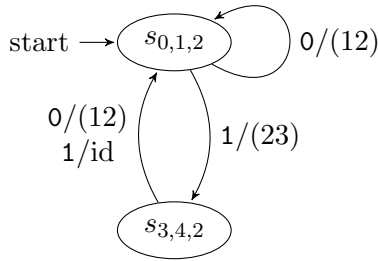
  \begin{tikzpicture}[->,>=stealth',shorten >=1pt,auto,node distance=2.5cm, on grid,]
    
    \node[elliptic state, initial]	(A)                    {$s_{0,1,2}$};
    \node[elliptic state]	         (B) [below of=A] 	{$s_{3,4,2}$};
    
    \path [pos = 0.5]
    (A) edge [loop right, min distance = 6mm, looseness = 7, in = 30, out = -30] node 		 	{\texttt{0}/(12)}      	(A)
	edge [bend left]  node 		 	{\texttt{1}/(23)} 	(B)
    (B) edge [bend left]  node [align=center]	{\texttt{0}/(12)\\\texttt{1}/id}	(A);
  \end{tikzpicture}
  \end{center}
  The output is given by
  \begin{align*}
&&
	\tau(s_{0,1,2}, \id) &= \tau(s_{0,1,2}, (23)) = 1,&&&
	\tau(s_{3,4,2}, \id) &= \tau(s_{3,4,2}, (23)) = 4, &&& \\ &&
	\tau(s_{0,1,2}, (12)) &= \tau(s_{0,1,2}, (132)) = 2,&&& 
	\tau(s_{3,4,2}, (12)) &= \tau(s_{3,4,2}, (132)) = 5,&&& \\ &&	
	\tau(s_{0,1,2}, (13)) &= \tau(s_{0,1,2}, (123)) = 3,&&& 
	\tau(s_{3,4,2}, (13)) &= \tau(s_{3,4,2}, (123)) = 3.&&& \\	
\end{align*}
\end{example}

\subsection{Efficient \geas{}}\label{ssec:Trans:natural}

As we have seen, all sequences produced by \geaoab{}s are automatic and conversely any automatic sequence is produced by a \geaoab{}. In \cite{Mullner-2017} it is shown that any sequence can be produced by an especially well-behaved \geaoab{}. We will now review the key points of the construction in \cite{Mullner-2017} and refer to that paper for more details. For the convenience of the Reader, we add the notation used in \cite{Mullner-2017} in square brackets.

Let $\cA = (S, s_0, \Sigma_{k}, \delta,\Omega, \tau)$ [$A = (S', s_0', \Sigma_{k}, \delta', \tau')$] be an idempotent $k$-automaton. Let $m$ [$n_0$] be the smallest possible cardinality of a set $\set{ \delta(s,\word w)}{ s \in S} $ with $\word w \in \Sigma_k^*$. The states of the \geaoab{} $\hat S \subset S^m$ [$S \subset (S')^{n_0}$]  consist of ordered $m$-tuples of distinct states $\hat s = (s_1,s_2,\dots,s_m)$ of $\cA$, no two of which contain the same set of entries. The transition function is defined by the condition that for $\hat s = (s_1,\dots,s_m) \in \hat S$ and $j \in \Sigma_k$ the entries of $\hat\delta(\hat s, j)$ are, up to rearrangement, $\delta(s_1,j),\dots,\delta(s_m,j)$. The initial state is any $m$-tuple $\hat s_0 = (s_{0,1}, \dots, s_{0,m}) \in \hat S$ with $s_{0,1} = s_0$. The group $G$ [$\Delta$] consists of permutations of $\{1,2,\dots,m\}$, $G \subset \mathrm{Sym}(m)$. The group labels are chosen so that for $\hat s = (s_1,\dots,s_m) \in \hat S$ and $j \in \Sigma_k$ the label $g = \lambda(\hat s, j)$ is the unique permutation such that
\[
	\hat\delta(\hat s,j) = \bra{ \delta(s_{g(1)},j),\dots,\delta(s_{g(m)},j) }.
\]
Hence, $\delta(s_1,j),\dots,\delta(s_m,j)$ can be recovered by permuting the entries of $\hat\delta(\hat s,j)$ according to $\lambda(\hat s, j)$ \cite[Lem.\ 2.4]{Mullner-2017}. More generally, for all $\word u \in \Sigma_k^*$ we have
\[
\bra{ \delta(s_1, \word u),\dots,\delta(s_m, \word u)} = \lambda(\hat s, \word u) \cdot \hat\delta(\hat s, \word u),
\]
where $\Sym(m)$ acts on $\hat S$ by $g \cdot (s_1,\dots,s_m) = (s_{g^{-1}(1)}, \dots, s_{g^{-1}(m)})$. Finally, for $\hat s \in \hat S$ and $g \in G$ we set $\hat \tau(\hat s,g) = \tau\bra{\operatorname{pr_1}\bra{g \cdot \hat s}}$, where $\operatorname{pr_1}$ denotes the projection onto the first coordinate. Put $\cT = \cT_{\cA} := (\hat S, \hat s_0, \Sigma_k, \hat\delta, G, \lambda, \Omega, \hat\tau)$. Then the construction discussed so far guarantees that $a_{\cA} = a_{\cT}$ \cite[Prop.\ 2.5]{Mullner-2017} and also that $\cT$ is strongly connected and that the underlying automaton of $\cT$ is synchronising \cite[Prop.\ 2.2]{Mullner-2017}.

The \geaoab{} $\cT$ is essentially unique with respect to the properties mentioned above, except for two important degrees of freedom: we may rearrange the elements of the $m$-tuples in $\hat S$ and we may change $\hat s_0$ to any other state beginning with $s_0$. Let $S_0$ denote the image of $\delta(\cdot,\texttt 0)$ and let $\hat S_0 \subset S_0^m$ denote the image of $\hat\delta(\cdot,\texttt 0)$. The assumption that $\cA$ is idempotent guarantees that for each $\hat s \in \hat S_0$ we have $\hat\delta(\hat s,\texttt 0) = \hat s$ and $\lambda(\hat s,\texttt 0) = \id$. It follows that we may choose $\hat s_0 \in \hat S_0$, so that $\cT$ ignores the leading zeros, i.e. it is prolongable. Consequently, we may assume that $\cT$ is idempotent.

Rearranging the $m$-tuples in $\hat S$ corresponds to replacing the labels $\lambda(\hat s, j)$ ($\hat s \in \hat S,\ j \in \Sigma_k$) with conjugated labels $\lambda'(h(\hat s), j) = h(\hat s) \lambda(\hat s, j) h(\hat\delta(\hat s,j))^{-1}$ for any $h \colon \hat S \to \Sym(m)$ (to retain $\hat s_0$ as a valid initial state, we also need to guarantee that $h(\hat s_0)(1) = 1$). More generally, for $\word u \in \Sigma_k^*$ we have $\lambda'(h(\hat s), \word u) = h(\hat s) \lambda(\hat s, \word u) h(\hat\delta(\hat s,\word u))^{-1}$ \cite[Prop.\ 2.6]{Mullner-2017}.
To avoid redundancies, we always assume that the group $G$ is the subgroup of $\Sym(m)$ generated by all of the labels $\lambda(\hat s,j)$ ($\hat s \in \hat S$, $j \in \Sigma_k$); such conjugation may allow us to replace $G$ with a smaller group. In fact, we may ensure a minimality property \cite[Thm.\ 2.7 + Cor.\ 2.26]{Mullner-2017}: 
\begin{enumerate}[label={$(\hat{\mathtt{T}}_{\arabic*})$},ref={$\hat{\mathtt{T}}_{\arabic*}$}]
\item\label{item:74hatB} For any $\hat s,\hat s' \in \hat S$ and sufficiently large $l \in \NN$ we have
      \begin{align*}
        \set{ \lambda(\hat s,\word w) }{ \word w \in \Sigma^l_k,\ \hat \delta(\hat s,\word w) = \hat s'} = G.
      \end{align*}
\end{enumerate}
This property is preserved by any further conjugations, as long as we restrict to $h \colon \hat S \to G$.

The condition \ref{item:74hatB} guarantees that all elements of $G$ appear as labels attached to paths between any two states. It is natural to ask what happens if additional restrictions are imposed on the integer $[\word w]_k$ corresponding to a path. The remainder of $[\word w]_k$ modulo $k^l$ ($l \in \NN$) records the terminal $l$ entries of $\word w$ and hence is of limited interest. We will instead be concerned with the remainder of $[\word w]_k$ modulo integers coprime to $k$. This motivates us to let $\gcd_k^*(A)$ denote the greatest among the common divisors of a set $A \subset \NN_0$ which are coprime to $k$ and put (following nomenclature from \cite{Mullner-2017})
\begin{equation}
	d' = d'_{\cT} = \gcd{}_k^* \set{[\word w]_k}{\word w \in \Sigma_k^*,\ \hat\delta(\hat s_0,\word w) = \hat s_0,\ \lambda(\hat s,\word w) = \id}.
\end{equation}
After applying further conjugations, we can find a normal subgroup $G_0 < G$ together with a group element $g_0 \in G$ such that \cite[Thm.\ 2.16 + Cor.\ 2.26]{Mullner-2017}: 
  \begin{enumerate}[label={$(\hat{\mathtt{T}}_{\arabic*})$},ref={$\hat{\mathtt{T}}_{\arabic*}$}] \setcounter{enumi}{1}
    \item\label{item:74hatC} For any $\hat s, \hat s' \in \hat S$ and $0 \leq r < d'$ it holds that
          \begin{align*}
            \set{\lambda(\hat s,\word w)}{  \word w \in \Sigma_k^*,\ \delta(\hat s,\word w) = \hat s',\  [\word w]_k \equiv r \bmod d'} = G_0 g_0^r = g_0^r G_0.
          \end{align*}
        \item\label{item:74hatD} For any $\hat s, \hat s'\in \hat S$, any $g \in G_0$ and any sufficiently large $l \in \NN$ it holds that
          \begin{align*}
             \gcd{}_k^* \set{[\word w]_k}{\word w \in \Sigma^l_k,\ \hat\delta(\hat s,\word w) = \hat s',\ \lambda(\hat s,\word w) = g} = d'.
          \end{align*}
  \end{enumerate}
  The properties listed above imply in particular that $G/G_0$ is a cyclic group of order $d'$ generated by $g_0$. We also mention that \cite{Mullner-2017} has a somewhat stronger variant of \ref{item:74hatD} which is not needed for our purposes.

Let $\word w$ be a word synchronising the underlying automaton of $\cT$ to $\hat s_0$. Prolonging $\word w$ if necessary we may assume without loss of generality that $d' \mid [\word w]_k$ and that $\word w$ begins with $\mathtt 0$. Repeating $\word w$ if necessary we may further assume that $\lambda(\hat s_0, \word w) = \id$. Conjugating by $h(\hat s) = \lambda^{-1}(\hat s, \word w) \in G_0$ we may finally assume that $\lambda(\hat s, \word w) = \id$ for all $\hat s \in \hat S$, and hence that the \geaoab{} $\cT$ is synchronising. Note that thanks to idempotence, for each $\hat s \in S$ we have $\lambda(\hat s, \mathtt 0) = \lambda(\hat s, \mathtt 0 \word w) = \lambda(\hat s, \word w) = \id_G$.

In broader generality, let us say that a \geaoab{} $\cT = (S, s_0, \Sigma_k, \delta, G, \lambda, \Omega,\tau)$ (not necessarily arising from the construction discussed above) is \emph{efficient} if it is strongly connected, idempotent, synchronising, $\lambda(s,\mathtt 0) = \id_G$ for all $s \in S$ and it satisfies the ``unhatted'' versions of the properties \ref{item:74hatB}, \ref{item:74hatC} and \ref{item:74hatD}, that is, there exist $d' = d'_{\cT}$, $g_0 \in G$ and $G_0 < G$ such that
\begin{enumerate}[label={$( {\mathtt{T}}_{\arabic*})$},ref={$ {\mathtt{T}}_{\arabic*}$}]
\item\label{item:74B} For any $ s, s' \in  S$ and sufficiently large $l \in \NN$ we have
   \begin{align*}
    \set{ \lambda(s,\word w) }{ \word w \in \Sigma^l_k,\  \delta( s,\word w) =  s'} = G.
   \end{align*}
  \item\label{item:74C} For any $ s,  s' \in  S$ and $0 \leq r < d'$ it holds that
     \begin{align*}
      \set{\lambda( s,\word w)}{ \word w \in \Sigma_k^*,\ \delta( s,\word w) =  s',\ [\word w]_k \equiv r \bmod d'} = G_0 g_0^r = g_0^r G_0.
     \end{align*}
    \item\label{item:74D} For any $ s,  s'\in  S$, any $g \in G_0$ and any sufficiently large $l \in \NN$ it holds that
     \begin{align*}
       \gcd{}_k^* \set{[\word w]_k}{\word w \in \Sigma^l_k,\ \delta( s,\word w) = s',\ \lambda( s,\word w) = g} = d'.
     \end{align*}
 \end{enumerate}
We let $\word w^{\cT}_0$ denote a synchoronising word for $\cT$.

{
 The above discussion can be summarised by the following theorem. We note that this theorem is essentially contained in \cite{Mullner-2017}, except for some of the reductions presented here. Additionally, \cite{Mullner-2017} contains a slightly stronger version of property \ref{item:74C} where $\word w$ is restricted to $\Sigma_k^l$ for large $l$, which can be derived from properties \ref{item:74B} and \ref{item:74C}.
} 

\begin{theorem}\label{thm:transducers}
	Let $\cA$ be a strongly connected idempotent automaton. Then there exists an \egeaoab{} $\cT$ which produces the same sequence: $a_{\cA} = a_{\cT}$.
\end{theorem}
}


In analogy with Proposition \ref{prop:wlog_change_s0}, the veracity of  Theorem \ref{thm:main_sc_nobasis} is independent of the initial state of the \geao{}.

\begin{proposition}\label{prop:wlog-change-s_0-T}
	Let $\cT = (S,s_0,\Sigma_k,\delta,G,\lambda,\Omega,\tau)$ be an \egeaoab{} and let $S_0 \subset S$ denote the set of all states  $s \in S$ such that $\delta(s,\texttt 0) = s$ and $\lambda(s,\texttt 0) = \id_G$. Then the following conditions are equivalent.
	\begin{enumerate}
	\item\label{item:44A} Theorem \ref{thm:main_sc_nobasis} holds for $a_{\cT,s,h}$ for some $s \in S_0,\ h \in G$;
	\item\label{item:44B} Theorem \ref{thm:main_sc_nobasis} holds for $a_{\cT,s,h}$ for all $s \in S_0,\ h \in G$;
	\end{enumerate}
\end{proposition}
\begin{proof}
	Assume without loss of generality that Theorem \ref{thm:main_sc_nobasis} holds for $a_{\cT}$, and let $s \in S,\ h \in G$. It follows from condition \ref{item:74B} there exists $\word u \in \Sigma_k^*$ such that $a_{\cT,s,h}(n) = a_{\cT}([\word u(n)_k]_k)$. The claim now follows from Proposition \ref{prop:wlog_change_s0} applied to the automaton $\cA_{\cT}$ corresponding to $\cT$ discussed at the end of Section \ref{ssec:Trans:defs}.
\end{proof}

\subsection{Representation theory} 

Let $\cT = (S,s_0,\Sigma_k,\delta,G,\lambda,\Omega,\tau)$ be an \egeaoab{} (cf.\ Theorem~\ref{thm:transducers}) and $\cT_0 = (S,s_0,\Sigma_k,\delta,G,\lambda)$ be the underlying \geaab{}. In this section we use representation theory to separate the sequence $a_{\cT}$ produced by $\cT$ into simpler components, later shown to be either strongly structured or highly Gowers uniform.

We begin by reviewing some fundamental results from representation theory. A \emph{(unitary) representation} $\rho$ of the finite group $G$ is a homomorphism $\rho \colon G \rightarrow \mathrm{U}(V)$, where $\mathrm{U}(V)$ denotes the group of unitary automorphisms of a finitely dimensional complex vector space $V$ equipped with a scalar product. The representation $\rho$ is called \emph{irreducible} if there exists no non-trivial subspace $W \subsetneq V$ such that $\rho(g)W \subseteq W$ for all $g\in G$. Every representation uniquely decomposes as the direct sum of irreducible representations.

The representation $\rho$ induces a dual representation $\rho^*$ defined on the dual space $V^*$, given by $\rho^*(g)(\varphi) = \varphi \circ \rho(g^{-1})$. Note that any element $\varphi$ of $V^*$ can be represented as $\varphi = \varphi_v$, where $\varphi_v(u) = \left<u,v\right>$ for $v \in V$, and $V^*$ inherits from $V$ the scalar product given by the formula $\left<\varphi_v,\varphi_u\right> = \left< u,v \right>$. The representation $\rho^*$ is unitary with respect to this scalar product. For a given choice of orthonormal basis, the endomorphisms on $V$ can be identified with matrices and $V^*$ can be identified with $V$. Under this identification, $\rho^*(g)$ is simply the complex conjugate of $\rho(g)$.

There only exist finitely many equivalence classes of unitary irreducible representations of $G$ and the matrix coefficients of irreducible representations of $G$ span the space of all functions $f \colon G \to \CC$ (see e.g.\ \cite[Cor\ 2.13, Prop.\ 3.29]{FultonHarris-book}; the latter can also be seen as a special case of the Peter--Weyl theorem). Here, matrix coefficients of $\rho$ are maps $G \to \CC$ of the form $g \mapsto \left< u, \rho(g)v \right>$ for some $u,v \in V$. Hence, we have the following decomposition result.

\begin{lemma}\label{lem:reps_span}
Let $\cT$ be an \egea{}. The $\CC$-vector space of maps $G \to \CC$ is spanned by maps of the form $ \alpha \circ \rho$ where $\rho \colon G \to V$ is an irreducible unitary representation of $G$ and $\alpha$ is a linear map $\End(V) \to \CC$.	
\end{lemma}

We will call $b \colon \NN_0 \to \CC$ a \emph{basic sequence produced by $\cT$} if it takes the form
\begin{equation}\label{eq:def-of-basic}
	b(n) =  \alpha \circ \rho(\lambda(s_0,(n)_k)) \ifbra{\delta(s_0,(n)_k) = s} \qquad (n \in \NN_0),
\end{equation}
where $\rho \colon G \to \U(V)$ is an irreducible unitary representation of $G$, $\alpha$ is a linear map $\End(V) \to \CC$, and $s \in S$ is a state. As a direct consequence of Lemma \ref{lem:reps_span} we have the following.

\begin{corollary}\label{cor:reps-decomposition}
Let $\cT$ be an \egea{}. The $\CC$-vector space of sequences $\NN_0 \to \CC$ produced by $\cT$ is spanned by basic sequences defined in \eqref{eq:def-of-basic}.
\end{corollary}

It follows that in order to prove Theorem \ref{thm:main_sc_nobasis} in full generality it is enough to prove it for basic sequences. There are two significantly different cases to consider, depending on the size of the kernel $\ker \rho = \set{ g \in G}{ \rho(g) = \id_V}$. Theorem \ref{thm:main_sc_nobasis} follows immediately from the following result combined with Theorem \ref{thm:transducers} and Corollary \ref{cor:reps-decomposition}.

\begin{theorem}\label{thm:dichotomy}
	Let $\cT$ be an \egea{} and let $b$ be a basic sequence given by \eqref{eq:def-of-basic}.
\begin{enumerate}[wide]
\item\label{item:383A} If $G_0 \subset \ker \rho$ then $b$ is strongly structured.
\item\label{item:383B} If $G_0 \not\subset \ker \rho$ then $b$ is highly Gowers uniform.
\end{enumerate}
\end{theorem}

One of the items above is relatively straightforward and we prove it now. The proof of the other one occupies the remainder of the paper.

\begin{proof}[Proof of Theorem \ref{thm:dichotomy}\eqref{item:383A}]
We use the same notation as in Theorem \ref{thm:transducers}. Since $\rho$ vanishes on $G_0$, it follows from property \ref{item:74C} that $\rho(\lambda(s,\word w)) = \rho(g_0^{[\word w]_k})$ for any $\word w \in \Sigma_k^*$. 
In particular, the sequence $n \mapsto \alpha \circ \rho\bra{ \lambda(s_0,(n)_k) }$ is periodic with period $d'$. Since the underlying automaton of $\cT$ is synchronising, so is the sequence $n \mapsto \ifbra{ \delta(s_0, (n)_k) = s }$. It follows that $b$ is the product of a periodic sequence and a synchronising sequence, whence $b$ is strongly structured.  
\end{proof}

\begin{example}\label{ex:main-1c}
Let $a,b,c$ be the sequences defined in Example \ref{ex:main-1}. Recall the corresponding \geaoab{} is introduced in Example \ref{ex:main-1b}. The group of the labels is $G = \{+1,-1\}$, and the corresponding group $G_0$ equals $G$. Note that $G$ has two irreducible representations: the trivial one $g \mapsto 1$, and the non-trivial one $g \mapsto g$. The trivial representation gives rise to the basic sequences $\frac{1+b}{2}$ and $\frac{1-b}{2}$, which are strongly structured. The non-trivial representations gives rise to the basic sequences $\frac{1+b}{2}c$ and $\frac{1-b}{2}c$, which are highly Gowers uniform. We have $a = 3 \frac{1+b}{2} + \frac{1-b}{2}+ \frac{1+b}{2}c$.
\end{example}

We close this section with a technical result which will play an important role in the proof of Theorem \ref{thm:dichotomy}\eqref{item:383B}. Given two representations $\rho \colon G \to \U(V)$ and $\sigma \colon H \to \U(W)$ we can consider their tensor product $\rho \otimes \sigma \colon G \times H \to \U(V \otimes W)$ which is uniquely determined by the property that $(\rho \otimes \sigma)(g,h)(v \otimes w) = \rho(g)(v) \otimes \sigma(h)(w)$ for all $v \in V,\ w \in W$. (Note that $V \otimes W$ carries a natural scalar product such that $\left< v \otimes w, v' \otimes w' \right>_{V \otimes W} = \left< v , v' \right>_V \left< w , w' \right>_W$, with respect to which $\rho \otimes \sigma$ is unitary.) In particular, for $D \geq 0$ we can define the $D$-fold tensor product $\rho^{\otimes D} \colon G^D \to \U(V^{\otimes D})$.

\begin{proposition}\label{prop:rho-averages}
	Let $\rho \colon G \to \mathrm{U}(V)$ be an irreducible representation of a group $G$ and let $G_0$ be a subgroup of $G$ such that $G_0 \not \subset \ker \rho$. Then for any $D \geq 1$ we have
	\begin{align}\label{eq:sum-not-trivial}
		\sum_{\bb g \in G_0^{D}} \rho^{\otimes D} (\bb g) = 0.
	\end{align}
\end{proposition}
\begin{proof}
	By the definition of the tensor product we find
	\begin{align*}
		\sum_{\bb g \in G_0^D} \rho^{\otimes D} (\bb g) = \bigotimes_{\omega \in [D]}\rb{\sum_{g_\omega \in G_0} \rho(g_{\omega})}.
	\end{align*}
	Thus it is sufficient to show that
	\begin{align}\label{eq_sum_rho}
		P := \EE_{g \in G_0} \rho(g) = 0.
	\end{align}
	A standard computation shows that $\rho(h) P = P$ for each $h \in G_0$, whence in particular $P^2 = P$. It follows that $P$ is a projection onto the space $U < V$ consisting of the vectors $u \in V$ such that $\rho(g)u = u$ for all $g \in G_0$. Note that $U \subsetneq V$ because $G_0 \not \subset \ker \rho$. 
	
	We claim that $U$ is an invariant space for $\rho$. It will suffice to verify that $U$ is preserved by $\rho(g_0)$, meaning that $\rho(h) \rho(g_0) u = \rho(g_0) u$ for each $u \in U$ and each $h \in G_0$. Pick any $h$ and let $h' := g_0^{-1} h g_0 \in G_0$. Then, for each $u \in U$ we have
\[
	\rho(h) \rho(g_0) u = \rho(g_0) \rho(h') u = \rho(g_0)u.
\]
Since $\rho$ is irreducible, it follows that $U = \{0\}$ is trivial. Consequently, $P = 0$.
\end{proof}

\section{Recursive relations and the cube groupoid}\label{sec:Recursive}
\subsection{Introducing the Gowers-type averages}\label{ssec:Trans:intro}

The key idea behind our proof of Theorem \ref{thm:dichotomy}\eqref{item:383B} is to exploit recursive relations connecting $\norm{a}_{U^d[k^L]}$ with $\norm{a}_{U^d[k^{L-l}]}$ for $0 < l < L$. In fact, in order to find such relations we consider somewhat more general averages which we will shortly introduce. A similar idea, in a simpler form, was used in \cite{JK-Gowers-Thue-Morse}.

\textit{Throughout this section, $\cT = (S,s_0,\Sigma_k,\delta,G,\lambda, \Omega, \tau)$ denotes an \egeaoab{}, $d \geq 1$ denotes an integer, and $\rho \colon G \to \mathrm{U}(V)$ denotes an irreducible unitary representation. All error terms are allowed to depend on $d$ and $\cT$.}

In order to study Gowers norms of basic sequences, we need to define certain averages of linear operators obtained from the representation $\rho$ in a manner rather analogous as in the definition of Gowers norms, the key difference being that the tensor product replaces the product of scalars. We define the space (using the terminology of \cite[\S 2.2]{Tao-book}, we can construe it as a higher order Hilbert space)
\begin{equation}\label{eq:def-of-Erho}
E(V)=E^d(V):=\bigotimes_{\substack{\vec\omega\in\{0,1\}^d\\  |\vec\omega| \text{ even}}} V \otimes \bigotimes_{\substack{\vec\omega \in\{0,1\}^d\\ 	 |\vec\omega| \text{ odd}}}V^*.
\end{equation}
Recall that $E(V)$ has a natural scalar product; we let $\norm{\cdot}$ denote the corresponding norm on $E(V)$ and the operator norm on $\End(E(V))$.

 The representation $\rho $ of $ G $ on $V$ induces a representation $\bbrho$ of the group $G^{[d]}=\prod_{\vec\omega \in \{0,1\}^d} G$ on $E(V)$, given by the formula
\begin{equation}\label{eq:def-of-Erho}
\bbrho(\bb g):=  \bigotimes_{\vec\omega \in \{0,1\}^d} \mathscr{C}^{\abs{\vec\omega}}\rho(g_{\vec\omega}) = \bigotimes_{\substack{\vec\omega \in\{0,1\}^d\\  \abs{\vec\omega} \text{ even} }}\rho(g_{\vec\omega}) \otimes \bigotimes_{\substack{\vec\omega \in\{0,1\}^d\\  \abs{\vec\omega} \text{ odd}	}}\rho^*(g_{\vec\omega}),
\end{equation}
where $\bb g =(g_{\vec\omega})_{\vec\omega \in \{0,1\}^d}$ and $\mathscr{C} \rho = \rho^*$ denotes the dual representation ($\mathscr{C}^2 \rho = \rho$). This is nothing else than the external tensor product of copies of $\rho$ on $V$ and $\rho^*$ on $V^*$, and as such it is irreducible and unitary with respect to the induced scalar product on $E(V)$.

Using $\bb r$ as a shorthand for $(r_{\vec\omega})_{\vec\omega \in \{0,1\}^d}$, we consider the set
\begin{align*}
	R := \set{\bb r  \in \Z^{[d]} }{ \exists \vec t \in [0,1)^{d+1}\ \forall \vec\omega \in \{0,1\}^d  \ r_{\vec \omega} = \floor{1 \vec \omega \cdot \vec t} }.
\end{align*}
\begin{definition} For $\bb s = (s_{\vec\omega})_{\vec\omega \in \{0,1\}^d} \in S^{[d]}$, $\bb r = (r_{\vec\omega})_{\vec\omega \in \{0,1\}^d} \in R$ and $L \geq 0$ we define the averages $A(\bb s, \bb r;L)\in \End(E(V))$  by the formula
\begin{align}\label{eq:def-of-A}
	A(\bb s, \bb r ;L) &= \frac{1}{k^{(d+1)L}} 	\sum_{\vec n \in \Z^{d+1} }
	\prod_{\vec\omega \in \{0,1\}^d} \ifbra{ 1\vec\omega \cdot \vec n + r_{\vec\omega} \in [k^L]}
	\\ \nonumber & \times \prod_{\vec\omega \in \{0,1\}^d}	\ifbra{\delta(s_0,(1\vec\omega \cdot \vec n + r_{\vec\omega})_k) = s_{\vec\omega}}
	\\\nonumber & \times  \bigotimes_{\vec\omega \in \{0,1\}^d} \mathscr{C}^{\abs{\vec\omega}}\rho\bra{\lambda(s_0,(1\vec\omega \cdot \vec n + r_{\vec \omega})_k)}.
\end{align}
\end{definition} 

Let us now elucidate the connection between the averages \eqref{eq:def-of-A} and Gowers norms. For $s \in S$ we let $s^{[d]} = (s)_{\vec\omega \in \{0,1\}^d}$ denote the `constant' cube with copies of $s$ on each coordinate.

\begin{lemma}\label{lem:Us-vs-A}
Let $b$ be a basic sequence produced by $\cT$, written in the form \eqref{eq:def-of-basic} for some linear map $\alpha \colon \End(V) \to \CC$ and $s \in S$. Then
 \begin{equation}\label{eq:Us-vs-A-2}
 	\norm{b}_{U^d[k^L]} \ll \norm{ A(s^{[d]},\bb 0; L)}^{1/2^d},
\end{equation}	
where the implicit constant depends on $\alpha$.
\end{lemma}
\begin{proof}
	Let $\alpha^* \colon \End(V^*) \to \CC$ denote the conjugate dual map given by the formula $\alpha^*(\psi^*)=\overline{\alpha(\psi)}$. For $\vec\omega \in \{0,1\}^d$ let $\alpha_{\vec\omega} := \alpha$ if $\abs{\vec\omega}$ is even and $\alpha_{\vec\omega} := \alpha^*$ if $\abs{\vec\omega}$ odd. Using the natural identification $$\End(E(V)) \cong \bigotimes_{\substack{\vec\omega\in\{0,1\}^d\\ |\vec\omega| \text{ even}}} \End(V) \otimes \bigotimes_{\substack{\vec\omega \in\{0,1\}^d\\ 	 |\vec\omega| \text{ odd}}}\End(V^*),$$ we define a linear map $\bbalpha \colon \End(E(V)) \to \CC$ by the formula $$\bbalpha \bra{\bigotimes_{\vec\omega \in \{0,1\}^d} \psi_{\vec\omega}} = \prod_{\vec\omega \in \{0,1\}^d} \alpha_{\vec\omega} (\psi_{\vec\omega}).$$
	
With these definitions, an elementary computation shows that 
 \begin{equation}\label{eq:Us-vs-A}
 	\norm{b}_{U^d[k^L]}^{2^d} = \frac{k^{(d+1)L}}{\Pi(k^L)} \bbalpha (A( s^{[d]},\bb 0; L)).
 \end{equation}
The factor $k^{(d+1)L}/\Pi(k^L)$, corresponding to the different normalisations used in \eqref{eq:def-of-A} and \eqref{eq:def_Gowers_N2}, has a finite limit as $L \to \infty$. Since $\bbalpha$ is linear, we have $\abs{\bbalpha(B)} \ll \norm{B}$ and \eqref{eq:Us-vs-A-2} follows.
\end{proof}

{
\begin{remark}
\begin{enumerate}[wide]
\item Generalising \eqref{eq:Us-vs-A}, the average $\bbalpha(A(\bb s, \bb r; L))$ can be construed (up to a multiplicative factor and a small error term) as the Gowers product of the $2^d$ functions $n \mapsto b(n + r_{\vec\omega})$ for all $\vec\omega \in\{0,1\}^d$.
\item 	As seen from the formulation of Lemma \ref{lem:Us-vs-A}, we are ultimately interested in the averages \eqref{eq:def-of-A} when $\bb r = \bb 0$. The non-zero values of $\bb r$ correspond to ancillary averages, which naturally appear in the course of the argument.
\item Note that for $\bb r = \bb 0$ the first product on the right hand side of \eqref{eq:def-of-A} simply encodes the condition that $\vec n \in \Pi(k^L)$. The normalising factor $k^{-(d+1)L}$ ensures that $A(\bb s, \bb r;L)$ remain bounded as $L \to \infty$.
\end{enumerate}
\end{remark}

Our next goal is to obtain a recursive relation for the averages given by \eqref{eq:def-of-A}. 
Note that any $\vec n \in \ZZ^{d+1}$ can be written  uniquely in the form $\vec n = k^l \vec m + \vec e$ where $\vec e \in [k^l]^{d+1}$ and $\vec m \in \ZZ^{d+1}$. Let $\bb v = (\bb s,\bb r) \in S^{[d]} \times R$ be arbitrary. Writing $\vec n$ as above in the definition of $A(\bb v; L)$, and letting $\bb s' \in S^{[d]}$ and $\bb r' \in \NN_0^{[d]}$ denote the `intermediate data', we obtain
\begin{align}\label{eq:def-of-A-2}
	A(\bb v;L) &= \frac{1}{k^{(d+1)L}} \sum_{\bb s'\in S^{[d]}} \sum_{\bb r' \in \NN_0^{[d]}} \sum_{\vec m \in \Z^{d+1} } \sum_{\vec e \in [k^l]^{d+1}} 
	\\ \nonumber & \phantom{\times} \prod_{\vec \omega \in \{0,1\}^d} \ifbra{ 1\vec\omega \cdot \vec m + r_{\vec\omega}' \in [k^{L-l}]} \cdot \ifbra{ \floor{ \frac{1\vec\omega \cdot \vec e + r_{\vec\omega}}{k^l }} = r_{\vec\omega}' }
	 \\ \nonumber & \times \prod_{\vec\omega \in \{0,1\}^d}  
	 \ifbra{\delta(s_0,(1\vec\omega \cdot \vec m + r_{\vec\omega}')_k) = s_{\vec\omega}'} \cdot  
	 \ifbra{\delta(s_{\vec\omega}',(1\vec\omega \cdot \vec e + r_{\vec\omega})_k^l) = s_{\vec\omega}}
	  \\ \nonumber & \times \bigotimes_{\vec\omega \in \{0,1\}^d} \mathscr{C}^{\abs{\vec\omega}}\rho\bra{\lambda(s_0,(1\vec\omega \cdot \vec m + r_{\vec\omega}' )_k)} \cdot
	 \mathscr{C}^{\abs{\vec\omega}}\rho\bra{\lambda(s_{\vec\omega}',(1\vec\omega \cdot \vec e + r_{\vec\omega} )_k^l)}.\end{align}	  
	 In this formula the term corresponding to $(\bb s', \bb r', \vec m, \vec{e}\,)$ vanishes unless $\bb r'$ belongs to $R$. Indeed, since $\bb r$ is in $R$, we can write  $r_{\vec\omega} = \lfloor 1\vec\omega \cdot \vec t\rfloor$ for some $\vec t \in [0,1)^{d+1}$, and then the corresponding term vanishes unless 
	\begin{align*}
		r'_{\vec \omega} &= \floor{\frac{1 \vec \omega \cdot \vec e + r_{\vec\omega}}{k^{l}}} =
		\floor{\frac{1 \vec \omega \cdot \vec e + \floor{1 \vec \omega \cdot \vec t}}{k^{l}}}
		= \floor{\frac{1 \vec \omega \cdot \vec e + 1 \vec \omega \cdot \vec t}{k^{l}}}
			= \lfloor 1 \vec \omega \cdot \vec{t'}\rfloor,
	\end{align*}
	where $\vec{t'} := (\vec e + \vec{t}\,)/k^{l} \in [0,1)^{d+1}$.
The key feature of formula \eqref{eq:def-of-A-2} is that the two inner sums over $\vec m$ and $\vec e$ can be separated, leading to
  
\begin{align}\label{eq:def-of-A-3}
	A(\bb v;L) &= \sum_{\bb v' \in S^{[d]} \times R} A(\bb v'; L-l) \cdot M(\bb v', \bb v;l),
\end{align}	  
where the expression   $M(\bb v', \bb v;l)$ is given for any $\bb v = (\bb s, \bb r)$ and $\bb v' = (\bb s', \bb r') $ in $S^{[d]} \times R$ by the formula
\begin{align}\label{eq:def-of-A-4}
	M(\bb v', \bb v;l) &= \frac{1}{k^{(d+1)l}} \sum_{\vec e \in [k^l]^{d+1}}  \ifbra{ \floor{ \frac{1\vec\omega \cdot \vec e + r_{\vec\omega}}{k^l }} = r_{\vec\omega}' }
	 \\ \nonumber & \times \prod_{\vec\omega \in \{0,1\}^d} 
	 \ifbra{\delta(s_{\vec\omega}',(1\vec\omega \cdot \vec e + r_{\vec\omega})_k^l) = s_{\vec\omega}}
	 \\ \nonumber & \times \bigotimes_{\vec\omega \in \{0,1\}^d} 
	 \mathscr{C}^{\abs{\vec\omega}} \rho\bra{\lambda(s_{\vec\omega}',(1\vec\omega \cdot \vec e + r_{\vec\omega} )_k^l)}.\end{align}	  
The form of the expression above is our main motivation for introducing in the next section the category $\cV$.

\subsection{The category $\cV^d(\cT)$}\label{ssec:Trans:cV}

To keep track of the data parametrising the averages defined above, we define the \emph{$d$-dimensional category $\cV  =\cV^d (\cT)$ associated to the \geaoab{} $\cT$} (or, strictly speaking, to the underlying \gea{} without output). The objects $\Ob_{\cV}$ of this category are the pairs $\bb v = (\bb s, \bb r) \in S^{[d]} \times R$. Since $R$ and $S$ are finite, there are only finitely many objects.
The morphisms of $\cV$ will help us keep track of the objects $\bb v' = (\bb s', \bb r')$ appearing in formul\ae\ \eqref{eq:def-of-A-3} and \eqref{eq:def-of-A-4}. These morphisms are parametrised by the tuples 
\[ 
(l,\vec e,\bb s', \bb r) \in \N_0 \times [k^l]^{d+1} \times S^{[d]} \times R=\Mor_{\cV}.
\] 
The tuple $(l,\vec e,\bb s', \bb r) $ describes an arrow from   $\bb v' = (\bb s', \bb r')$ to $\bb v = (\bb s, \bb r)$, where $\bb s = ( s_{\vec\omega})_{\vec\omega}$ and $\bb r' = ( r'_{\vec\omega})_{\vec\omega}$ are given by the formul\ae\ 
  
\begin{equation}\label{eq:deq-s-and-r'}
s_{\vec\omega} = \delta(s'_{\vec\omega},(1\vec\omega\cdot{\vec e} + r_{\vec\omega})_k^l) \quad 
\text{ and } \quad
r'_{\vec\omega}= \floor{ \frac{1\vec\omega\cdot{\vec e} +r_{\vec\omega}}{k^l}}.
\end{equation}
 We will denote this morphism by  $\tilde e = (l,\vec e\,) \colon \bb v' \to \bb v$. The number $\mathrm{deg}(\tilde e):=l$ is called the \emph{degree} of $\tilde e$. In order to define the composition of morphisms, we state the following lemma.
{

 \begin{lemma}\label{lem:catvcomp}
 If  $\tilde{e'} = (l',\vec{e'})$ is a morphism from $ \bb v''$ to $\bb v'$ and $\tilde{e} = (l,\vec{e}) $ is a morphism from $\bb v' $ to $\bb v$, then $\tilde{e''} = (l+l', k^{l} \vec{e'}+\vec {e}) $ is a morphism from $ \bb v'' $ to $ \bb v$.\end{lemma}
 \begin{proof}
 Using the same notation as above, for each $\vec\omega \in \{0,1\}^d$ we have the equality
\begin{equation}\label{eq:558:1}
(1\vec\omega\cdot\vec{e''}+ r_{\vec\omega}'')_k^{l+l'}= (1\vec\omega\cdot\vec{e'}+ r'_{\vec\omega})_k^{l'} (1\vec\omega\cdot\vec e+  r_{\vec\omega})_k^{l} 
\end{equation} 
which allows us to verify that
\begin{align*}
\delta\bra{s''_{\vec\omega},(1\vec\omega\cdot \vec {e''} + r_{\vec\omega}'')_k^{l''}}
&= \delta\bra{s''_{\vec\omega},(1\vec\omega\cdot\vec{e'}+ r'_{\vec\omega})_k^{l'} (1\vec\omega\cdot\vec e+  r_{\vec\omega})_k^{l}}
\\&= \delta\bra{s'_{\vec\omega},(1\vec\omega\cdot\vec e+  r_{\vec\omega})_k^{l}} = s_{\vec\omega},
\end{align*}
and by basic algebra we have
\begin{align*}
	r_{\vec\omega}'' = \floor{\frac{1\vec\omega \cdot \vec{e'} + r_{\vec\omega}'}{k^{l'}}} 
	&= \floor{\frac{1\vec\omega \cdot \vec{e'} + \floor{ \frac{1\omega\cdot{\vec e} +r_{\vec\omega}}{k^l}}}{k^{l'}}} 
	\\&= \floor{\frac{ k^l(1\vec\omega \cdot \vec{e'}) + 1\omega\cdot{\vec e} +r_{\vec\omega}}{k^{l'+l}}}
	= \floor{\frac{1\vec\omega \cdot \vec{e''} + r_{\vec\omega}}{k^{l''}}}. \qedhere
\end{align*}
 \end{proof}

Lemma \ref{lem:catvcomp} allows us to define the composition of two morphisms   $\tilde{e'} = (l',\vec{e'}) \colon \bb v'' \to \bb v'$ and $\tilde e = (l,\vec e\,) \colon \bb v' \to \bb v$ as $\tilde{e''} = \tilde{e'}\circ \tilde{e} := (l'',\vec{e''})= (l+l', k^l \vec{ e'}+\vec e) \colon \bb v'' \to \bb v.$ 
The composition is clearly associative, and for each object $\bb v$ the map $(0,\vec 0)\colon \bb v \to \bb v$ is the identity map. This shows that $\cV$ is indeed a category.

We let  $\Mor(\bb v',\bb v)$ denote the set of morphism from $\bb v'$ to $\bb v$. The degree induces an $\N_0$-valued gradation on this set, which means that $\Mor(\bb v',\bb v)$ decomposes into a disjoint union $\coprod_{l = 0}^{\infty} \Mor_l(\bb v',\bb v)$, where $\Mor_l(\bb v',\bb v)$ is the set of morphisms $\tilde e \colon \bb v' \to \bb v$ of degree $l$. The degree of the composition of two morphisms is equal to the sum of their degrees. A crucial property of the category $\cV$ is that morphisms can also be uniquely decomposed in the following sense.

\begin{lemma}
 Let $\tilde{ e''} \colon \bb v'' \to \bb v$ be a morphism and let $0\leq l' \leq \mathrm{deg}(\tilde{ e''})$ be an integer. Then there exist unique morphisms $\tilde e'$ and $\tilde e$ with  $\tilde{ e''} = \tilde{e'}  \circ \tilde {e}$ and $\mathrm{deg}(\tilde{ e'})=l'$.
\end{lemma}
\begin{proof}
Put $\bb v'' = (\bb s'', \bb r'')$, $\bb v = (\bb s, \bb r)$, $l=\deg(\tilde {e''})-l'$ and $\tilde{e''}=(l'+l'',\vec{e''})$. Then there exists a unique decompositon $\vec {e''} = k^{l} \vec{e'} + \vec{e}$, where $\vec{e'} \in [k^{l'}]^{d+1}$ and $\vec{e} \in [k^{l}]^{d+1}$.
Thus, we can define $\bb v' = (\bb s', \bb r')$ by the formul\ae
$$s'_{\vec \omega} := \delta(s''_{\vec \omega}, (1 \vec \omega \cdot \vec{e'} + r'_{\omega})_k^{l'}) \quad \text{ and }  \quad r'_{\vec\omega} := \floor{\frac{1 \vec \omega \cdot \vec{e} + r_{\vec \omega}}{k^{l'}}} .$$
A computation analogous to the one showing that composition of morphisms is well-defined shows that $(l' + l, \vec {e''}) = (l', \vec{e'}) \circ (l, \vec{e})$. Conversely, it is immediate that such a decomposition is unique.
\end{proof}

\begin{remark}
	As a particular case of \eqref{eq:def-of-A-3}, we can recover $A(\bb v; L)$ from  $ M(\bb v', \bb v;L)$. Indeed, it follows from  \eqref{eq:def-of-A-3} that
\begin{align}\label{eq:402:1}
	A(\bb v;L) &= \sum_{\bb v' \in S^{[d]} \times R} A(\bb v'; 0) \cdot M(\bb v', \bb v;L).
\end{align}	
	Recalling the definition of $A(\bb v'; 0)$ in \eqref{eq:def-of-A} we see that the only non-zero terms in the sum \eqref{eq:402:1} above correspond to objects of the form $\bb v' = (s_0^{[d]}, \bb r')$ where $\bb r' \in R$ is such that there exists $\vec n \in \ZZ^{d+1}$ with $r_{\vec\omega}' = 1\vec\omega \cdot \vec n$ for each $\vec\omega \in \{0,1\}^d$. Let $R' \subset R$ denote the set of all $\bb r'$ with the property just described and note that if $\bb r' \in R'$ then $A(s_0^{[d]},\bb r'; 0) = \id_{E(V)}$ is the identity map. It follows that  
\begin{align}\label{eq:402:2}
	A(\bb v;L) &= \sum_{\bb r' \in R'} M((s_0^{[d]},\bb r'), \bb v ;L).
\end{align}	 
We stress that $\bb 0 \in R'$, but as long as $d\geq 2$, $R'$ contains also other elements. For instance, when $d = 2$ the set $R$ consists of exactly the elements $(r_{00},r_{01},r_{10},r_{11})$ of the form $$(0,0,0,0),(0,0,0,1),(0,0,1,1),(0,1,0,1),(0,1,1,1),(0,1,1,2),$$ while $R'$ consists of elements of the form $$(0,0,0,0),(0,0,1,1),(0,1,0,1),(0,1,1,2).$$
\end{remark}

\subsection{The subcategory $\cU^d(\cT)$}\label{ssec:Trans:cU}

The object
\begin{equation}\label{eq:def-of-v0}
	\bb v_0 = \bb v_0^{\cT} = (s_0^{[d]}, 0^{[d]}) \in \Ob_{\cV}
\end{equation}
is called the \emph{base object}. In the recurrence formul\ae\ above the objects of particular importance are those which map to the base object. We define a (full) subcategory $\cU$ of $\cV$, whose objects are those among $\bb v \in \Ob_{\cV}$ for which $\Mor(\bb v, \bb v_0) \neq \emptyset$ and $ \Mor(\bb v_0, \bb v) \neq \emptyset$ (in fact, we will prove in Lemma~\ref{lem:Ob0-reachable} that the former condition is redundant), and whose morphisms are the same as those in $\cV$.

\begin{lemma}\label{lem:Ob0-reachable}
	There exists $l_0 \geq 0$ such that  $\Mor_l(\bb v, \bb v_0) \neq \emptyset$ for any $\bb v \in \Ob_{\cV}$ and any $l \geq l_0$.
\end{lemma}
\begin{proof}
	We first consider objects of the form $\bb v = (\bb s, \bb 0)$. 
	Letting $e_0 = [\word w_0^{\cT}]_k$\footnote{We recall that $\bfw_0^{\cT}$ is a synchronizing word for $\cT$, i.e. for any $s \in S$ we have $\delta(s, \bfw_0^{\cT}) = s_0, \lambda(s, \bfw_0^{\cT}) = id$.} and $e_i = 0$ for $1 \leq i \leq d$ and taking sufficiently large $l$ we find a morphism   $\tilde e = (l, \vec e) \colon \bb v \to \bb v_0$.
	
	In the general case, since   $\Mor(\bb v, \bb v_0) \subset \Mor(\bb v, \bb v') \circ \Mor(\bb v', \bb v_0)$, it only remains to show that for each object $\bb v = (\bb s, \bb r) \in \Ob_{\cV}$ there exists some $\bb v' = (\bb s', \bb 0) \in \Ob_{\cV}$ such that   $\Mor(\bb v, \bb v') \neq \emptyset$.
	Since $\bb r \in R$, there exists a vector $\vec t \in [0,1)^{d+1}$ such that 
	\begin{align}\label{eq_t}
		r_{\vec \omega} = \lfloor 1 \vec \omega \cdot \vec t\rfloor \text{ for all } \vec \omega \in \{0,1\}^d.
	\end{align}
	It follows from piecewise continuity of the floor function that there exists an open set of $\vec t \in [0,1)^{d+1}$ that fulfill~\eqref{eq_t}. Hence, one can pick, for any sufficiently large $l \geq 0$, $\vec t $ of the form $\vec t = \vec e/k^l$, where $\vec e \in [k^l]^{d+1}$. 
	Choosing $s'_{\vec\omega} = \delta(s_{\vec\omega}, (1 \vec \omega \cdot \vec e)_k^l)$ finishes the proof.
\end{proof}

\begin{corollary}\label{cor:Ob0-is-closed}
Let $\bb v, \bb v' \in \Ob_{\cV}$. If $\Mor(\bb v, \bb v') \neq \emptyset$ and $\bb v \in \Ob_{\cU}$, then $\bb v' \in \Ob_{\cU}$.
\end{corollary}
\begin{proof}
	  By Lemma \ref{lem:Ob0-reachable}, we have $\Mor(\bb v', \bb v_0) \neq \emptyset$. Moreover, we find by Lemma~\ref{lem:catvcomp}, that $\Mor(\bb v_0, \bb v') \supset \Mor(\bb v_0, \bb v) \circ \Mor(\bb v, \bb v') \neq \emptyset$.
\end{proof}

\begin{lemma}\label{lem:Ob0-const-cube}
	Let $s \in S$ and let $\bb v = (s^{[d]},\bb 0) \in \Ob_{\cV}$. Then $\bb v \in \Ob_{\cU}$.
\end{lemma}
\begin{proof}
	It is enough to show that $\bb v_0$ is reachable from $\bb v$. Let $\bb w \in \Sigma_k^*$ be a word synchronising the underlying automaton of $\cT$ to $s$. Let $e_0 = [\word w]_k$, $e_i = 0$ for $1 \leq i \leq d$ and let $l > \abs{\word w} + \log_k(d)$. Then we have the morphism   $\tilde e = (l, \vec e\,) \colon \bb v_0 \to \bb v$, as needed.
\end{proof}

\subsection{The cube groupoid}\label{ssec:Trans:Q}

By essentially the same argument as in \eqref{eq:def-of-A-3} we conclude that for any $\bb v, \bb v', \bb v'' \in S^{[d]} \times R$ we have
 \begin{align}\label{eq:def-of-A-6}
	M(\bb v,\bb v'';L) &= \sum_{\bb v' \in S^{[d]} \times R} M(\bb v, \bb v'; L-l) \cdot M(\bb v', \bb v'';l).
\end{align}	  
}

Regarding the group $G^{[d]}$ as a category with one object, we define the \emph{$d$-dimensional fundamental functor} $\bblambda=\bblambda^d_{\cT}\colon \cV^d(\cT) \to G^{[d]}$ as follows. All objects are mapped to the unique object of $G^{[d]}$ and an arrow $\tilde e=(l,\vec e) \colon \bb v =(\bb s,\bb r) \to \bb v' =(\bb s',\bb r')$ is mapped to   
\begin{equation}
{\bblambda}(\tilde e)
=\bra{\lambda_{\vec\omega}(\tilde e)}_{\vec\omega \in \{0,1\}^d}
=\bra{\lambda( s_{\vec\omega},(1\vec\omega \cdot \vec e +  r'_{\vec\omega})_k^l}_{\vec\omega \in \{0,1\}^d}.
\end{equation}
It follows from 
Lemma~\ref{lem:catvcomp}
that $\bblambda$ is indeed a functor. 

We are now ready to rewrite $M$ in a more convenient form:
\begin{equation}\label{eq:558:2}
	M(\bb v, \bb v'; l) = \sum_{\tilde e \in \Mor_l(\bb v, \bb v') } \bbrho(\bblambda(\tilde e)).
\end{equation}

In order to keep track of the terms appearing in \eqref{eq:558:2}, we introduce the families of cubes $\cQ^d_l$. For two objects $\bb v, \bb v' \in \Ob_{\cV}$ the \emph{cube family} $\cQ^d_l(\cT)(\bb v,\bb v')$ is defined to be the subset of $G^{[d]}$ given by
\begin{equation}\label{eq:def-of-Q}
\cQ^d_l(\cT)(\bb v,\bb v')= 
	\{ \bblambda(\tilde e) \mid \tilde e \in \Mor_l(\bb v,\bb v')\}.
\end{equation}

\subsection{Frobenius--Perron theory}

In this section we review some properties of nonnegative matrices and their spectra.
For a matrix $W$ we let $\varrho(W)$ denote its spectral radius. By Gelfand's formula, for any matrix norm $\norm{\cdot}$ we have
\begin{equation}\label{eq:Gelfand}
	\varrho(W) = \lim_{l \to \infty} \norm{W^l}^{1/l}.
\end{equation}

If $W,W'$ are two matrices of the same dimensions, then we say that $W \geq W'$ if the matrix $W-W'$ has nonnegative entries. Accordingly, $W > W'$ if $W-W'$ has strictly positive entries. In particular, $W$ has nonnegative entries if and only if $W \geq 0$.

 Let $W=(W_{ij})_{i,j\in I}$ be a nonnegative matrix with rows and columns indexed by a (finite) set  $I$. For $J\subset I$, we let $W[J]=(W_{ij})_{i,j\in J}$ denote the corresponding principal submatrix. We define a directed graph with the vertex set $I$ and with an arrow from $i\in I$ to $j\in I$ whenever $W_{ij}>0$. We say that $i\in I$ \emph{dominates} $j \in I$ if there is a directed path from $i$ to $j$\footnote{We note that $i$ always dominates itself via the empty path.}, and that $i$ and $j$ are \emph{equivalent} if they dominate each other. We refer to the equivalence classes of this relation as the \emph{classes} of $W$. We say that a class $J_1$ \emph{dominates} a class $J_2$ if $j_1$ dominates $j_2$ for some (equivalently, all) $j_1\in J_1$ and $j_2\in J_2$. This is a weak partial order on the set of classes. 

A nonnegative matrix $W$ is called \emph{irreducible} if it has only one class. The Frobenius--Perron theorem says that every  irreducible matrix has a real eigenvalue $\lambda$ equal to its spectral radius, its multiplicity is one, and there is a corresponding eigenvector with strictly positive entries \cite[Thm.\ I.4.1 \& I.4.3]{bookMinc}.
 For any nonempty subset $J\subset I$ we have $\varrho(W[J])\leq \varrho(W)$, and the inequality is strict if $W$ is irreducible and $J\neq I$ \cite[Cor.\ II.2.1 \& II.2.2]{bookMinc}.  We call a class $J\subset I$ \emph{basic} if $\varrho(W[J]) = \varrho(W)$, and \emph{nonbasic} otherwise.

\newcommand{\M}{N}

\begin{proposition}\label{prop:FrobeniusPerron}
 Let $W=(W_{ij})_{i,j\in I}$ be a nonnegative matrix such that the matrices $W^l$ are jointly bounded for all $l\geq 0$. Let $\M\leq W$ be a nonnegative matrix, and let $J \subset I$ be a \emph{basic} class of $W$ such that $\M[J] \neq W[J]$. Then there is a constant $\gamma<1$ such that for all $i\in I$ and $j\in J$ we have 
\begin{equation}\label{eq:34:00}
(\M^l)_{ij} \ll  \gamma^l \text{ as }  l\to \infty.
\end{equation}
\end{proposition}
\begin{proof}
Let $V = \RR^I$ denote the vector space with basis $I$ equipped with the standard Euclidean norm. We identify matrices indexed by $I$ with linear maps on $V$ and let $\norm{A}$ denote the operator norm of a matrix $A$ (in fact, we could use any norm such that $0 \leq A_1 \leq A_2$ implies $\norm{A_1} \leq \norm{A_2}$).
For $J\subset I$ let $V[J]$ denote the vector subspace of $V$ with basis $J$.

By Gelfand's theorem, the spectral radius of $W$ can be computed as $\varrho(W) = \lim_{l\to \infty} \norm{W^l}^{1/l}.$
Since the matrices $W^l$ are jointly bounded, we have $\varrho(W)\leq 1$. Furthermore, if $\varrho(W)<1$, then there is some $\lambda<1$ such that $\norm{W^l} \leq \lambda^l$ for $l$ large enough, and hence all entries of $W^l$ (and a fortiori of $\M^l$) tend to zero at an exponential rate, proving the claim.
Thus, we may assume that $\varrho(W)=1$. 

\begin{step}\label{step1}
No two distinct basic blocks of $W$ dominate each other.
\end{step}
\begin{proof}
Let $J_1$ and $J_2$ be distinct basic blocks of $W$, and for the sake of contradiction suppose that $J_1$ dominates $J_2$. By Frobenius--Perron theorem applied to the matrices $W[J_1]$ and $W[J_2]$, there are vectors $x_1 \in V[J_1]$ and $x_2\in V[J_2]$ with $x_1,x_2>0$ and $W[J_1]x_1=x_1$, $W[J_2]x_2=x_2$. Since $J_1$ dominates $J_2$, there exists $m \geq 1$ such that any vertex $i \in J_1$ is connected to any vertex $j \in J_2$ by a path of length $< m$. Let $U := \frac{1}{m}(I+W+\dots+W^{m-1})$.
It follows (cf.\ \cite[Thm.\ I.2.1]{bookMinc}) for a sufficiently small value of $\e > 0$ that we have
\begin{equation}\label{eq:wform}
	U x_1 \geq x_1 + \varepsilon x_2, \qquad U x_2 \geq x_2.
\end{equation}
Iterating \eqref{eq:wform}, for any $l \geq 0$ we obtain
\begin{equation}\label{eq:wform-2}
U^l x_1 \geq x_1 + l \varepsilon x_2.
\end{equation}
On the other hand, powers of $U$ are jointly bounded because the powers of $W$ are jointly bounded, which yields a contradiction.
\end{proof}

Let $x\in V[J]$, $x>0$, be the eigenvector of $W[J]$ with eigenvalue $1$. Let $K$ be the union of all the classes of $W$ dominated by $J$ except for $J$ itself. By Step \ref{step1} all the classes in $K$ are nonbasic, and the subspace $V[K]$ is $W$-invariant. The spectral radius of the matrix $W[K]$ is equal to the maximum of the spectral radii of $W[J']$ taken over all the classes $J' \subset K$, and hence $\varrho(W[K])<1$. 

\begin{step}\label{step2}
We have $\M[J]^l x < W[J]^l x$ for all $l \geq \abs{J}$.
\end{step}
\begin{proof}
	As $\M[J] \neq W[J]$, there exist $i,i' \in J$ such that $\M[J]_{i,i'} < W[J]_{i,i'}$. Since $x > 0$, we have $(\M[J]^l x)_j < (W[J]^l x)_j$ for each $j \in J$ that is an endpoint of a path of length $l$ containing the arrow $i,i'$. As $W[J]$ is irreducible, such path exists for all $l \geq \abs{J}$.
\end{proof}

\begin{step}\label{step3}
We have $\norm{\M^l x} \ll  \gamma^l $ for some  $\gamma<1$ as $ l\to \infty$
\end{step}
\begin{proof}
Since $\varrho(W[K])<1$, it follows from Gelfand's theorem that for any sufficiently large $n$ we have  
\begin{equation}\label{eq:34:10}
\norm{\M[K]^n}\leq \norm{W[K]^n} < 1.
\end{equation}
By Step \ref{step2}, for any sufficiently large $n$ there exist $\lambda < 1$ and $v \in V[K]$ such that
\begin{equation}\label{eq:34:11}
	\M^n x \leq \lambda x + v.
\end{equation}
Pick $n$, $\lambda$ and $v$ such that \eqref{eq:34:10} and \eqref{eq:34:11} hold, and assume additionally that $\lambda$ is close enough to $1$ so that $\norm{\M[K]^n}\leq \lambda$, whence 
\begin{equation}\label{eq:34:12}
	\M^n v = \M[K]^n v  \leq \lambda v.
\end{equation}
Applying \eqref{eq:34:11} iteratively, for any $l \geq 0$ we obtain
\[
\M^{ln}x \leq \lambda^l x + l \lambda^{l-1} v.
\]
It follows that Step \ref{step3} holds with any $\gamma$ such that $ \gamma < \lambda^{1/n}$.
\end{proof}
Since $x > 0$ (as an element of $V[J]$) the claim \eqref{eq:34:00} follows immediately from Step \ref{step3}.
\end{proof}

\subsection{From recursion to uniformity}\label{ssec:Trans:R->U}

In Section \ref{sec:Cubes} we obtain a fairly complete description of the cubes $\cQ_l^d(\bb v, \bb v')$. The main conclusion is the following (for a more intuitively appealing equivalent formulation, see Theorem \ref{thm:char_fact_is_Z}).

\begin{theorem}\label{thm:cubes}
	There exist cubes $\bb g_{\bb v} \in G^{[d]}$, $\bb v \in \Ob_{\cU}$, and a threshold $l_0 \geq 0$ such that for each $l \geq l_0$ and each $\bb v, \bb v' \in \Ob_{\cU}$ we have 
	\[
		\cQ^d(\cT)(\bb v, \bb v') = \bb g_{\bb v}^{-1} G_0^{[d]}  \bb H \bb g_{\bb v'},
	\]
	where $\bb H < G^{[d]}$ is given by
	\[
		\bb H = \set{ \bra{g_0^{1\vec\omega \cdot \vec e}}_{\vec \omega \in \{0,1\}^d}}{ \vec e \in \NN_0^{d+1}}.
	\]
\end{theorem}

Presently, we show how the above result completes the derivation of our main theorems. We will need the following corollary.

\begin{corollary}\label{cor:G_0_v_0}
	There exists $l_0 \geq 0$ such that for all $l \geq l_0$ we have
	\begin{align}\label{eq:sum-not-trivial-2}
		G_0^{[d]} \subset \cQ^d_{l}(\bb v_0, \bb v_0).
	\end{align}
\end{corollary}
\begin{proof}
	Follows directly from the observation that $\id^{[d]}_G \in \bb H$ (where we use the notation from Theorem \ref{thm:cubes}) and $G_0$ is normal in $G$.
\end{proof}

\begin{proof}[Proof of Theorem \ref{thm:dichotomy}\eqref{item:383B}]
Recall that in \eqref{eq:Us-vs-A} we related the Gowers norms in question to the averages $A(\bb v; \bb L)$ with $\bb v \in \Ob_{\cV}$ taking the form $\bb v = (s^{[d]},\bb 0)$ and that by Lemma \ref{lem:Ob0-const-cube} the relevant cubes belong to $\Ob_{\cU}$. Hence, it will suffice to show that for any $\bb v \in \Ob_{\cU}$ we have the bound $\norm{A(\bb v; L)} \ll k^{-cL}$ for a positive constant $c > 0$. 

Let us write $A$ and $M$ (defined in \eqref{eq:def-of-A} and \eqref{eq:def-of-A-4} respectively) in the matrix forms: 
 
\[ 
A(L) = \big(A(\bb v; L)\big)_{\bb v \in \Ob_{\cV}} \text{ and  } M(L) = \big(M(\bb v, \bb v'; L)\big)_{\bb v,\bb v' \in \Ob_{\cV}};
\]
note 
that the entries of the matrices $A(L)$ and $M(L)$ are elements of $\End(E(V))$. This allows us to rewrite the recursive relations \eqref{eq:def-of-A-3} and \eqref{eq:def-of-A-6} as matrix multiplication:
\begin{align}\label{eq_recursion_A}
    A(l+l')  &= A(l)M(l'), &&& M(l+l') &= M(l)M(l'), &&& (l,l' \geq 0). 
  \end{align}
Consider also the real-valued matrices $N(L)$ and $W(L)$, of the same dimension as $M(L)$, given by
  \begin{align*}
    N(L)_{\bb v, \bb v'} &= \norm{ M(L)_{\bb v, \bb v'} }_2 =  \frac{1}{k^{(d+1)L}} \norm{\sum_{\tilde e \in \Mor_L(\bb v,\bb v')} \bbrho(\bblambda(\tilde e)) }_2\\
    W(L)_{\bb v, \bb v'} &= \frac{\abs{\Mor_L(\bb v,\bb v')}}{k^{(d+1)L}}.
  \end{align*}

Note that $0 \leq N(l) \leq W(l)$ for each $l \geq 0$ by a straightforward application of the triangle inequality and the fact that $\bbrho$ is unitary. Moreover, for reasons analogous to \eqref{eq_recursion_A} we also have
\begin{align}\label{eq_recursion_A-2}
	N(l+l') &\leq N(l)N(l') &&&
  W(l+l') = W(l)W(l'), &&& (l,l' \geq 0).
  \end{align}
As a consequence, $W(l) = W^l$, where $W := W(1)$.  
It also follows directly from how morphisms are defined that $W(l)_{\bb v, \bb v'} \leq 1$ for all $\bb v,\bb v' \in \Ob_{\cV}$ and $l \geq 0$. 

Let $l_0$ be the constant from Corollary \ref{cor:G_0_v_0}. Then, by Proposition \ref{prop:rho-averages} we have $N(l)_{\bb v_0,\bb v_0} \neq W(l)_{\bb v_0,\bb v_0}$ for all $l \geq l_0$. We are now in position to apply Proposition \ref{prop:FrobeniusPerron}, which implies that there exits $\gamma < 1$ such that for any $\bb v \in \cV$ and any $\bb u \in \cU$ we have
\begin{equation}\label{eq:778:1}
	N(l_0)^l_{\bb v, \bb u} \ll \gamma^{l/l_0}.
\end{equation}
Using with \eqref{eq_recursion_A-2}, \eqref{eq:778:1} can be strengthened to
\begin{equation}\label{eq:778:2}
	N(L)_{\bb v, \bb u}  \ll \gamma^L.
\end{equation}
Finally, using \eqref{eq_recursion_A} and the fact that all norms on finitely dimensional spaces are equivalent, for any $\bb u \in \Ob_{\cU}$ and $L \geq 0$ we conclude that 
\begin{equation}\label{eq:778:3}
	\norm{A(\bb u; L)} = \norm{ A(L)_{\bb u} } = \norm{ \sum_{\bb v \in \cV} A(0)_{\bb v} M(L)_{\bb v, \bb u} } 
	\ll \sum_{\bb v \in \cV} N(L)_{\bb v, \bb u} \ll \gamma^L.
\end{equation}
\end{proof}

\section{Cube groups}\label{sec:Cubes}

\subsection{Groupoid structure}\label{ssec:Cubes:groupoid}

We devote the remainder of this paper to proving Theorem \ref{thm:cubes}, which provides a description of the cube sets $\cQ^d_l(\bb v, \bb v')$.
In this section we record some basic relations between the $\cQ^d_l(\bb v, \bb v')$ for different $\bb v, \bb v' \in \Ob_{\cV}$. Our key intention here is to reduce the problem of describing $\cQ_l^d(\cT)(\bb v, \bb v')$ for arbitrary $\bb v, \bb v' \in \Ob_{\cU}$ to the special case when $\bb v = \bb v' = \bb v^{\cT}_0$.

\begin{lemma}\label{lem:Q_transitivity}
	Let $\cT$ be an \egeaab{} and let $\bb v, \bb v', \bb v'' \in \Ob_{\cV}$ and $l,l' \geq 0$. Then  
	\[
		\cQ^d_{l'}(\cT)(\bb v, \bb v') \cdot \cQ^d_l(\cT)(\bb v', \bb v'') \subseteq \cQ^d_{l+l'}(\cT)(\bb v, \bb v'').
	\] 
\end{lemma}
\begin{proof}
	This is an immediate consequence of the fact that $\bblambda$ is a functor.
\end{proof}

\begin{lemma}\label{lem:Q_limit}
	Let $\cT$ be an \egeaab{} and $\bb v, \bb v' \in \Ob_{\cU}$. Then the limit
	\begin{equation}\label{eq:def-of-Q-infty}
		 \cQ^d(\cT)(\bb v, \bb v') = \lim_{l \to \infty} \cQ^d_l(\cT)(\bb v, \bb v')
	\end{equation}
	exists. Moreover, there exist cubes $\bb g_{\bb v} \in G^{[d]}$ such that for any $\bb v, \bb v' \in \Ob_{\cU}$ the limit in \eqref{eq:def-of-Q-infty} is given by 
	\begin{equation}\label{eq:def-of-Q-infty-2}
	\cQ^d(\cT)(\bb v, \bb v') = \bb g_{\bb v}^{-1} \cdot \cQ^d(\cT)(\bb v_0^{\cT}, \bb v_0^{\cT}) \cdot \bb g_{\bb v'}.
	\end{equation}
\end{lemma}
\begin{remark}
	Since $\cQ^d_l(\cT)(\bb v, \bb v')$ are finite, \eqref{eq:def-of-Q-infty} is just a shorthand for the statement that there exists $l_0 = l_0(\cT,\bb v, \bb v') \geq 0$ and a set $\cQ^d(\cT)(\bb v, \bb v')$ such that $\cQ^d_l(\cT)(\bb v, \bb v') = \cQ^d(\cT)(\bb v, \bb v')$ for all $l \geq l_0$.
\end{remark}
\begin{proof}
	Note first that $\cQ^d_1(\cT)(\bb v_0^{\cT}, \bb v_0^{\cT}) \neq 0$ contains the identity cube $\id_G^{[d]}$, arising from the morphism $(1,\vec 0) \colon \bb v_0^{\cT} \to \bb v_0^{\cT}$. It follows from Lemma \ref{lem:Q_transitivity} that the sequence $\cQ^d_l(\cT)(\bb v_0^{\cT}, \bb v_0^{\cT})$ is increasing in the sense that $\cQ^d_{l}(\cT)(\bb v_0^{\cT}, \bb v_0^{\cT}) \subseteq \cQ^d_{l+1}(\cT)(\bb v_0^{\cT}, \bb v_0^{\cT})$ for each $l \geq 0$. Since the ambient space $G^{[d]}$ is finite, it follows that the sequence $\cQ^d_l(\cT)(\bb v_0^{\cT}, \bb v_0^{\cT})$ needs to stabilise, and in particular the limit \eqref{eq:def-of-Q-infty} exists for $\bb v = \bb v' = \bb v_0^{\cT}$. 
	
	It follows from Lemma \ref{lem:Q_limit} that for any $m,m',l \geq 0$ we have the inclusion 
	\[
	\cQ^d_{m'}(\cT)(\bb v_0^{\cT}, \bb v) \cdot \cQ^d_l(\cT)(\bb v, \bb v') \cdot \cQ^d_m(\cT)(\bb v', \bb v_0^{\cT}) \subseteq \cQ^d_{m+m'+l}(\cT)(\bb v_0^{\cT}, \bb v_0^{\cT}).
	\]
	Since there exist morphisms $\bb v_0^{\cT} \to \bb v, \bb v' \to \bb v_0^{\cT}$, there exist $m,m' \geq 0$ and $\bb g_{\bb v}, \tilde{\bb g}_{\bb v'}$ (any elements of $\cQ^d_m(\cT)(\bb v_0^{\cT}, \bb v)$ and $\cQ^d_{m'}(\cT)(\bb v', \bb v_0^{\cT})^{-1}$ respectively) such that for all $l \geq 0$ we have  
	\[
	\bb g_{\bb v} \cdot \cQ^d_l(\cT)(\bb v, \bb v') \cdot \tilde{\bb g}_{\bb v'}^{-1} \subseteq \cQ^d_{m+m'+l}(\cT)(\bb v_0^{\cT}, \bb v_0^{\cT}).
	\]
	We thus conclude that if $l \geq 0$ is sufficiently large then 
	\begin{equation}\label{eq:574:1}
	 \cQ^d_l(\cT)(\bb v, \bb v')  \subseteq 
	 \bb g_{\bb v}^{-1} \cdot \cQ^d(\cT)(\bb v_0^{\cT}, \bb v_0^{\cT}) \cdot \tilde{\bb g}_{\bb v'}.
	\end{equation}
	Reasoning in a fully analogous manner (with pairs $(\bb v, \bb v')$ and $(\bb v_0^{\cT}, \bb v_0^{\cT})$ swapped), for sufficiently large $l$ we obtain the reverse inclusion  
	\begin{equation}\label{eq:574:2}
	 \cQ^d(\cT)(\bb v_0^{\cT}, \bb v_0^{\cT})  \subseteq 
	 \bb h_{\bb v}^{-1} \cdot \cQ^d_{l}(\cT)(\bb v, \bb v')\cdot \tilde{\bb h}_{\bb v'},
	\end{equation}
	for some cubes $\bb h_{\bb v}, \tilde{\bb h}_{\bb v'} \in G^{[d]}$. Comparing cardinalities we conclude that both \eqref{eq:574:1} and \eqref{eq:574:2} are in fact equalities. Hence, the limit  \eqref{eq:def-of-Q-infty} exists for all $\bb v, \bb v' \in \Ob_{\cU}$ and  
	\begin{equation}\label{eq:574:1a}
	 \cQ^d(\cT)(\bb v, \bb v')  = 
	 \bb g_{\bb v}^{-1} \cdot \cQ^d(\cT)(\bb v_0^{\cT}, \bb v_0^{\cT}) \cdot \tilde{\bb g}_{\bb v'}.
	\end{equation}
	
	Note that $\bb g_{\bb v}$ and $\tilde{\bb g}_{\bb v}$ are determined up to multiplication on the left 
	 by an element of $\cQ^d(\cT)(\bb v_0^{\cT}, \bb v_0^{\cT})$ and we may take $\bb g_{\bb v_0^{\cT}} = \tilde{\bb g}_{\bb v_0^{\cT}}  = \id_G^{[d]}$. Hence, $\cQ^d(\cT)(\bb v_0^{\cT}, \bb v_0^{\cT})$ is a group. It now follows from Lemma \ref{lem:Q_transitivity} that
	 $\cQ^d(\cT)(\bb v_0^{\cT}, \bb v) \cdot \cQ^d(\cT)(\bb v, \bb v_0^{\cT}) \subseteq \cQ^d(\cT)(\bb v_0^{\cT}, \bb v_0^{\cT})$, or equivalently
	\begin{equation}\label{eq:574:1b}
	  \cQ^d(\cT)(\bb v_0^{\cT}, \bb v_0^{\cT}) \cdot \tilde{\bb g}_{\bb v} \bb g_{\bb v}^{-1} \cdot \cQ^d(\cT)(\bb v_0^{\cT}, \bb v_0^{\cT})
	  \subseteq  \cQ^d(\cT)(\bb v_0^{\cT}, \bb v_0^{\cT}),
	\end{equation}
meaning that $\tilde{\bb g}_{\bb v} \bb g_{\bb v}^{-1} \in \cQ^d(\cT)(\bb v_0^{\cT}, \bb v_0^{\cT})$. Hence, we may take $\tilde{\bb g}_{\bb v} = \bb g_{\bb v}$, since we can multiply $\tilde{\bb g}_{\bb v}$ from the left with $(\tilde{\bb g}_{\bb v} \bb g_{\bb v}^{-1})^{-1} \in \cQ^d(\cT)(\bb v_0^{\cT}, \bb v_0^{\cT})$. 
\end{proof}

As a consequence of Lemma \ref{lem:Q_limit}, the sets $\cQ^d(\cT)(\bb v, \bb v')$ for $\bb v, \bb v' \in \Ob_{\cU}$ form a groupoid, in the sense that we have the following variant of Lemma \ref{lem:Q_transitivity}.

\begin{corollary}\label{cor:Q_transitivity}
	Let $\cT$ be an \egeaab{} and let $\bb v, \bb v', \bb v'' \in \Ob_{\cU}$. Then 
	\[
		\cQ^d(\cT)(\bb v, \bb v') \cdot \cQ^d(\cT)(\bb v', \bb v'') = \cQ^d(\cT)(\bb v, \bb v'').
	\] 
\end{corollary}

In particular, in order to understand all of the sets $\cQ^d(\bb v, \bb v')$ (up to conjugation) it will suffice to understand one of them. This motivates us to put
\begin{equation}\label{eq:def-of-Q-no-l}
	\cQ^d(\cT) = \cQ^d(\cT)(\bb v_0^{\cT}, \bb v_0^{\cT}).
\end{equation}

We also mention that the sets $\cQ^d(\cT)$ are easy to describe for small values of $d$. 

\begin{lemma}\label{lem:Q_small_d}
	Let $\cT$ be an \egeaab{} and $d \in \{0,1\}$. Then
	\[
		\cQ^d(\cT) = G^{[d]}.
	\]
\end{lemma}
\begin{proof}
	Immediate consequence of the definition of $\cQ^d(\cT)$ and property \ref{item:74B}.
\end{proof}

\subsection{Characteristic factors}\label{ssec:Cubes:factors}

A morphism between \geaab{}  $\cT$ and $\bar\cT$ given by $(\phi,\pi)$ is a \emph{factor map} if both $\phi \colon S \to \bar S$ and $\pi \colon G \to \bar G$ are surjective. In this case, $\bar\cT$ is a \emph{factor} of $\cT$. The group homomorphism $\pi$ induces a projection map $\bbpi \colon G^{[d]} \to \bar G^{[d]}$. As $\bblambda$ is a functor, $\bbpi(\cQ^d(\cT)) \subset \cQ^d(\bar\cT)$ for all $d \geq 0$. In fact, for large $l \geq 0$ we have the following commutative diagram:
\begin{center}
\begin{tikzcd}
\Mor_{l}(\bb v_0, \bb v_0) \arrow[r, "\id", hook] \arrow[d, "\bblambda", two heads]
& \Mor_{l}(\bar{\bb v}_0, \bar{\bb v}_0)  \arrow[d, "\bblambda", two heads] \\
\cQ^d_l(\cT) \arrow[r, "\bbpi"]
& \cQ^d_l(\bar\cT)
\end{tikzcd}
\end{center}
The map labelled $\id$ takes the morphism $(l,\vec e)\colon \bb v_0 \to \bb v_0$ to morphism given by the same data $(l,\vec e)\colon \bar{\bb v}_0 \to \bar{\bb v}_0$. We will say that the factor $\bar\cT$ of $\cT$ is \emph{characteristic} if for each $d \geq 0$ we have the equality $\cQ^d(\cT) = \bbpi^{-1}\bra{\cQ^d(\bar\cT)}$. Note that if $\bar\cT$ is a characteristic factor of $\cT$ then the cube groups $\cQ^d(\cT)$ are entirely described in terms of the simpler cube groups $\cQ^d(\bar\cT)$. It is also easy to verify that if $\bar \cT$ is a characteristic factor of $\cT$ then any characteristic factor of $\bar\cT$ is also a characteristic factor of $\cT$.

For instance, a \geaab{} is always its own factor, which is always characteristic. A possibly even more trivial\footnote{no pun intended} example of a factor is the trivial \geaab{} $\cT_{\textrm{triv}}$ with a single state, trivial group, and the other data defined in the only possible way. In fact, $\cT_{\textrm{triv}}$ is the terminal object, meaning that it is a factor of any \geaab{} . The trivial \geaab{} is a characteristic factor of $\cT$ if and only if $\cQ^d(\cT) = G^{[d]}$ for all $d \geq 0$.

%
%

\begin{lemma}\label{lem:trans_fact_is_natural}
	Let $\cT$ be an \egeaab{}  and let $(\phi,\pi)$ be a factor map from $\cT$ to $\bar \cT$. If $\ker \pi \subset G_0$ then $\bar \cT$ is an \egeaab{} and $d'_{\cT} = d'_{\bar{\cT}}$.
\end{lemma}
\begin{proof}
We verify each of the defining properties of an \egeaab{} in turn. It is clear that $\bar\cT$ is strongly connected and that $\bar\cT$ is synchronising; in fact, if $\word w \in \Sigma_k^*$ is synchronising to the state $s \in S$ for $\cT$ then $\word w$ is also synchronising to the state $\phi(s) \in \bar S$ for $\bar\cT$. 
We also find that $\bar\cT$ is idempotent and $\bar{\lambda}(\bar s, \mathtt{0}) = \id$ for all $\bar s \in \bar S$.
Put also $\bar G_0 = \pi(G_0)$ and $\bar g_0 = \pi(g_0)$.
	
	For \ref{item:74B}, let $\bar s, \bar s' \in \bar S$ and let $s \in \phi^{-1}(\bar s)$ and $s' \in \phi^{-1}(\bar s')$. Then
	\begin{align*}
		\set{ \bar \lambda(\bar s,\word w) }{ \word w \in \Sigma_k^l, \ \bar \delta(\bar s) = \bar{s}'} &\supseteq
		\set{ \pi\bra{ \lambda(s,\word w)} }{ \word w \in \Sigma_k^l, \ \delta(s) = {s}'} = \bar G,
	\end{align*}
	and the reverse inclusion is automatic.
	
For  \ref{item:74C}, let. Let $\bar s, \bar s' \in \bar S$ and $s \in \phi^{-1}(\bar s)$. Then
\begin{align*}
	&\set{ \bar\lambda(\bar s, \word w)}{ \word w \in \Sigma_k^*,\ \bar\delta(\bar s, \word w) = \bar s', [\word w]_k \equiv r \bmod d'}
	\\ = &\bigcup_{s' \in \phi^{-1}(\bar s')} 
	\set{ \pi( \lambda(s, \word w)) }{ \word w \in \Sigma_k^*,\ \delta(s, \word w) = s', [\word w]_k \equiv r \bmod d'}
	\\ = &\bigcup_{s' \in \phi^{-1}(\bar s')} \pi(g_0^r G_0) = \bar g_0^r \bar G_0.
\end{align*}

For \ref{item:74D}, let $\bar s, \bar s' \in \bar S$, $\bar g \in \bar G_0$, let $s \in \phi^{-1}(\bar s)$.
Then 
\begin{align*}
	&\gcd{}_k^* \bra{\set{[\word w]_k}{\word w \in \Sigma_k^l,\ \bar\delta(\bar s,\word w) = \bar s',\ \bar\lambda(\bar s,\word w) = \bar g}}
	\\ = &\gcd{}_k^*\bra{
	\bigcup_{g \in \pi^{-1}(\bar g)} \bigcup_{s' \in \phi^{-1}(\bar s')}
	\set{[\word w]_k}{\word w \in \Sigma_k^l,\ \delta(s,\word w) = s',\ \lambda(s,\word w) = g}
	}
	\\ = &\gcd{}_k^*\bra{	
	\set{
	c(g,s')}{
	{g \in \pi^{-1}(\bar g)},\ {s' \in \phi^{-1}(\bar s')}
	}
	} = d',	
\end{align*}
where $c(g,s')$ is, thanks to \ref{item:74D} for $\cT$, given by
\[
c(g,s') = \gcd{}_k^*\bra{\set{[\word w]_k}{\word w \in \Sigma_k^l,\ \delta(s,\word w) = s',\ \lambda(s,\word w) = g}} = d'.	\qedhere
\]
\end{proof}

\subsection{Group quotients}\label{ssec:Cubes:grp_quot}

Let $\cT = (S, s_0, \Sigma_k, \delta, G,  {\lambda})$ be a \geaab{} . One of the basic ways to construct a factor of $\cT$ is to leave the state set unaltered and replace $G$ with a quotient group. 
More precisely, for a normal subgroup $H < G$, we can consider the quotient \geaab{} without output $\cT/H = (S, s_0, \Sigma_k, \delta, G/H,  \bar{\lambda})$ with the same underlying automaton and group labels given by $\bar \lambda(s,j) = \overline {\lambda(s,j)} \in G/H$ for $s \in S$, $j \in \Sigma_k$. Thus defined \geaab{} is a factor of $\cT$, with the factor map given by $(\id_S,\pi)$, where $\pi \colon G \to G/H$ is the quotient map.
The purpose of this section is to identify an easily verifiable criterion ensuring that the factor $\cT/H$ is characteristic. As a convenient byproduct, this will allow us to mostly suppress the dependency on the dimension $d$ from now on.

In fact, it is not hard to identify the maximal normal subgroup of $G$ such that the corresponding factor is characteristic. Let $H < G$ be normal and let
$\pi \colon G \to G/H$ denote the quotient map. For any $d \geq 0$, the map $\bbpi \colon \cQ^d(\cT) \to \cQ^d(\cT/H)$ is surjective and for any $\bb g \in \cQ^d(\cT)$ we have 
\( \bbpi^{-1}(\bbpi(\bb g)) = \bb g H^{[d]}. \)
It follows that $\cT/H$ is characteristic if and only if $H^{[d]} \subset \cQ^d(\cT)$. In particular, if $\cT/H$ is characteristic then $\cQ^d(\cT)$ contains all cubes with an element of $h$ at one vertex and $\id_G$ elsewhere. In order to have convenient access to such cubes, for $g \in G$ and $\vec\sigma \in \{0,1\}^d$ put
\begin{equation}
\label{eq:def-of-c-cube}
	\bb c^d_{\vec\sigma}(h) = \bra{h^{\ifbra{\vec\omega = \vec\sigma}}}_{\vec\omega \in \{0,1\}^d} = \bra{c_{\vec\omega}}_{\vec\omega \in \{0,1\}^d} \text{ where } c_{\vec\omega} = 
\begin{cases}	
	 h & \text{ if }  \vec\omega = \vec\sigma,\\
	 \id_G & \text{ if } \vec\omega \neq \vec\sigma.
\end{cases}
\end{equation}
We also use the shorthand $\vec 1 = (1,1,\dots,1) \in \{0,1\}^d$, where $d$ will always be clear from the context.
This motivates us to define
\begin{equation}
	\label{eq:def-of-K}
	K = K(\cT) = \set{ h \in G }{ \bb c_{\vec\sigma}^d(h) \in \cQ^d(\cT) \text{ for all } d \geq 0 \text{ and } \vec\sigma \in \{0,1\}^d}.
\end{equation}
Since $\bb c_{\vec\sigma}^d \colon G \to G^{[d]}$ is a group homomorphism for each $d \geq 0$  and $\vec\sigma \in \{0,1\}^d$, $K$ is a group. 
As any cube can be written as a product of cubes with a single non-identity entry, the condition $H^{[d]} \subset \cQ^d(\cT)$ for all $d \geq 0$ holds if and only if $H < K$. If $\cT$ is an \egea{} then \eqref{eq:def-of-K} and \ref{item:74C} guarantee that $K < G_0$.

\begin{proposition}\label{prop:grp_quot_is_char}
	Let $\cT$ be an \egeaab{} and let $H < G$ be a normal subgroup. Then the following conditions are equivalent:
\begin{enumerate}
\item $\cT/H$ is a characteristic;
\item $H < K(\cT)$.
\end{enumerate}
\end{proposition}
\begin{proof}
	Immediate consequence of the above discussion.
\end{proof}

We devote the remainder of this section to obtaining a description of $K$ that is easier to work with.
Fix a value of $d \geq 0$ for now, and let $\cT$ be a \geaab{}. For each $1 \leq j \leq d+1$, there is a natural projection $\pi_j \colon \{0,1\}^{d+1} \to \{0,1\}^{d}$ which discards the $j$-th coordinate, that is,
\[ \pi_j( \omega_1,\omega_2,\dots,\omega_{j-1},\omega_{j},\omega_{j+1},\dots \omega_{d+1} ) = (\omega_1,\dots,\omega_{j-1},\omega_{j+1},\dots,\omega_{d+1}) \] Accordingly, for each $1 \leq j \leq d+1$, we have the embedding $\iota_j \colon G^{[d]} \to G^{[d+1]}$ which copies the entries along the $j$-th coordinate, that is,
\[
\iota_j(\bb g) = \bra{ g_{\pi_j(\vec\omega)} }_{\vec \omega \in \{0,1\}^{d+1}}.
\]

\begin{lemma}\label{lem:Q-dim-inclusion}
	Let $1 \leq j \leq d+1$ and let $\cT$ be an \egeaab{}. Then
	\begin{equation}\label{eq:315:1}
		\iota_j \bra{ \cQ^d(\cT) } \subset  \cQ^{d+1}(\cT).
	\end{equation}
\end{lemma} 
\begin{proof}
	Let $\tilde e = (l,\vec e) \colon \bb v_0^{\cT} \to \bb v_0^{\cT}$ be a morphism in $\cV^d(\cT)$, and let $\bb g = \bblambda(\tilde e)$ be an element of $\cQ^d(\cT)$. Then there is a corresponding morphism $\tilde f = (l,\vec f) \colon \bb v_0^{\cT} \to \bb v_0^{\cT}$ in $\cV^{d+1}(\cT)$ obtained by inserting $0$ in $\vec e$ at $j$-th coordinate, that is,
	\[
		(f_0,f_1,\dots,f_{j-1},f_j,f_{j+1},\dots,f_{d+1}) = (e_0,e_1, \dots,e_{j-1},0,e_j,\dots, e_{d}).
	\]
	It follows directly from the definition of $\bblambda$ that $\bblambda(\tilde f) = \iota_j(\bblambda(\tilde e))$. Since $\tilde e$ was arbitrary, \eqref{eq:315:1} follows.
\end{proof}

\begin{corollary}
	Let $\cT$ be an \egeaab{}. Then $g^{[d]} \in \cQ^d(\cT)$ for all $d \geq 0$ and $g \in G$. Moreover, the group $K$ is normal in $G$ and contained in $G_0$.
\end{corollary}
\begin{proof}
	The first statement follows from Lemma \ref{lem:Q_small_d}. The second one follows, since
\[
\bb c_{\vec \sigma}^d(g h g^{-1})g^{[d]} = g^{[d]} \bb c_{\vec \sigma}^d(h) \text{ for all } d \geq 0,\ \sigma \in \{0,1\}^d \text{ and } g,h \in G. \qedhere
\] 
\end{proof}

\begin{lemma}\label{lem:H-in-Q-reduction}
	Let $\cT$ be an \egeaab{} and let $h \in G$. Suppose that for each $d \geq 0$ there exists $\vec\rho = \vec\rho(d) \in \{0,1\}^d$ such that $\bb c_{\vec\rho}^d(h) \in \cQ^d(\cT)$. Then $h \in K$.
\end{lemma}
\begin{proof}
	We need to show that $\bb c_{\vec\sigma}^d(h) \in \cQ^d(\cT)$ for each $d \geq 0$ and $\vec\sigma \in \{0,1\}^d$. We proceed by double induction, first on $d$ and then on $\abs{\set{i \leq d}{\sigma_i \neq \rho_i}}$, where $\vec\rho = \vec\rho(d)$. The cases $d = 0$ and $\vec\sigma = \vec\rho$ are clear.
	
	Suppose now that $d \geq 1$ and $\vec\sigma \neq \vec\rho$. For the sake of notational convenience, assume further that $\vec\rho = \vec{1}$; one can easily reduce to this case by reflecting along relevant axes. 
By inductive assumption (with respect to $\vec\sigma$), $\cQ^d(\cT)$ contains $\bb c^d_{\vec\omega}(h)$ for all $\vec\omega \in \{0,1\}^d$ with $\abs{\vec\omega} > \abs{\vec\sigma}$. Moreover, by inductive assumption (with respect to $d$) and as $\cQ^{d-1}(\cT)$ is a group, we have $\{\id, h\}^{[d-1]} \subseteq \cQ^{d-1}(\cT)$.
	Consider the product 
	\[ 
	\bb g = \prod_{\vec\omega \geq \vec\sigma} c_{\vec\omega}^d(h)
	 = \bra{g_{\vec\omega}}_{\vec\omega \in \{0,1\}^d} \text{ where } g_{\vec\omega} =
	 \begin{cases}
	  h &\text{if } \vec\omega \geq \vec\sigma, \\
	  \id_G &\text{otherwise},
	 \end{cases}
	\]
	where the order on $\{0,1\}^d$ is defined coordinatewise, meaning that $\vec\omega \geq \vec\sigma$ if and only if $\omega_j \geq \sigma_j$ for all $1 \leq j \leq d$. 
It follows from Lemma \ref{lem:Q-dim-inclusion} that $\bb g \in \cQ^d(\cT)$. In fact $\bb g \in \iota_j(\{\id, h\}^{[d-1]}) \subseteq \iota_j(\cQ^{d-1}(\cT))$ for each $1 \leq j \leq d$ such that $\sigma_j = 0$. It remains to notice that all terms in the product defining $\bb g$, except for $\bb c_{\vec\sigma}^d(h)$, are independently known to belong to $\cQ^d(\cT)$.
\end{proof}

The following reformulation of Lemma \ref{lem:H-in-Q-reduction} above will often be convenient.

\begin{corollary}\label{cor:H-in-Q-reduction}
	Let $\cT$ be an \egeaab{} and let $g,h \in G$. Suppose that for each $d \geq 0$, the group $\cQ^d(\cT)$ contains a cube with $h$ on one coordinate and $g$ on all the remaining $2^d-1$ coordinates. Then $g \equiv h \bmod{K}$.
\end{corollary}

We are now ready to state the criterion for characteristicity of the quotient \geaab{} in terms of the generating set.

\begin{corollary}\label{cor:grp_quot_is_char}
	Let $\cT$ be an \egeaab{}, let $X \subset G$ be any set and put $H := \abra{X}^{G}$ be the normal closure of $X$. Suppose that for each $h \in X$ and $d \geq 0$ there exists $\vec\rho \in \{0,1\}^d$ such that $\bb c_{\vec \rho}^d(h) \in \cQ^d(\cT)$. Then the factor $\cT/H$ is characteristic.
\end{corollary}

\subsection{State space reduction}\label{ssec:Cubes:state_quot}

In this section we consider another basic way of constructing factor maps, namely by removing redundancies in the set of states. Ultimately, we will reduce the number of states to $1$ by repeatedly applying Proposition \ref{prop:grp_quot_is_char} (which simplifies the group structure and hence makes some pairs of states equivalent) and Proposition \ref{prop:state_red_is_char} below (which identifies equivalent states, leading to a smaller \geaab{}).
The following example shows the kind of redundancy we have in mind.

\begin{example}
Consider the base-$3$ analogue of the Rudin--Shapiro sequence, given by the following \geaab{} with $G = \{+1,-1\}$ and output function $\tau(s,g) = g$ (cf.\ Example \ref{ex:Rudin-Shapiro}).
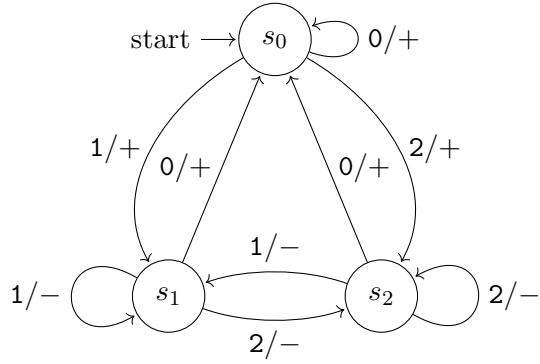
\begin{center}
\begin{tikzpicture}[shorten >=1pt,node distance=2cm, on grid, auto] 
   \node[initial, state] (s_0)   {$s_{0}$}; 
   \node (s_x) [below = of s_0] {}; 
   \node[state] (s_1) [below left =of s_x] {$s_{1}$}; 
   \node[state] (s_2) [below right=of s_x] {$s_{2}$}; 
   
 \tikzstyle{loop}=[min distance=4mm,in=120,out=210,looseness=1]
    \path[->]     
    (s_0) edge [loop left] node {\texttt 1/$+$} (s_1);

 \tikzstyle{loop}=[min distance=4mm,in=-115,out=90,looseness=.5]
    \path[->]     
    (s_1) edge [left] node {\texttt 0/$+$} (s_0);

 \tikzstyle{loop}=[min distance=4mm,in=60,out=-30,looseness=1]
    \path[->]     
    (s_0) edge [loop right] node {\texttt 2/$+$} (s_2);

 \tikzstyle{loop}=[min distance=4mm,in=-65,out=90,looseness=.5]
    \path[->]     
    (s_2) edge [right] node {\texttt 0/$+$} (s_0);

  \tikzstyle{loop}=[min distance=6mm,in=210,out=-210,looseness=7]        
   \path[->]     
    (s_1) edge [loop left] node  {\texttt 1/$-$} (s_1);
    
    \tikzstyle{loop}=[min distance=6mm,in=30,out=-30,looseness=7]        
   \path[->] 
    (s_2) edge [loop right] node  {\texttt 2/$-$} (s_2);
    
   \tikzstyle{loop}=[min distance=6mm,in=20,out=-20,looseness=7]        
   \path[->] 
    (s_0) edge [loop right] node  {\texttt 0/$+$} (s_0);

   \tikzstyle{loop}=[min distance=6mm,in=20,out=160,looseness=.5]     
   \path[->] (s_2) edge [loop above] node {\texttt 1/$-$} (s_1);
    
    \tikzstyle{loop}=[min distance=6mm,in=-160,out=-20,looseness=.5]  
   \path[->] (s_1) edge [loop below] node {\texttt 2/$-$} (s_2); 
\end{tikzpicture}
\end{center}
The states $s_1$ and $s_2$ serve the same purpose and can be identified, leading to a smaller \geaab{}:
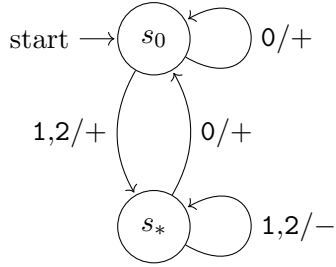
\begin{center}
\begin{tikzpicture}[shorten >=1pt,node distance=2.5cm, on grid, auto] 
   \node[initial, state] (s_0)   {$s_{0}$}; 
   \node[state] (s_1) [below=of s_0] {$s_{*}$}; 
 \tikzstyle{loop}=[min distance=4mm,in=120,out=240,looseness=1]
    \path[->]     
    (s_0) edge [loop left] node {\texttt 1,\texttt 2/$+$} (s_1);

 \tikzstyle{loop}=[min distance=4mm,in=-60,out=60,looseness=1]
    \path[->]     
    (s_1) edge [loop right] node {\texttt 0/$+$} (s_0);

  \tikzstyle{loop}=[min distance=6mm,in=30,out=-30,looseness=7]        
   \path[->] 
    (s_0) edge [loop right] node {\texttt 0/$+$} (s_0);
   \path[->]     
    (s_1) edge [loop right] node  {\texttt 1,\texttt 2/$-$} (s_1);
\end{tikzpicture}
\end{center}
\end{example}

Motivated by the example above, for a \geaab{} $\cT$ we consider the equivalence relation $\sim$ of $S$, where $s \sim s'$ if and only if $\lambda(s,\word u) = \lambda(s',\word u)$ for all $\word u \in \Sigma_k^*$. Equivalently, $\sim$ is the minimal equivalence relation such that $s \sim s'$ implies that $\lambda(s,j) = \lambda(s',j)$ and $\delta(s,j) \sim \delta(s',j)$ for all $j \in \Sigma_k$. We define the reduced \geaab{} $\cT_{\mathrm{red}} = (\bar S,\bar s_0, \Sigma_k, \bar\delta, \bar{\lambda}, G)$, where $\bar S = S/{\sim}$, $\bar \delta(\bar s, j) = \overline{\delta(s,j)}$ and $\bar \lambda(\bar s, j) = \lambda(s,j)$ for all $s \in S$, $j \in \Sigma_k$. There is a natural factor map $\cT \to \bar\cT$ given by $(\phi,\id_G)$ where $\phi \colon S \to S/{\sim}$ takes $s \in S$ to its equivalence class. Note that if $\cT$ is natural, then Lemma \ref{lem:trans_fact_is_natural} guarantees that so is $\cT_{\mathrm{red}}$.

\begin{proposition}\label{prop:state_red_is_char}
	Let $\cT$ be an \egeaab{}. Then the factor $\cT_{\mathrm{red}}$ is characteristic.
\end{proposition}
\begin{proof}
	Pick any $d \geq 0$. Let $S_0 = \set{s \in S}{s \sim s_0}$ be the equivalence class of $s_0$. Any morphism $\tilde e = (l,\vec e) \colon \bar{\bb v}_0 \to \bar{\bb v}_0$ in $\cT_{\mathrm{red}}$ can be lifted to a morphism $(l,\vec e) \colon (\bb s,0) \to (\bb s',0)$ in $\cT$, where $\bb s, \bb s' \in S_0^{[d]}$. Conversely, any morphism $(l,\vec e) \colon (\bb s,0) \to (\bb s',0)$ in $\cT$ with $\bb s, \bb s' \in S_0^{[d]}$ gives rise to the corresponding morphism $(l,\vec e) \colon \bar{\bb v}_0 \to \bar{\bb v}_0$. Hence,
\begin{equation}\label{eq:297:1}
	\cQ^d(\cT_{\mathrm{red}}) = \bigcup_{\bb s, \bb s' \in S_0^{[d]}} \cQ^d(\cT)((\bb s,0), (\bb s',0)).
\end{equation}
	
	Let $l$ be a large integer and let $\vec f = ([\word w_0^{\cT}]_k,0,\dots,0) \in \NN_0^{d+1}$\footnote{We recall that $\bfw_0^{\cT}$ is a synchronizing word for $\cT$, i.e. for any $s \in S$ we have $\delta(s, \bfw_0^{\cT}) = s_0, \lambda(s, \bfw_0^{\cT}) = id$.}. Then $1\vec\omega \cdot \vec f = [\word w_0^{\cT}]_k$ for each $\vec\omega \in \{0,1\}^d$, whence we have the morphism $\tilde f = (l, \vec f) \colon  (\bb s, 0)\to \bb v_0^{\cT}$ with $\bblambda(\tilde f) = \id_G^{[d]}$ for any $\bb s \in S_0^{[d]}$. It follows from Lemma \ref{lem:Q_limit}, that we can take $g_{(\bb s, \bb 0)} = \id_G^{[d]}$, and said Lemma guarantees that $\cQ^d(\cT)\bra{(\bb s, \bb 0),(\bb s', \bb 0)} = \cQ^d(\cT)$ for all $\bb s, \bb s' \in S_0^{[d]}$. Inserting this into \eqref{eq:297:1} we conclude that 	$\cQ^d(\cT_{\mathrm{red}}) = \cQ^d(\cT)$, meaning that $\cT_{\mathrm{red}}$ is a characteristic factor of $\cT$.
\end{proof}

\subsection{Host--Kra cube groups}\label{ssec:Cubes:Host-Kra}

The groups $\cQ^d(\cT)$ can be viewed as distant analogues of Host--Kra cube groups, originating from the work of these two authors in ergodic theory \cite{HostKra-2005,HostKra-2008} (the name, in turn, originates from \cite{GreenTao-2010-Ann}).

Let $G$ be a group and let $d \geq 0$. The Host--Kra cube group $\HK^d(G)$ is the subgroup of $G^{[d]}$ generated by the \emph{upper face cubes} $\bra{g^{\ifbra{\omega_j = 1}}}_{\vec\omega \in \{0,1\}^d}$ where $1 \leq j \leq d$ and $g \in G$. If $G$ is abelian then $\HK^d(G)$ consists of the cubes $\bb g = (g_{\vec\omega})_{\vec\omega \in \{0,1\}^d}$ where $g_{\vec\omega} = h_0\prod_{j=1}^d h_j^{\omega_j}$ for some sequence $h_0,h_1,\dots,h_d \in G$. In general, let $G = G_0 = G_1 \supseteq G_2 \supseteq \dots$ be the lower central series of $G$, where for each $i \geq 1$ the group $G_{i+1}$ is generated by the commutators $ghg^{-1}h^{-1}$ with $g \in G_i$, $h \in G$. Let also $\vec\sigma_1,\vec\sigma_2,\dots,\vec\sigma_{2^d}$ be an ordering of $\{0,1\}^d$ consistent with inclusion in the sense that if $\vec\sigma_i \leq \vec\sigma_j$ (coordinatewise) then $i \leq j$. Then $\HK^d(G)$ consists precisely of the cubes which can be written as $\bb g_1 \bb g_2 \dots \bb g_{2^d}$ where for each $j$ there exists $g_j \in G_{\abs{\vec\sigma_j}}$ such that $\bb g_j = \bra{g_{j,\vec\omega}}_{\vec\omega \in \{0,1\}^d}$ and $g_{j,\vec\omega} = g_{j}$ if $\vec\omega \geq \vec\sigma_j$ (coordinatewise) and $g_{j,\vec\omega} = \id_G$ otherwise. The Host--Kra cube groups are usually considered for nilpotent groups $G$, that is, groups such that $G_{s+1} = \{\id_G\}$ for some $s \in \NN$, called the step of $G$. (In fact, one can consider the Host--Kra cube groups corresponding to filtrations other than the lower central series, but these are not relevant to the discussion at hand.)

Let $\cT$ be an invertible \egeaab{} given by $(\Sigma_k,G,\lambda)$. Then a direct inspection of the definition shows that $\cQ^d(\cT)$ consists of all the cubes of the form $\bra{ \lambda\bra{( 1\vec\omega \cdot \vec e)_k}}_{\vec\omega \in \{0,1\}^d}$ where $\vec e \in \NN_0^k$. In particular, letting $e_i = 0$ for $i \neq j$ and taking $e_j \in \NN_0$ such that $\lambda((e_j)_k) = g$ (whose existence is guaranteed by \ref{item:74B}) we conclude that $\cQ^d(\cT)$ contains the upper face cube corresponding to any $g \in G$ and $1 \leq j \leq d$. Hence,
\begin{equation}\label{eq:731:1}
	\cQ^d(\cT) \supseteq \HK^d(G).
\end{equation}
In fact, the cube $\bra{ \lambda\bra{(1\vec\omega \cdot \vec e)_k }}_{\vec\omega \in \{0,1\}^d}$ belongs to $\HK^d(G)$ if $\vec e \in \NN_0^{d+1}$ has non-overlapping digits in the sense that for each $m$ there is at most one $j$ such that the $m$-th digit of $(e_j)_k$ is non-zero. Since the cube groups $\HK^d(G)$ are relatively easy to describe, especially in the abelian case, one can view the indices $[\cQ^d(\cT):\HK^d(G)]$ ($d \geq 0$) as a measure of complexity of $\cT$. We will ultimately reduce to the case when $\cQ^d(\cT) = \HK^d(G)$.

As alluded to above, the inclusion in \eqref{eq:731:1} can be strict. For instance, one can show that $\cQ^2(\cT) = \HK^2(G)$ if and only if $\lambda((e_0)_k)\lambda((e_0+e_1+e_2)_k) \equiv \lambda((e_0+e_1)_k)\lambda((e_0+e_2)_k) \bmod G_2$ for all $e_0,e_1,e_2 \in \NN_0$.

	Suppose now, more generally, that $\cQ^d(\cT) = \HK^d(G)$ for all $d \geq 0$. Put $G_\infty := \lim_{n \to \infty} G_n$. It follows from Lemma \ref{lem:H-in-Q-reduction} that $K(\cT) = G_\infty$. If $G$ is nilpotent then $K(\cT) = \{\id_G\}$ is trivial and consequently $\cT$ has no proper characteristic factors. If $G$ is not nilpotent then the factor $\cT/G_{\infty}$ is characteristic, and one can check that $\cQ^d(\cT/G_{\infty}) = \HK^d(G/G_{\infty})$. In particular, iterating this reasoning we see that if $\cQ^d(\cT) = \HK^d(G)$ then $\cT$ has a characteristic factor given by $(\Sigma_k, \bar G, \bar\lambda)$ where $G$ is a nilpotent group. In fact, this is only possible if $G$ is a cyclic group, as shown by the following lemma. Since its importance is purely as a motivation and we do not use it in the proof of our main results, we only provide a sketch of the proof.

\begin{lemma}\label{lem:Q=HK_implies_abelian}
	Let $\cT$ be an invertible \egeaab{} given by $(\Sigma_k, G, \lambda)$. Assume further that $G$ is nilpotent and $\cQ^d(\cT) = \HK^d(G)$ for all $d \geq 0$. Then $G$ is a subgroup of $\ZZ/(k-1)\ZZ$ and $\lambda((n)_k) = \lambda(1)^n$ for all $n \in \Sigma_k$.	
\end{lemma}
\begin{proof}[Sketch of a proof]
	Let $s$ be the step of $G$ so that $G_{s+1} = \{\id_G\}$, and for ease of writing identify $\lambda$ with a map $\lambda \colon \NN_0 \to G$.
	Since $\bblambda = \lambda^{[d]}$ maps parallelepipeds of the form $\bra{ 1\vec\omega \cdot \vec e}_{\vec\omega \in \{0,1\}^d}$ for $\vec e \in \NN_0^{d+1}$ to $\cQ^d(\cT) = \HK^d(G)$, the sequence $\lambda$ is a polynomial with respect to the lower central series (see e.g.\ \cite[Def.{} 1.8 and Prop.{} 6.5 ]{GreenTao-2012} for the relevant definition of a polynomial sequence). It follows \cite[Lem.{} A.1]{GreenTao-2010-ARL} that there exist $g_i \in G_i$ for $0 \leq i \leq s$ such that 
	\begin{equation}\label{eq:373:0}
		\lambda(n) = g_0 g_1^n g_2^{\binom{n}{2}} \dots g_s^{\binom{n}{s}}, \qquad (n \in \NN_0).
	\end{equation}
	Moreover, $g_i$ are uniquely determined by the sequence $\lambda$.  Note also that $g_0 = \id_G$ since $\lambda(0) = \id_G$. We will show that $g_i = \id_G$ for all $i \geq 2$. In fact, we will show by induction on $r$ that $g_2,g_3,\dots,g_r \in G_{r+1}$ for each $r \geq 1$ (the case $r = 1$ being vacuously true). 
	
	Pick $r \geq 2$ and assume that $g_2,g_3,\dots,g_r \in G_{r}$. We will work modulo $G_{r+1}$, which means that (the projections of) all of $g_1,g_2,\dots,g_r$ commute: $g_ig_j G_{r+1} = g_j g_i G_{r+1}$. It follows directly from how the sequence $\lambda$ is computed by $\cT$ that for any $m \geq 0$ and any  $I \subset \NN_0$ with $\abs{I} = m$ we have
	\begin{equation}\label{eq:373:1}
		\lambda\bra{\textstyle\sum_{l\in I} k^l} = \lambda([\texttt{10}^{j_1}\texttt{10}^{j_2}\dots\texttt{10}^{j_l}]_k) = \lambda(1)^m = g_1^m,
	\end{equation}
	for some $j_1,\dots,j_r\geq 0$.	Let $J = \{l_1,\dots,l_r\}$ be any set of cardinality $\abs{J} = r$.
	Substituting \eqref{eq:373:0} in \eqref{eq:373:1} and taking the oscillating product over all subsets $I \subset J$ we conclude that 
	 \begin{equation}\label{eq:373:2}
		g_r^{k^{l_1} k^{l_2} \cdot \dots \cdot k^{l_r}} \equiv \prod_{I \subset J} \lambda\bra{\sum_{l \in I}k^l}^{(-1)^{\abs{I}}} \equiv \id_G \pmod{G_{r+1}}, 
	\end{equation}
	meaning that the order of $g_r $ in $G/G_{r+1}$ divides a power of $k$: $g_r^{k^{L_r}} \in G_{r+1}$ for some $L_r \geq 0$. (Equation \eqref{eq:373:2} can be verified by a direct computation, relying on the fact that the finite difference operator reduces the degree of any polynomial by $1$.)
	Reasoning inductively, we show that for each $j = r-1,r-2,\dots,2$ there exists $L_{j} \geq 0$ such that $g_{j}^{k^{L_j}} \in G_{r+1}$: towards this end, it is enough to repeat the same computation as above with $\abs{J} = j$ and $\min J \geq \max(L_{j+1},\dots,L_{r})$. In particular, there exists $L_* \geq 0$ such that for all $n \geq 0$ divisible by $L_*$ we have
	\begin{equation}\label{eq:373:3c}
		\lambda\bra{n} = g_1^n g_2^{\binom{n}{2}} \dots g_s^{\binom{n}{s}} \equiv g_1^n \bmod{G_{r+1}}.
	\end{equation}
	
	Next, recall that from how $\lambda$ is computed by $\cT$ it also follows that $\lambda$ is invariant under dilation by $k$ in the sense that for any $n \geq 0$ and any $l \geq 0$ we have
	\begin{equation}\label{eq:373:3}
		\lambda\bra{n k^l} = \lambda(n).
	\end{equation}
	Taking $l \geq L_*$ and combining \eqref{eq:373:0}, \eqref{eq:373:3c} and \eqref{eq:373:3}, for any $n \geq 0$ we obtain
	\begin{equation}\label{eq:373:0a}
		g_1^{k^ln} \equiv \lambda(k^l n) = \lambda(n)= g_1^n g_2^{\binom{n}{2}} \dots g_s^{\binom{n}{s}} \bmod{G_{r+1}}.
	\end{equation}
	Since the representation of the sequence $\lambda$ in the form \eqref{eq:373:0} is unique, it follows that $g_r \equiv g_{r-1} \equiv \dots \equiv g_2 \equiv \id_{G} \bmod{G_{r+1}}$, which finishes this part of the argument. 
	
	We have now shown that $g_2=g_3=\dots = g_s = \id_G$. It remains to notice that since $g_1^k = \lambda(k) = \lambda(1) = g_1$ and $\lambda \colon \NN_0 \to G$ is surjective, the group $G$ is cyclic and $\abs{G} \mid k-1$.
\end{proof}

As suggested by the above lemma, \geas{} which arise from cyclic groups will play an important role in our considerations. Let $k \geq 2$ denote the basis, which we view as fixed. For $m \geq 1$ define the invertible \geaab{}
\begin{equation}\label{eq:def-of-Z(m)}
	\cZ(m) := \bra{\Sigma_k,\ZZ/m\ZZ, \lambda_m }, \qquad \lambda_m \colon \Sigma_k \ni j \mapsto j \bmod{m} \in \ZZ/m\ZZ.
\end{equation}
We will primarily be interested in the case when $m \mid k-1$.

\begin{lemma}\label{lem:Zm-basic}
Fix $k \geq 2$ and let $m,m' \geq 1$ and let $\cT$ be an \egeka{}.
\begin{enumerate}[wide] 
\item If $m \mid k-1$ then the \geaab{} $\cZ(m)$ is efficient, $\lambda_m(\word u) = [\word u]_k \bmod m$ for all $\word u \in \Sigma_k^*$, and $\cQ^d(\cZ(m)) = \HK^d(\ZZ/m\ZZ)$.
\item If $m,m' \mid k-1$ then $\cZ(m)$ is a factor of $\cZ(m')$ if and only if $m \mid m'$. The factor is not characteristic unless $m = m'$. 
\item If $m \mid k-1$ then $\cZ(m)$ is a factor of $\cT$ if and only if $m \mid d'_{\cT}$.
\item If $m \mid k-1$ and $\cZ(m)$ is a characteristic factor of $\cT$ then $m = d'_{\cT}$.
\end{enumerate}
\end{lemma}
\begin{proof}
\begin{enumerate}[wide]
\item Each of the defining properties of an \egeaab{} can be verified directly (we take $d_0' = 1$ and $G_0 = G$).
\item This easily follows from the fact that $\ZZ/m\ZZ$ is a subgroup of $\ZZ/m'\ZZ$ if and only if $m \mid m'$.
\item Suppose first that $\cZ(m)$ is a factor of $\cT$ and the factor map is given by $(\phi,\pi)$. Then for any $\word w \in \Sigma_k^*$ with $\delta(s_0,\word w) = s_0$ and $\lambda(s_0,\word w) = \id_G$ we have 
\[ 0 = \pi(\id_G) = \lambda_{m}(\word w) = [\word w]_k \bmod m. \]
Hence, by property \ref{item:74C}, $m \mid d'$. In the opposite direction, property \ref{item:74C} guarantees that $\cZ(d')$ is a factor of $\cT$, with the group homomorphism given by $g_0^r h \mapsto r \bmod d'$ for all $h \in G_0$, $0 \leq r < d'$. It remains to notice that if $m \mid d'$ then $\cZ(m)$ is a factor of $\cZ(d')$. 
\item We already know that $m \mid d'_{\cT}$ so it remains to show that $m \geq d'_{\cT}$. Consider the probability $p$ that a random cube $\bb g \in G^{[2]}$ belongs to $\cQ^d(\cT)$. On one hand, since $\cZ(d'_{\cT})$ is a factor of $\cT$, we have $p \leq 1/d_{\cT}'$ (three coordinates of $\bb g$ determine the projection of the fourth to $\ZZ/d'_{\cT}\ZZ$). On the other hand, since $\cZ(m)$ is characteristic, we have $p = 1/m$. It follows that $m \geq d'_{\cT}$. 
\qedhere
\end{enumerate}
\end{proof}

We are now ready to reformulate our description of the cube groups $\cQ^d(\cT)$ in Theorem \eqref{thm:cubes} in a more succinct way using the language of characteristic factors. Equivalence of the said theorem and the following result is easily seen once one unwinds the definitions.

\begin{theorem}\label{thm:char_fact_is_Z} 
	Let $\cT$ be an \egeaab{}. Then $\cZ(d'_{\cT})$ is a characteristic factor of $\cT$.
\end{theorem}

\subsection{Strong synchronisation}\label{ssec:Cubes:str_sync}

Recall that \egeaab{} are built on automata that are synchronising. A stronger synchronisation property is enjoyed, for example, by the \geaab{} producing the Rudin--Shapiro sequence discussed in Example \ref{ex:Rudin-Shapiro}: \emph{all} sufficiently long words are synchronising for the underlying automaton (in fact, all nonempty words have this property). In this section we show that, passing to a characteristic factor, we can ensure this stronger synchronisation property for the underlying automata in general.

Let $\cT$ be a \geaab{}. For the purposes of this section, we will say that a pair of states $s,s' \in S$ is \emph{mistakable} if for every length $l$ there exists a word $\word u \in \Sigma^*_k$ with $\abs{\word u} \geq l$ and two states $r,r' \in S$ such that $\delta(r, \word u) = s$ and $\delta(r', \word u) = s'$. Note that in this situation $\word u$ cannot be a synchronising word for the underlying automaton unless $s = s'$. We will also say that the pair $s,s' \in S$ is \emph{strongly mistakable} if there exists a nonempty word $\word w \in \Sigma_k^*\setminus \{\epsilon\}$ such that $\delta(s, \word w) = s$ and $\delta(s', \word w) = s'$, while $\lambda(s,\word w) = \lambda(s',\word w) = \id_G$. As the terminology suggests, if $s,s'$ are strongly mistakable then they are also mistakable (we may take $\word u = \word w^l$ and $r = s$, $r' = s'$). The following lemma elucidates the connection between mistakable states and synchronisation.



\begin{lemma}\label{lem:mist_FCAE}
	Let $\cT$ be a natural tranducer and let $\cA$ be the underlying automaton. Then the following properties are equivalent: 
	\begin{enumerate}
	\item\label{item:98A} There exists a pair of distinct mistakable states $s, s' \in S$.
	\item\label{item:98B} There exists a pair of distinct strongly mistakable states $s, s' \in S$.
	\item\label{item:98C} There exist infinitely many words in $\Sigma_k^*$ which are not synchronising for $\cA$.
	\end{enumerate}
\end{lemma}
\begin{proof}
	As any pair of strongly mistakable states is mistakable, \eqref{item:98B} implies \eqref{item:98A}. Moreover, as we have remarked above, \eqref{item:98A} implies \eqref{item:98C}. 
	
	In the reverse direction, \eqref{item:98C} implies \eqref{item:98A}: indeed, if \eqref{item:98C} holds, then there exist infinitely many words $\word u_i \in \Sigma_k^*$ ($i \in \NN$) with corresponding quadruples $r_i,r'_i,s_i,s'_i \in S$ such that $s_i \neq s_i'$ and $\delta(r_i,\word u_i) = s_i,\ \delta(r_i',\word u_i) = s_i'$. Any pair $s,s' \in S$ such that $s = s_i$ and $s' = s_i'$ for infinitely many values of $i$ is mistakable, so \eqref{item:98A} holds.

	It remains to show that \eqref{item:98A} implies \eqref{item:98B}. By definition, it follows from \eqref{item:98A} that there exists a word $\word u = u_1 u_2 \dots u_l \in \Sigma^{*}_k$ with $\abs{\word u} = l \geq \abs{S}^2$ and states $r,r',s,s' \in S$ with $s \neq s'$ such that 
  $\delta(r,\word u) = s$ and $\delta(r', \word u) = s'$. For $0 \leq i \leq l$, let $s_i$ and $s_i'$ be the states reached form $r$ and $r'$ respectively after reading the first $i$ digits of $\word u$. More precisely, $s_i,s_i'$ are given by $s_0 = r$, $s_0' = r'$ and $s_{i} = \delta(s_{i-1},u_i)$, $s_{i}' = \delta(s_{i-1}',u_i)$ for all $1 \leq i \leq l$. Note that since $s_l \neq s_l'$ we have more generally $s_i \neq s_i'$ for all $0 \leq i \leq l$. By the pigdeonhole principle, there exists a pair of indices $0 \leq i < j \leq l$ and a pair of states $t,t'$ such that $s_i = s_j = t$ and $s_i' = s_j' = t'$. Put $\word v = u_{i+1} u_{i+2} \dots u_{j}$ so that $\delta(t,\word v) = t$ and $\delta(t',\word v) = t'$. Finally, put $\word w = \word v^{\abs G}$ so that $\delta(t,\word w) = t$ and $\delta(t',\word w) = t'$ and by the Lagrange's theorem we have $\lambda(t, \word w) = \lambda(t, \word v)^{\abs{G}} = \id_G$ and likewise $\lambda(t', \word w) = \id_G$. It follows that $t,t'$ are strongly mistakable.
\end{proof}

\begin{proposition}\label{prop:reduce_to_str_sync}
  Let $\cT$ be an \egeaab{}.
  Then $\cT$ has a characteristic factor $\bar \cT$ such that every sufficiently long word is synchronizing for the underlying automaton.
\end{proposition}

The proof of Proposition \ref{prop:reduce_to_str_sync} proceeds by iterating the following lemma.

\begin{lemma}\label{lem:reduce_to_str_sync}
  Let $\cT$ be an \egeaab{} and let $H < G$ be given by
  	\begin{equation}
  	\label{eq:590-def-H}
  	H = \abra{ \lambda(s, \word u)^{-1} \lambda(s', \word u) : s \text{ and } s' \text{ are strongly mistakable}, \word u \in \Sigma_k^*}^{G}.
  	\end{equation}
 	Then $\bar \cT/H$ is a characteristic factor of $\cT$.
\end{lemma}
\begin{proof}

  	  	

Recall from Section \ref{ssec:Cubes:grp_quot} that it will suffice to verify that $H < K = K(\cT)$. Let $h$ be one of the generators of $H$ in \eqref{eq:590-def-H}. Pick a pair of strongly mistakable states $s,s' \in S$ and a word $\word u \in \Sigma_k^*$ such that $h = \lambda(s, \word u)^{-1} \lambda(s', \word u)$. Replacing $\word u$ with $\word u \word w_0^{\cT}$, where $\bfw_0^{\cT}$ is a synchronizing word of $\cT$, we may assume without loss of generality that $\word u$ synchronises  the underlying automaton to $s_0$, so in particular $\delta(s,\word u) = \delta(s',\word u) = s_0$.
	
	In order to construct the relevant morphism $(l,\vec e) \colon \bb v_0^{\cT} \to \bb v_0^{\cT}$, we first need to specify several auxiliary words with certain helpful properties, described by the diagram below. Let $\word w$ be a word such that $\delta(s_0, \word w) = s$ and $\lambda(s_0,\word w) = \id_G$, whose existence is guaranteed by property \ref{item:74B}. Let $\word v_1$ be a word such that $\delta(s, \word v_1) = s$, $\delta(s', \word v_1) = s'$, and $\lambda(s,\word v_1) = \lambda(s',\word v_1) = \id_G$, which exists because $s,s'$ are strongly mistakable. Lastly, let $\word v_0$ be a word such that $\delta(s, \word v_0) = \delta(s', \word v_0) = s'$ and $\lambda(s', \word v_0) = \lambda(s', \word v_0) = \id_G$. One can obtain such a word by concatenating $\word w_0^{\cT}$ with a word taking $s_0$ to $s'$ with identity group label, whose existence is guaranteed by property \ref{item:74B}.
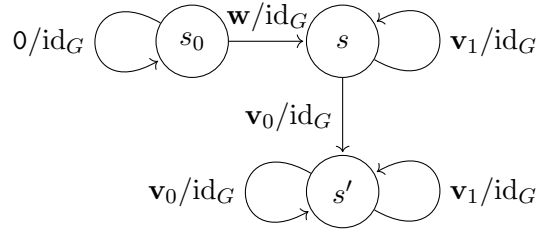
	\begin{center}
\begin{tikzpicture}[scale=.5,shorten >=1pt,node distance=2cm, on grid, auto] 
   \node[state] (s_00)   {$s_{0}$}; 
   \node[state] (s_10) [right=of s_00] {$s$}; 
   \node[state] (s_11) [below=of s_10] {$s'$}; 

	\path[->]
    (s_00) edge [left] node [above]  {$\word w/\id_G$} (s_10);

 \tikzstyle{loop}=[min distance=6mm,in=30,out=-30,looseness=7]
   \path[->] 
    (s_10) edge [loop right] node {$\word v_1/\id_G$} (s_10);

 \tikzstyle{loop}=[min distance=6mm,in=30,out=-30,looseness=7]
   \path[->] 
    (s_11) edge [loop right] node {$\word v_1/\id_G$} (s_11);

 \tikzstyle{loop}=[min distance=6mm,in=210,out=-210,looseness=7]
   \path[->] 
    (s_11) edge [loop left] node {$\word v_0/\id_G$} (s_11);
   \path[->] 
    (s_00) edge [loop left] node {$\texttt 0/\id_G$} (s_00);

 \tikzstyle{loop}=[min distance=4mm,in=120,out=240,looseness=1]
    \path[->]     
    (s_10) edge [left] node {$\word v_0/\id_G$} (s_11);

\end{tikzpicture}
\end{center}	
\newcommand{\m}{m}
\newcommand{\vDelta}{\word v}
We may additionally assume that the words ${\word v_1}$ and ${\word v_0}$ have the same length $\m$; otherwise we can replace them with $\word v_0^{\abs{ \word v_1}}$ and $\word v_1^{\abs{\word v_2}}$ respectively. Note that $\word v_0 \neq \word v_1$ since $s \neq s'$. Assume for concreteness that $[\word v_0]_k < [\word v_1]_k$; the argument in the case $[\word v_0]_k > [\word v_1]_k$ is analogous. Let $\vDelta = \bra{[\word v_1]_k - [\word v_0]_k}_k^{\m}$ be the result of subtracting $\word v_0$ from $\word v_1$. Put also $l = \abs{\word w} + d \m + \abs{\word u}$. We are now ready to define the coordinates $e_i$, which are given by
\begin{align*}
	e_0 &= [ \word w \underbrace{ \word v_0 \word v_0 \dots \word v_0}_{\text{$d$ times}} \word u ]_k; 
	& e_j &= [ \vDelta \underbrace{\texttt 0^{\m} \texttt 0^{\m} \dots \texttt 0^{\m}}_{\text{$d-j$ times}} \texttt 0^{\abs{\word u}} ]_k \quad (0 < j \leq d).
\end{align*}
This definition is set up so that for each $\vec \omega \in \{0,1\}^d$ we have
\begin{align*}
	1 \vec\omega \cdot \vec e &= [  \word w  \word v_{\omega_1}  \word v_{\omega_2} \dots  \word v_{\omega_d}  \word u ]_k.
\end{align*}
Since $\word u$ synchronises the underlying automaton of $\cT$ to $s_0$ and $1 \vec\omega \cdot \vec e < k^l$ for each $\vec \omega \in \{0,1\}^d$, it follows directly from \eqref{eq:deq-s-and-r'} that we have a morphism $\tilde e = (l,\vec e) \colon \bb v_0^{\cT} \to \bb v_0^{\cT}$, and so $\bblambda(\tilde e) \in \cQ^d(\cT)$. Our next step is to compute $\bblambda(\tilde e)$.

It follows directly from the properties of $\word w, \word v_0$ and $\word v_1$ listed above that
$$
	\delta(s_0, \word w \word v_{ \omega_1} \word v_{ \omega_2} \dots \word v_{ \omega_{j} }) = 
	\begin{cases}
		s, & \text{ if } \omega_1 = \omega_2 = \dots = \omega_j  = 1, \\
		s', &	\text{ otherwise.}
	\end{cases}
$$
for any $\vec\omega \in \{0,1\}^d$ and $0 \leq j \leq d$ (the case $j=0$ corresponds to $\delta(s_0,  \word w) = s$). Hence, for any $\vec \omega \in \{0,1\}^d$ different from $\vec 1$ we have
\begin{align*}
	\lambda(s_0, (1\omega \cdot \vec e)_k^l ) &=
	\lambda(s_0, \word w) \lambda(s, \word v_1)^{j-1} \lambda(s, \word  v_0) \lambda(s', \word v_{\omega_{j+1}}) \dots \lambda(s', \word v_{\omega_{d}}) \lambda(s', \word u)
	\\&= \lambda(s',\word u),
\end{align*}
where $j$ is the first index with $\omega_j = 0$.
For $\vec \omega = \vec 1$ we obtain a similar formula, which simplifies to
\begin{align*}
	\lambda(s_0, (\vec 1 \cdot \vec e)_k^{l} ) &= \lambda(s, \word u).
\end{align*}
Since $d \geq 0$ was arbitrary, it follows from Corollary \ref{cor:H-in-Q-reduction} that $\lambda(s, \word u) \equiv \lambda(s', \word u) \bmod{K}$, and consequently $H < K$, as needed. 
\end{proof}

\begin{proof}[Proof of Proposition \ref{prop:reduce_to_str_sync}]

Let $\cT' := (\cT/H)_{\mathrm{red}}$, where $H = H(\cT)$ is given by \eqref{eq:590-def-H}. Recall that $\cT'$ is efficient by Lemma \ref{lem:trans_fact_is_natural}. Note that either 
\begin{enumerate}[wide]
\item\label{it:str_sync:A} $\cT'$ is a proper factor of $\cT$; or
\item\label{it:str_sync:B} all sufficiently long words synchronise the underlying automaton of $\cT$.
\end{enumerate}
Indeed, if \eqref{it:str_sync:B} does not hold then it follows from Lemma \ref{lem:mist_FCAE} that there exists a pair of distinct strongly mistakable states $s,s' \in S$. The definition of $H$ guarantees that the images of those states in $\cT/H$ give rise to the same label maps: $\bar\lambda(s,\word u) = \bar\lambda(s',\word u)$ for all $\word u \in \Sigma_k^*$. It follows that $s$ and $s'$ are mapped to the same state in $(\cT/H)_{\mathrm{red}}$. In particular, $(\cT/H)_{\mathrm{red}}$ has strictly fewer states than $\cT$.

Iterating the construction described above, we obtain a sequence of characteristic factors 
\[
 \cT' \to \cT'' \to \dots \to \cT^{(n)} \to  \cT^{(n+1)} \to \dots, \]
where $\cT^{(n+1)} = \bra{\cT^{(n)}}' = \bra{\cT^{(n)}/H(\cT^{(n)})}_{\mathrm{red}}$ for each $n \geq 0$. 
  Since all objects under consideration are finite, this sequence needs to stabilise at some point, meaning that there exists $n \geq 0$ such that $\cT^{(n)} = \cT^{(n+1)} = \dots := \bar\cT$. Since $\bar \cT' = \bar \cT$, it follows from the discussion above that all sufficiently long words are synchronising for the underlying automaton of $\bar \cT $. By Lemma \ref{lem:reduce_to_str_sync}, $\bar \cT$ is a characteristic factor of $\cT$.
\end{proof}

\begin{example}\label{ex:str_sycn}
Consider the \geaab{} described by the following diagram, where $g,h \in G$ are two distinct group elements.
	
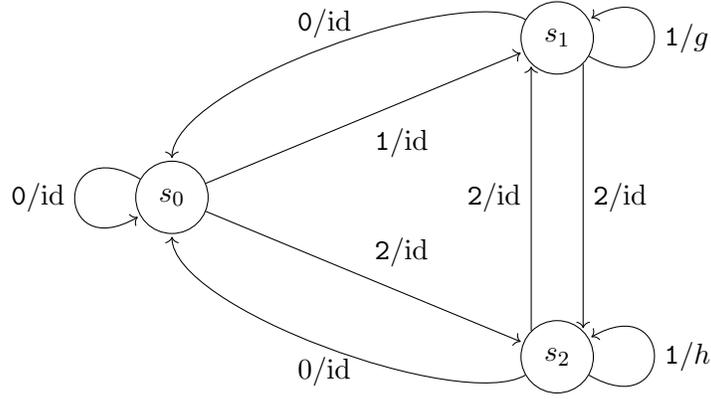
\begin{center}
\begin{tikzpicture}[scale=.5,shorten >=1pt,node distance=3cm, on grid, auto] 
	\node (s_00) {};
   \node[state] (s_0) [left= of s_00]  {$s_{0}$}; 
   \node[state] (s_1) [above right= of s_00] {$s_1$}; 
   \node[state] (s_2) [below right= of s_00] {$s_2$}; 

	\path[->]
    (s_0) edge [left] node [below right]  {$\texttt 1/\id$} (s_1);
   	\path[->]
    (s_0) edge [left] node [above right]  {$\texttt 2/\id$} (s_2);

 \tikzstyle{loop}=[min distance=6mm,in=30,out=-30,looseness=7]
   \path[->] 
    (s_1) edge [loop right] node {$\texttt 1/g$} (s_1);

 \tikzstyle{loop}=[min distance=6mm,in=30,out=-30,looseness=7]
   \path[->] 
    (s_2) edge [loop right] node {$\texttt 1/h$} (s_2);

 \tikzstyle{loop}=[min distance=6mm,in=210,out=-210,looseness=7]
   \path[->] 
    (s_0) edge [loop left] node {$\texttt 0/\id$} (s_0);

 \tikzstyle{loop}=[min distance=4mm,in=120,out=240,looseness=1]
    \path[->]     
    (s_1.south east) edge [right] node {$\texttt 2/\id$} (s_2.north east);
    \path[->]     
    (s_2.north west) edge [left] node {$\texttt 2/\id$} (s_1.south west);
    
 \tikzstyle{loop}=[in=-90,out=210,looseness=.5]
    \path[->]     
    (s_2) edge [loop below]  node {$0/\id$} (s_0);

 \tikzstyle{loop}=[in=90,out=-210,looseness=.5]
    \path[->]     
    (s_1) edge [loop above]  node {$\texttt 0/\id$} (s_0);
\end{tikzpicture}
\end{center}	
The word $0$ is synchronising for the \geaab{} and no word in $\{1,2\}^*$ is synchronising for the underlying automaton. The states $s_1$ and $s_2$ are strongly mistakable and the loops are given by $1^m$ where $m$ is any common multiple of the orders of $g$ and $h$. The group $H$ in Lemma \ref{lem:reduce_to_str_sync} is generated by $gh^{-1}$ and its conjugates, and the \geaab{} $\cT' = \bar \cT$ in the proof of Proposition \ref{prop:reduce_to_str_sync} is obtained by collapsing $s_1$ and $s_2$  into a single state. 
\end{example}

\subsection{Invertible factors}

In this section we further reduce the number of states of the \geaab{} under consideration. In fact, we show that it is enough to consider \geaab{} with just a single state. Recall that such \geaab{}s with one states are called invertible.

\begin{proposition}\label{prop:reduce_to_invert}
  Let $\cT$ be an \egeaab{} such that all sufficiently long words are synchronising for the underlying automaton. Then $\cT$ has an invertible characteristic factor.
\end{proposition}

It will be convenient to say for any $N,L \geq 0$ that a \geaab{} $\cT$ is \emph{$(N,L)$-nondiscriminating} if $\lambda(s,\word u) = \lambda(s',\word u)$ for all $s,s' \in S$ and all $\word u \in \Sigma_k^L$ such that $[\word u]_k < N$. In particular, any \geaab{} $\cT$ is vacuously $(0,L)$-nondiscriminating for all $L \geq 0$, and if $\cT$ is additionally efficient then it is $(1,L)$-nondiscriminating for all $L \geq 0$ (recall that efficiency implies that $\lambda(s,\mathtt 0) = \id_G$ for all $s \in S$). Our proximate goal on the path to prove Proposition \ref{prop:reduce_to_invert} is to find a characteristic factor that is $(N,L)$-nondiscriminating for all $N,L \geq 0$. Indeed, note that any invertible \geaab{} is $(N,L)$-nondiscriminating for all $N,L \geq 0$. Conversely, as we will shortly see, a \geaab{} that is $(N,L)$-nondiscriminating for all $N,L \geq 0$ can be reduced to an invertible \geaab{} by removing redundant states.

\begin{lemma}\label{lem:nondiscriminating}
	Let $\cT$ be an \egeka{}. Suppose that there exist $L \geq 1$ and $N \geq k^L$ such that $\cT$ is $(N,L)$-nondiscriminating. Then $\cT$ is  $(N,L)$-nondiscriminating for all $N,L \geq 0$. 
\end{lemma}
\begin{proof}
It is clear that the property of being $(N,L)$-nondiscriminating becomes stronger as $N$ increases. The values of $N$ above $k^L$ will be mostly irrelevant: if $\cT$ is $(k^L,L)$-nondiscriminating then it is immediate that it is $(N,k^L)$-nondiscriminating for all $N \geq 0$. By assumption, $\cT$ is $(k^L,L)$-nondiscriminating for at least one $L \geq 1$. Let $\mathcal{L}$ denote the set of all $L \geq 0$ with the aforementioned property (in particular, $0 \in \mathcal{L}$).

If $L_1, L_2 \in \mathcal{L}$ then also $L_1+L_2 \in \mathcal{L}$. Indeed, any $\word u \in \Sigma_k^{L_1+L_2}$ can be written as $\word u = \word u_1 \word u_2$ with $\word u_1 \in \Sigma_k^{L_1}$ and $\word u_2 \in \Sigma_k^{L_2}$, whence for any $s,s' \in S$ we have $\lambda(s,\word u) = \lambda(s_0,\word u_1) \lambda(s_0,\word u_2) = \lambda(s',\word u)$. 
Moreover, if $L \in \mathcal{L}$ and $L \neq 0$ then $L-1 \in \mathcal{L}$. Indeed, if $\word u \in \Sigma^{L-1}_k$ then for any $s,s' \in S$ we have $\lambda(s,\word u) = \lambda(s_0,\word u \texttt{0}) = \lambda(s',\word u)$. 

It remains to note that the only set $\mathcal{L} \subset \NN_0$ with all of the properties listed above is $\NN_0$.
\end{proof}

\begin{lemma}\label{lem:inductive}
	Let $\cT$ be an \egeka{}, let $\cA$ be the underlying automaton and
	  $0 < N < k^L$. Suppose that every word in $\Sigma_k^L$ is synchronising for $\cA$ and that $\cT$ is $(N,L)$-nondiscriminating.  Then $\cT$ has a characteristic factor $\cT'$ which is $(N+1,L)$-nondiscriminating.
\end{lemma}
\begin{proof}
	Following a strategy similar to the one employed in the proof of Proposition \ref{prop:reduce_to_str_sync}, let $\word u = (N)_k^L$ and consider the normal subgroup of $G$ given by
	\begin{equation}
	 H:= \abra{ \lambda(s, \word u )^{-1} \lambda(s', \word u ): s,s' \in S}^G.
	 \label{eq:def-H-2}
	\end{equation}	
	We aim to use Proposition \ref{prop:grp_quot_is_char} to show that $\cT/H$ is a characteristic factor of $\cT$. Fix for now the dimension $d \geq 0$ and an integer $M$ such that $k^M > d$.
	Pick $s \in S$ and a word $\word v$ such that $\delta(s_0,\word v) = s$ and $\lambda(s_0,\word v) = \id_G$, whose existence is guaranteed by property \ref{item:74B}. We recall that $\word w_0^{\cT}$ denotes a word that synchronizes $\cT$ to $s_0$. 
	 Consider $\vec e \in \NN_0^{d+1}$ given by
\begin{align*}
	e_0 = [\word v \word u \texttt{0}^{M}\word w_0^{\cT}]_k - d [\texttt{10}^{\abs{\word w_0^{\cT}}}]_k;
	\qquad e_j = [ \texttt{10}^{\abs{\word w_0^{\cT}}} ]_k \quad (0 < j \leq d). 
\end{align*}
Put also $l := \abs{\word v} + L + M + \abs{\word w_0^{\cT}}$ and let $\word u' := (N-1)_k^L$. 
These definitions are arranged so that for each $\vec\omega \in \{0,1\}^d$ the word 
$(1\vec\omega \cdot \vec e)_k^l$ takes the form
\[
 (1\vec\omega \cdot \vec e)_k^l =
 \begin{cases}
 \word v \word u' \word x_{\vec\omega} \word w_0^{\cT} &\text{if } \vec \omega \neq \vec 1;\\ 
 \word v \word u 0^{M} \word w_0^{\cT} &\text{if } \vec \omega = \vec 1,
 \end{cases}
 \]
 where $\word x_{\vec\omega} = (k^M - d + \abs{\vec\omega})_k^M \in \Sigma_k^M$. Since for each $\vec\omega \in \{0,1\}^d$ the word $(1\vec\omega \cdot \vec e)_k^l$ ends with $\word w_0^{\cT}$ and $(1\vec\omega \cdot \vec e)_k < k^L$, the data constructed above describes a morphism $\tilde e = (l,\vec e) \colon \bb v_0^{\cT} \to \bb v_0^{\cT}$.

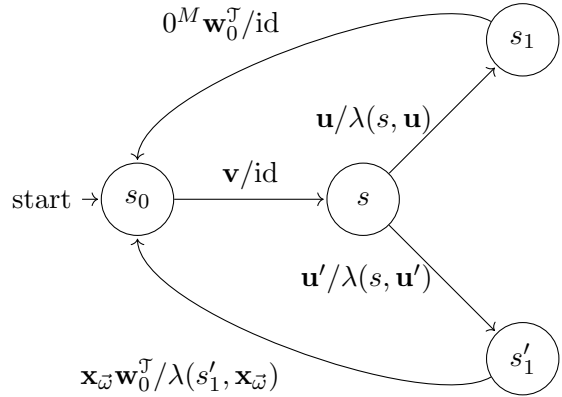
	\begin{center}
\begin{tikzpicture}[scale=.5,shorten >=1pt,node distance=3cm, on grid, auto] 
 	\node[state] (s_00) [] {$s$}; 
     \node[state] (s_0) [initial,left= of s_00]  {$s_{0}$}; 
   \node[state] (s_1) [above right= of s_00] {$s_1$}; 
   \node[state] (s_2) [below right= of s_00] {$s_1'$}; 

	\path[->]
    (s_0) edge [left] node [above]  {$\word v/\id$} (s_00);
	
	\path[->]
    (s_00) edge [left] node [left]  {$\word u/\lambda(s,\word u)$} (s_1);
   	\path[->]
    (s_00) edge [left] node [left]  {$\word u'/\lambda(s,\word u')$} (s_2);

 \tikzstyle{loop}=[in=-90,out=210,looseness=.5]
    \path[->]     
    (s_2) edge [loop]  node [below left] {$\word x_{\vec\omega} \word w_0^{\cT}/\lambda(s'_1,\word x_{\vec\omega})$} (s_0);

 \tikzstyle{loop}=[in=90,out=-210,looseness=.5]
    \path[->]     
    (s_1) edge [loop]  node [above left] {$0^M\word w_0^{\cT}/\id$} (s_0);
\end{tikzpicture}
\end{center}	

Our next step is to compute $\bblambda(\tilde e)$. In fact, we only need some basic facts rather than a complete description. For $\vec\omega \neq 1^d$ we have
\begin{align*}
	\lambda \bra{ s_0, (1\vec\omega \cdot \vec e)_k^l )} 
	&= \lambda(s_0,\word v) \lambda(s, \word u') \lambda( \delta(s,\word u'), \word x_{\vec\omega} ) \lambda( \delta(s,\word u' \word x_{\vec\omega}), \word w_0^{\cT} )
	\\&=  \lambda(s_0, \word u') \lambda( s_1', \word x_{\vec\omega} ),
\end{align*}
where the state $s_1' = \delta(s,\word u')$ is independent of $s$ because $\word u'$ is synchronising for $\cA$, and $\lambda(s, \word u') =  \lambda(s_0, \word u')$ because $\cT$ is $(N,L)$-nondiscriminating.
Similarly, 
\begin{align*}
	\lambda \bra{ s_0, (\vec 1 \cdot \vec e)_k^l )} &= \lambda(s, \word u \texttt{0}^{M}) = \lambda(s, \word u).
\end{align*}

Note that out of all the coordinates of $\bblambda(\tilde e)$, only one depends on $s$. Let $s' \in S$ be any other state, and let $\tilde e'\colon \bb v_0^{\cT} \to \bb v_0^{\cT} $ be the result of applying the same construction as above with $s'$ in place of $s$. Then
\[
	\bblambda(\tilde e) \bblambda(\tilde e')^{-1} = \bb c_{\vec 1}^d\bra{ \lambda(s,\word u) \lambda(s',\word u)^{-1} } \in \cQ^d(\cT).
\]
Since $d \geq 0$ was arbitrary, it follows from Lemma \ref{lem:H-in-Q-reduction} that $\lambda(s,\word u) \equiv \lambda(s',\word u) \bmod K$. Since $s,s' \in S$ were arbitrary, $H < K$ and hence $\cT/H$ is a characteristic factor. 

Let $\bar \cT = \cT/H$. Then $\bar \cT$ is $(N,L)$-nondiscriminating because $\cT$ is. Moreover, it follows directly from the definition of $H$ that $\bar \lambda(s,\word u) = \bar \lambda(s',\word u)$ for all $s,s' \in S$, whence $\bar\cT$ is $(N+1,L)$-nondiscriminating.
\end{proof}

\begin{proof}[Proof of Proposition \ref{prop:reduce_to_invert}]
	Let $L \geq 0$ be large enough that all words of length $\geq L$ are synchronising for $\cA$
	. Applying Lemma \ref{lem:inductive} we can construct a sequence of characteristic factors 
	\[
		\cT = \cT_0 \to \cT_1 \to \dots \to \cT_{k^L}
	\]
	such that for each $0 \leq N \leq k^L$ the \geaab{} $\cT_N$ is $(N,L)$-nondiscriminating. In particular, $\cT$ has a characteristic factor $\bar\cT = \cT_{k^L}$ which is $(k^L,L)$-nondiscriminating. 
Hence, $\bar\cT'$ is $(N,M)$-nondiscriminating for all $N,M \geq 0$ by Lemma \ref{lem:nondiscriminating}.
Next, it follows directly from the construction that $\bar\cT_{\mathrm{red}}$ is invertible. It remains to recall that $\bar\cT_{\mathrm{red}}$ is a characteristic factor of $\cT$ by Lemma \ref{prop:state_red_is_char}.
	\end{proof}

\begin{example}
  Consider the \geaab{} described by the following diagram. Then each of the first three applications of Lemma \ref{lem:inductive} removes one of the group labels $g_i$.
  	\begin{center}
\begin{tikzpicture}[scale=.5,shorten >=1pt,node distance=5cm, on grid, auto] 

   \node[state] (s_0) []  {$s_{0}$}; 
   \node[state] (s_1)[ right= of s_0] {$s_1$}; 

 \tikzstyle{loop}=[min distance=6mm,in=210,out=-210,looseness=7]
   \path[->] 
    (s_0) edge [loop left] node {$\texttt 0/\id$} (s_0);

 \tikzstyle{loop}=[in=30,out=-30,looseness=7]
   \path[->] 
    (s_1) edge [loop right] node {$\texttt 1/\id$} (s_1);
 \tikzstyle{loop}=[in=30,out=-30,looseness=16]
   \path[->] 
    (s_1) edge [loop right] node {$\texttt 2/\id$} (s_1);
 \tikzstyle{loop}=[in=30,out=-30,looseness=28]
   \path[->] 
    (s_1) edge [loop right] node {$\texttt 3/\id$} (s_1);
    
 \tikzstyle{loop}=[in=-30,out=-150,looseness=1]
   \path[->] 
    (s_1) edge [loop above] node {$\texttt 0/\id$} (s_0);
    
 \tikzstyle{loop}=[in=150,out=30,looseness=1]
   \path[->] 
    (s_0) edge [loop below] node {$\texttt 1/g_1$} (s_1);
 \tikzstyle{loop}=[in=130,out=50,looseness=1.5]
   \path[->] 
    (s_0) edge [loop below] node {$\texttt 2/g_2$} (s_1);
 \tikzstyle{loop}=[in=110,out=70,looseness=2]
   \path[->] 
    (s_0) edge [loop below] node {$\texttt 3/g_3$} (s_1);

\end{tikzpicture}
\end{center}	
\end{example}

\subsection{Invertible \geas{}}

In this section we deal exclusively with invertible \geas{}. As pointed out in Section \ref{ssec:Trans:defs}, an invertible \geaab{} can be identified with a triple $(\Sigma_k,G,\lambda)$ where $\lambda \colon \Sigma_k \to G$ is a labelling map. By a slight abuse of notation we identify $\lambda$ with a map $\NN_0 \to G$, denoted with the same symbol, $\lambda(n) = \lambda((n)_k)$.
Recall that the cyclic \geas{} $\cZ(m)$ were defined in Section \ref{ssec:Cubes:Host-Kra}.

\begin{proposition}\label{prop:char_fact_of_invertible}
  Let $\cT$ be an invertible \egeka{}. Then $\cT$ has a characteristic factor of the form $\cZ(m)$ for some $m$ which divides $k-1$.
\end{proposition}
\begin{proof}
	Following the usual strategy (cf.\ Propositions \ref{prop:reduce_to_str_sync} and \ref{prop:reduce_to_invert}), we will consider the normal subgroup of $G$ given by
	\begin{equation}\label{eq:def-of-H-3}
		H = \abra{ \lambda(n+1) \lambda(1)^{-1} \lambda(n)^{-1} : n \geq 0}^G.
	\end{equation}
	A simple inductive argument shows that $\lambda(n) \equiv \lambda(1)^n \bmod{H}$ for all $n \geq 0$, and in fact $H$ is the normal subgroup of $G$ generated by $\lambda(n)\lambda(1)^{-n}$ for $n \geq 0$. In particular, $G/H$ is cyclic.
	
	We will show that the factor $\cT/H$ is characteristic. Fix $d \geq 0$, take any $n \geq 0$. Let $t = \abs{G}$ so that $g^t = \id_G$ for all $g \in G$. Consider the vector $\vec e \in \NN_0^{d+1}$ given by
  \begin{align*}
    {e}_0 = n k^{td} + 1; \qquad  {e}_j =  (k^{t}-1)k^{(d-j)t} \quad (1 \leq j \leq d).
  \end{align*}
  Put also $l = \abs{ (n)_k } + td + 1$ so that $1\vec{\omega} \cdot \vec e < k^l$ for all $\vec\omega \in \{0,1\}^d$ and hence we have a morphism $\tilde e = (l,\vec e) \colon \bb v_0^{\cT} \to \bb v_0^{\cT}$. We next compute $\bblambda(\tilde e)$. If $\vec\omega \in \{0,1\}^d \setminus \{\vec 1\}$ and $0 \leq j \leq d$ be the largest index such that $\omega_j = 0$, then
	$$
	(1 \vec\omega \cdot \vec e)_k^l = \mathtt{0}(n)_k \word v_{\omega_1} \word v_{\omega_2} \dots \word v_{\omega_{j-1}} \mathtt{0}^{t-1}\mathtt{1 0}^{t(d-j)},
	$$
	where $\word v_1 = (k^t-1)_k \in \Sigma_k^t$ and $\word v_0 = \mathtt{0}^t \in \Sigma_k^t$. Since $\lambda(\word v_0) =\lambda(\word v_1) = \id_G$, we have
	$$
	\lambda{(1 \vec\omega \cdot \vec e)_k^l} = \lambda(n) \lambda(1).
	$$   
  By a similar reasoning,
  \[
  	\lambda\bra{(\vec 1 \cdot \vec e)_k^l} = \lambda\bra{n+1}.
  \]
	Since $d \geq 0$ was arbitrary, it follows by Corollary~\ref{cor:H-in-Q-reduction} that $\lambda(n+1) \equiv \lambda(n)\lambda(1) \bmod K$. Since $n$ was arbitrary, $H < K$ and $\cT/H = (\Sigma_k,G/H, \bar \lambda)$ is characteristic. Let $m$ denote the order the cyclic group $G/H$. Because $\bar\lambda(n) = \bar\lambda(1)^n$ for all $n \geq 0$, $\cT/H$ is isomorphic to $\cZ(m)$, and because
	\(
		\lambda(1) = \lambda(k) \equiv \lambda(1)^k \bmod{H},
	\)
$m$ is a divisor of $k-1$. 
\end{proof}


\subsection{The end of the chase}

In this section we finish the proof of the main result of this section. This task is virtually finished --- we just need to combine the ingredients obtained previously.

\begin{proof}[Proof of Theorem \ref{thm:cubes}]
Chaining together Propositions \ref{prop:reduce_to_str_sync}, \ref{prop:reduce_to_invert} and \ref{prop:char_fact_of_invertible} we conclude that the \egeaab{} $\cT$ has a characteristic factor of the form $\cZ(m)$ with $m \mid k-1$. By Lemma \ref{lem:Zm-basic} it follows that $m = d'_{\cT}$.
\end{proof}


\section*{Acknowledgments} 
The authors thank the anonymous reviewers for their careful reading of the paper and the feedback.

\bibliographystyle{amsplain}


\begin{dajauthors}

\begin{authorinfo}[jb]
  Jakub Byszewski\\
  Faculty of Mathematics and Computer Science\\
  Jagiellonian University\\
  \L{}ojasiewicza 6\\
  30-348 Krak\'{o}w, Poland\\
  jakub\imagedot{}byszewski\imageat{}uj\imagedot{}edu\imagedot{}pl\\
\end{authorinfo}
\begin{authorinfo}[jk]
  Jakub Konieczny\\
  Camille Jordan Institute\\ 
  Claude Bernard University Lyon 1\\
  43 Boulevard du 11 novembre 1918\\
  69622 Villeurbanne Cedex, France\\
  \vspace{1em}
  Faculty of Mathematics and Computer Science\\
  Jagiellonian University\\
  \L{}ojasiewicza 6\\
  30-348 Krak\'{o}w, Poland\\
  jakub\imagedot{}konieczny\imageat{}gmail\imagedot{}com\\
\end{authorinfo}
\begin{authorinfo}[cm]
  Clemens M\"{u}llner\\
  Institut f\"ur Diskrete Mathematik und Geometrie\\
  TU Wien\\
  Wiedner Hauptstr.\ 8--10\\
  1040 Wien, Austria\\
  clemens\imagedot{}muellner\imageat{}tuwien\imagedot{}ac\imagedot{}at\\
\end{authorinfo}
\end{dajauthors}

\end{document}